\documentclass[
a4paper,
11pt,
DIV=14,
toc=bibliography, 
headsepline, 
parskip=half-,
abstract=true,
]{scrartcl}

\usepackage[centertags]{amsmath}
\usepackage[utf8]{inputenc}
\usepackage{enumerate,amsfonts,amssymb,amsthm,amsopn,cite,array,comment}

\usepackage[usenames]{color}
\definecolor{citegreen}{rgb}{0,0.6,0}
\definecolor{refred}{rgb}{0.8,0,0}
\usepackage[colorlinks, citecolor=citegreen, linkcolor=refred, urlcolor=black]{hyperref}

\title{Geometric Structure of Ends of Ricci Shrinkers}
\author{Alessandro Bertellotti and Reto Buzano}
\date{}

\newtheorem{theorem}{Theorem}[section]
\newtheorem{lemma}[theorem]{Lemma}
\newtheorem{cor}[theorem]{Corollary}
\newtheorem{prop}[theorem]{Proposition}

\theoremstyle{definition}
\newtheorem{rem}[theorem]{Remark}

\newtheorem{defn}[theorem]{Definition}
\numberwithin{equation}{section}
\newtheorem{claim}[theorem]{Claim}

\newcommand{\Z}{\mathbb{Z}}

\newcommand{\N}{\mathbb{N}}
\newcommand{\R}{\mathbb{R}}

\newcommand{\X}{\mathcal{X}}
\newcommand{\s}{\mathcal{S}}
\newcommand{\reg}{\mathcal{R}}

\newcommand{\I}{\mathcal{I}}
\newcommand{\x}{\mathbf{x}}
\newcommand{\y}{\mathbf{y}}
\newcommand{\z}{\mathbf{z}}
\newcommand{\Y}{\mathcal{Y}}
\newcommand{\C}{\mathfrak{C}}
\newcommand{\F}{\mathbb{F}}
\newcommand{\M}{\mathcal{M}}

\DeclareMathOperator{\Ric}{Ric}
\DeclareMathOperator{\Rm}{Rm}

\DeclareMathOperator{\Vol}{Vol}

\DeclareMathOperator{\var}{Var}
\DeclareMathOperator{\supp}{supp}
\DeclareMathOperator{\id}{Id}

\newcommand\printaddress{{
\setlength{\parindent}{15pt}
\footnotesize~
\par
{\scshape Alessandro Bertellotti}
\newline 
SISSA, 
Via Bonomea 265, 
34136 Trieste, Italy
\newline
\emph{E-mail address:} 
\texttt{abertell@sissa.it}
\par
{\scshape Reto Buzano}
\newline 
Universit\`a di Torino, 
Dipartimento di Matematica,
Via Carlo Alberto 10, 
10123 Torino, Italy 
\newline
\emph{E-mail address:} 
\texttt{reto.buzano@unito.it}
\par
}}

\begin{document}
\maketitle
\begin{abstract}
We study blow-up sequences of Ricci shrinkers without global curvature assumptions based at points $q$ at which the scalar curvature satisfies a Type~I bound, proving that their $\F$-limits split a line. In the four-dimensional case these limits are smooth Ricci shrinkers and the convergence is in the pointed smooth Cheeger-Gromov sense. As a consequence, limits along the integral curve of $\nabla f$ starting at such a point $q$ split a line. This generalises known results about the geometry of ends of Ricci shrinkers that relied on global curvature bounds. To obtain our results, we extend the $\F$-convergence theory from \cite{Ba20_structure, heatkernel, heatkernel2}.
\end{abstract}

\section{Introduction}\label{introduction}

The main goal of this article is to initiate the study of the geometric structure of the ends of a gradient shrinking Ricci soliton $(M,g,f)$ \emph{without} global curvature assumptions. In order to do so, we study pointed limits of such solitons along sequences of points going to infinity along an integral curve of $\nabla f$, or equivalently blow-up limits at the singular time away from the base point.

Recall that a smooth, connected, complete, $n$-dimensional Riemannian manifold $(M^n,g)$ is a gradient shrinking Ricci soliton, or a \emph{Ricci shrinker} for short,
if there exists a function 
$f\in C^{\infty}(M)$, called a potential, such that 
\begin{equation}\label{shrinker-eq}
\Ric_g + \nabla^{2}f =\frac{1}{2}g.
\end{equation} 
Ricci shrinkers are generalisations of positive Einstein manifolds, self-similar solutions of the Ricci flow
\begin{equation}
\partial_{t}g(t)=-2\Ric_{g(t)}, 
\end{equation}
and models for its finite time singularities. Indeed, the standard approach to study singularities of the Ricci flow consists in considering blow-up sequences. When a Ricci flow $(M,g(t))_{t\in [0,T)}$ satisfies a (global) Type~I curvature condition, i.e.~when there exists a constant $C>0$ such that 
\begin{equation}\label{eq.TypeI}
\sup_{M}|\Rm_{g(t)}|_{g(t)} \leq C/(T-t)
\end{equation}
for all $t\in [0,T)$, then by Naber \cite{Na10}, Enders-M\"uller-Topping \cite{EMT11} and Mantegazza-M\"uller \cite{MM15}, the pointed blow-up sequence $(M,g_{k}(t),q)_{k\in \N}$, where $g_{k}(t)$ are given by
\begin{equation}\label{metric-blow-up-intro}
g_{k}(t)=\lambda_{k}g(T+t/\lambda_{k}), \qquad t\in [-\lambda_k T,0)
\end{equation}
for some $\lambda_k \nearrow \infty$, converges to a smooth Ricci shrinker in the pointed smooth Cheeger-Gromov sense on compact time-intervals. Moreover, the nature of the point $q$ determines if the limit Ricci shrinker is trivial or not. These results are based on the compactness theory for sequences of Ricci flows
developed by Hamilton in \cite{HamComp}, which heavily relies on curvature bounds, which is why \eqref{eq.TypeI} is needed. 

In order to remove or weaken the Type~I condition, one needs an adequate compactness theory
that works in a more general context. Towards such a theory \cite{hm, hm4, LLW} studied the compactness properties of sequences of pointed Ricci shrinkers (under an entropy lower bound) while \cite{bamler,Hallgren} extended the results from \cite{EMT11, MM15} to the case of Ricci flows with a Type~I condition on the scalar curvature only. From these results one sees that it is necessary to take into account weaker notions of convergence to limit spaces that are no longer smooth manifolds, but metric spaces with some kind of regular-singular decomposition.

Recently, in a series of papers \cite{Ba17, Ba20_hk, Ba20_compactness, Ba20_structure}, Bamler developed a groundbreaking new compactness theory which extends the previously known results and relies on the notion of $\F$-convergence. This convergence is expressed in terms of \emph{metric flow pairs}, that consist of a metric flow (some sort of metric space analogue of a Ricci flow), together with a conjugate heat flow (a family of probability measures satisfying a reproduction formula); see \cite{Ba20_compactness}. His theory allows to extract limits from \emph{general} sequences of Ricci flows, and in the compact case, under a volume non-collapsing condition, gives precise geometric information about the limit space; see \cite{Ba20_structure}. Bamler later explained how his results can be generalised to complete, non-compact flows with bounded curvature on compact time-intervals; see the appendix of \cite{Ba21}.

Li and Wang \cite{heatkernel, heatkernel2} generalised some of the results from \cite{Ba20_compactness, Ba20_structure} to sequences of pointed Ricci shrinkers with a lower entropy bound and \emph{without any curvature assumption}. We briefly illustrate the general idea.  
First of all, it is well known that any Ricci shrinker $(M,g,f)$ has an associated Ricci flow $(M,g(t))_{t<1}$, with $T=1$ as its singular time; in Section \ref{shrinker-preliminaries-section} we will review the precise construction. Choosing a basepoint from $M\times (-\infty,1)$, this flow can be turned into a metric flow pair in a canonical way, thanks to the existence of a conjugate heat kernel function $H$ on $M$; see \cite[Section~3]{heatkernel} and \cite[Section~6]{heatkernel2} for details. 
At this point, one can interpret any sequence of pointed Ricci flows associated with Ricci shrinkers, or more generally induced by them as in Definition \ref{ricci-flow-induced}, as a sequence of metric flow pairs, allowing one to use the theory from \cite{Ba20_compactness}. Finally, the arguments from \cite{Ba20_structure} can be adapted to this case, yielding the main results in \cite{heatkernel2}.

In the present article, we want to consider blow-up sequences for Ricci shrinkers (without global curvature bounds) where the basepoint lies in the \emph{singular} time slice of the flow. In other words, we want to generalise the results from Li and Wang mentioned above to the case where the basepoint comes from $M\times \{1\}$. A first step consists in the construction of a suitable heat kernel function based at the singular time (Theorem \ref{thm-sing-heat-kernel-intro}), similarly to what was done in \cite{MM15} in the context of compact Type~I Ricci flows. This will then yield a general convergence result (Theorem 
\ref{scaling-compactness-thm-intro}). We stress that all the technical concepts and notions not defined so far can be found in later sections of the present paper or, for the ones regarding the $\F$-convergence theory, in Bamler's article \cite{Ba20_compactness}, which is the standard reference on the subject.   

\begin{theorem}[Heat Kernel Based at the Singular Time]\label{thm-sing-heat-kernel-intro} 
Let $(M^n,g,f)$ be a Ricci shrinker, with associated Ricci flow $(M^n,g(t))_{t<1}$, and let $H\colon \mathcal{D}\to \R$ be its heat kernel function. Given a point $q\in M$ and a sequence of times 
$\mathsf{t}_{i}\nearrow 1$, consider the functions $v_{i}(y,s)=H(q,\mathsf{t}_{i};y,s)$ defined on $M\times (-\infty,\mathsf{t}_{i})$. 
Then the sequence $(v_{i})_{i\in \N}$ subconverges in $C^{\infty}_{\text{loc}}(M\times (-\infty,1))$ to a smooth, positive function 
$v_{q}\in C^{\infty}(M\times (-\infty,1))$ satisfying the following conditions.
\begin{enumerate}
\item $v_{q}$ solves the conjugate heat equation $\partial_{t}v_{q}=-\Delta v_{q} +S v_{q}$ on $M\times (-\infty,1)$, where $\Delta$ and $S$ denote, respectively, the Laplacian and the scalar curvature of $g(t)$.
\item $\int_{M}v_{q}(y,s)dV_{s}(y)=1$ for every $s<1$.
\item For every $y\in M$, $s<1$ and $\varrho \in (s,1)$, the following semigroup property holds
\begin{equation*}
v_{q}(y,s)=\int_{M}H(z,\varrho; y,s)v_{q}(z,\varrho)dV_{\varrho}(z).
\end{equation*}
\end{enumerate}
\end{theorem}
\vspace{-1mm}

\begin{theorem}[$\F$-Convergence]\label{scaling-compactness-thm-intro} 
Let $\left(M_{k}^{n}, g_{k}(t), (q_{k},0)\right)_{t<T_{k}}$, for $k\in \N$, be a sequence of $n$-dimensional pointed Ricci flows induced by Ricci shrinkers and defined over a maximal time-interval of the form $(-\infty,T_{k})$, with finite singular time $0\leq T_{k}<\infty$.
Denote by $(\X^{k},(\mu^{k}_{t})_{t \leq 0})$ the corresponding (extended) metric flow pairs over $(-\infty,0]$.  After passing to a subsequence, there exists, up to isometry, a unique 
$H_{n}$-concentrated and future continuous metric flow pair $(\X^{\infty},(\mu^{\infty}_{t})_{t\leq 0})$, representing a class in 
$\F_{(-\infty,0]}^{\ast}(H_{n})$, such that the following holds.
There is a correspondence $\C$ over $(-\infty,0]$ between the metric flows $\X^{k}$, $k\in \N\cup\{\infty\}$, such that
\begin{equation*}
(\X^{k},(\mu^{k}_{t})_{t\leq 0})\longrightarrow (\X^{\infty},(\mu^{\infty}_{t})_{t\leq 0})
\end{equation*}
within $\C$ on compact time-intervals.
Moreover, the convergence is uniform over any compact $J\subseteq (-\infty,0]$ that only contains times at which $\X^{\infty}$ is continuous.
\end{theorem}
\vspace{-1mm}

Observe that Theorem \ref{scaling-compactness-thm-intro} is a generalisation of Theorem~6.10 in \cite{heatkernel2}; the key difference being that we allow the case $T_{k}=0$, corresponding to taking basepoints in the singular time slice.

Obviously, one would like to say something more about the structure of the limit $\X^{\infty}$, and promote the $\F$-convergence from Theorem \ref{scaling-compactness-thm-intro} to a better one. This is the content of the next result, which is obtained constructing an approximating sequence of metric flow pairs converging to the same limit but to which Theorem~6.12 from \cite{heatkernel2} can be applied. 

\begin{theorem}[Structure of the $\F$-Limit]\label{good-convergence-intro} 
Under the same assumptions as in Theorem \ref{scaling-compactness-thm-intro}, supposing in addition that the Ricci flows $(M^{n}_{k},g_{k}(t))_{t<T_{k}}$ are induced by Ricci shrinkers with entropy uniformly bounded below by some $\underline{\mu}>-\infty$, the following hold.
\begin{enumerate}
\item We can decompose
\begin{equation*}
\X^{\infty}_{t<0}=\reg^{\infty}\sqcup \s^{\infty},
\end{equation*}
where $\reg^{\infty}$, the set of regular points of $\X^{\infty}$, has the structure of an $n$-dimensional Ricci flow spacetime $(\reg^{\infty}, \mathfrak{t}^{\infty},\partial^{\infty}_{\mathfrak{t}},g^{\infty})$ over $(-\infty,0)$. Moreover, $\dim_{\M^{\ast}}(\s^{\infty})\leq n-2$, where 
$\dim_{\M^{\ast}}$ denotes the $\ast$-Minkowski dimension, and $\mu^{\infty}_{t}(\s^{\infty}_{t})=0$ for every $t<0$.
\item Every tangent flow of the limit is a metric soliton.
\item Each time slice $\reg^{\infty}_{t}=\reg^{\infty}\cap \X^{\infty}_{t}$ is open and dense in $\X^{\infty}_{t}$, and $d_{t}^{\infty}\vert_{\reg^{\infty}_{t}}$ coincides with the Riemannian length metric induced by $g^{\infty}_{t}=g^{\infty}\vert_{\reg_{t}}$.
\end{enumerate}
\end{theorem}
\vspace{-1mm}

Furthermore, the convergence of the approximating sequence towards $(\X^{\infty},(\mu^{\infty}_{t})_{t\leq 0})$ (obtained in Proposition \ref{prop-convergence}) is smooth on $\reg^{\infty}$ in the sense of Proposition \ref{good-convergence-prop}. Finally, there is also a precise description of the structure of tangent flows at points of the limit $(\X^{\infty},(\mu^{\infty}_{t})_{t\leq 0})$, given by Proposition \ref{metric-soliton-structure}. While these results are crucial for what follows, we do not state these propositions here to keep the introduction shorter and less technical.

The above results can be specialised to the case of blow-up sequences of arbitrary Ricci shrinkers, obtaining the following result.

\begin{theorem}[Blow-ups of Ricci Shrinkers]\label{thm-blow-up-intro} Let $(M^{n},g, f)$ be any Ricci shrinker (possibly without curvature bounds). Let $(M,g_{k}(t),q)_{k\in \N}$, with $g_k(t)$ as in \eqref{metric-blow-up-intro}, be a blow-up sequence based at some $q\in M$. Let $(\X^{k},(\mu^{k}_{t})_{t\leq 0})$, for $k\in \N$, be the sequence of associated (extended) metric flow pairs over $(-\infty,0]$. Then the conclusions of Theorems \ref{scaling-compactness-thm-intro}
and \ref{good-convergence-intro} hold. Moreover, the limit $(\X^{\infty},(\mu^{\infty}_{t})_{t\leq 0})$ is a continuous metric soliton satisfying all the structural properties from Proposition \ref{metric-soliton-structure}.
\end{theorem}

In order to obtain more information about the geometry of the limit space, we are forced to impose a curvature condition on the Ricci shrinker \emph{at the point $q$}.

\begin{defn}\label{def.TypeIscalar}
Let $(M,g(t))_{t<1}$ be the Ricci flow associated with a Ricci shrinker $(M,g,f)$. We say that \emph{the scalar curvature at a point $q\in M$ is of Type~I} if there exists a constant $L>0$ such that 
\begin{equation*}
S(q,t)\leq L/(1-t)
\end{equation*}
for every $t\in [0,1)$.
\end{defn}
 
In order to understand our next result, recall that any Ricci shrinker $(M,g,f)$ has a natural basepoint $p\in M$, i.e.~a point where $f$ attains its minimum, and that $(\psi_{t})_{t<1}$ denotes the family of diffeomorphisms of $M$ generated by the vector field $(1-t)^{-1}\nabla f$; see Section \ref{shrinker-preliminaries-section}.

By self-similarity, a blow-up sequence of a Ricci flow associated with a Ricci shrinker pointed at the basepoint $q=p$ of the shrinker will simply converge to the Ricci shrinker itself. Similarly, if $d_{g}(p,\psi_{t}(q))$ remains bounded as $t \nearrow 1$, then $d_{g(t)}(p,q)\to 0$ by \eqref{change-of-distance}, and one can show that the blow-up limit based at such $q$ is again simply the Ricci shrinker we started with. Note that in these cases the Type~I assumption on the scalar curvature at $q$ is trivially satisfied (since the scalar curvature is uniformly bounded in terms of distance from $p$, see \cite{hm}). The only interesting case is therefore the one where $d_{g}(p,\psi_{t}(q))\to \infty$ (and in this case, the Type~I scalar curvature bound from Definition \ref{def.TypeIscalar} is a non-trivial condition about the growth at infinity of the scalar curvature of the shrinker along the integral curve of $\nabla f$ emanating from $q$).

We remark that all Ricci flows associated to any of the currently known Ricci shrinkers satisfy a Type I bound on the full curvature tensor at all of their points, so assuming only a scalar curvature bound at only one point is a relatively mild assumption.
 
\begin{theorem}[Blow-up Limit Splits a Line]\label{thm-splitting-intro} 
Under the same assumptions as in Theorem \ref{thm-blow-up-intro} suppose, in addition, that $d_{g}(p,\psi_{t}(q))\to \infty$ as $t\nearrow 1$ and that the scalar curvature at $q\in M$ is of Type~I.  Then the limit metric flow $\X^{\infty}_{<0}$ splits a line, i.e.~there exists a metric flow $\widetilde{\X}$ over $(-\infty,0)$ such that $\X^{\infty}_{<0}\simeq \widetilde{\X}\times \R$. In particular, the regular and singular parts also split a line, $\reg^{\infty}\simeq \widetilde{\reg}\times \R$ and $\s^{\infty}\simeq \widetilde{\s}\times \R$.
\end{theorem}

In the four-dimensional case we obtain the best possible convergence and structure statements. 

\begin{theorem}[Four-Dimensional Blow-up Limits]\label{thm-splitting-four-intro} 
Let $(M,g,f)$ be a four-dimensional Ricci shrinker with basepoint $p\in M$, and let $q\in M$ be a point with $d_{g}(p,\psi_{t}(q))\to \infty$ as $t\nearrow~1$ satisfying a Type~I scalar curvature bound. Given any blow-up sequence $(M,g_{k}(t),q)_{k\in \N}$, let 
$(\X^{\infty},(\mu^{\infty}_{t})_{t\leq 0})$ be the limit metric flow pair obtained in Theorem \ref{thm-blow-up-intro}. Then the following hold.
\begin{enumerate}
\item $\X^{\infty}_{<0}$ is smooth, meaning that its singular set $\s^{\infty}$ is empty. 
In particular, the singular space $(X,d)$ from Proposition \ref{metric-soliton-structure} is a smooth four-dimensional Ricci shrinker
that splits a line as $(N\times \R,g_{N}+dt^{2})$, where $N$ is a smooth three-dimensional Ricci shrinker.  
\item If $\lambda_{k}\nearrow \infty$ are the scales of the given blow-up sequence, or more precisely of the converging subsequence, 
letting $x_{k}=\psi_{1-1/\lambda_{k}}(q)\in M$, the sequence $(M,g,x_{k})_{k\in \N}$ converges, 
in the pointed smooth Cheeger-Gromov sense, to $(N\times \R, g_{N}+dt^{2},x_{\infty})$. 
\item The limit Ricci shrinker $X$ is isometric to flat $\R^{4}$ if and only if $S(x_{k})\to 0$ as $k\to \infty$.
\end{enumerate}
\end{theorem}

It is well-known that three-dimensional Ricci shrinkers are fully classified. In particular, a smooth three-dimensional Ricci shrinker $N$ is either flat $\R^3$, the round cylinder $\mathbb{S}^2 \times \R$ or its $\Z_{2}$ quotient, or a quotient of the round three-sphere $\mathbb{S}^3/\Gamma$.

Now we discuss how the preceding setting can be applied to study the geometry of ends of Ricci shrinkers. Understanding the properties of the ends of a Ricci shrinker, such as their possible geometries and their number, has recently received a lot of attention; see for example \cite{BB24, msw, mw15, mw19, mw22}.
In particular, in \cite[Theorem~1.5]{mw19}, Munteanu and Wang proved the following. Let $(M,g,f)$ be a four-dimensional Ricci shrinker with (globally) bounded scalar curvature, and let $E\subseteq M$ be an end with scalar curvature bounded from below away from zero. Then for any sequence of points $(x_{k})_{k\in \N}$ in $E$ going to infinity along an integral curve of $\nabla f$, the sequence $(M,g,x_{k})_{k\in \N}$ converges, in the smooth Cheeger-Gromov sense, to one of the following possible limits: $\R^{2}\times \mathbb{S}^{2}$ or its $\Z_{2}$ quotient, or a round cylinder $\R\times \mathbb{S}^{3}/\Gamma$. Moreover, in the latter case, $E$ is smoothly asymptotic to 
$\R\times \mathbb{S}^{3}/\Gamma$.  This result relies on \cite{Na10, EMT11, MM15}, since it reduces the convergence of 
$(M,g,x_{k})_{k\in \N}$ to the convergence of an adequate blow-up sequence $(M,g_{k}(t),q)_{k\in \N}$. Note that this uses the fact that the full curvature tensor of a four-dimensional soliton is controlled by its scalar curvature, and thus the global scalar curvature bound assumption implies that the shrinker in question is globally of Type~I. It is natural to try to extend their theorem, removing or weakening the global scalar curvature condition, using $\mathbb{F}$-convergence. In this direction, we obtain the following simple consequence of Theorem \ref{thm-splitting-four-intro}.

\begin{theorem}[Ends Split a Line]\label{ends-four-dim}
Let $(M,g,f)$ be a four-dimensional Ricci shrinker and let $q\in M$ be a point with scalar curvature of Type~I and with integral curve of $\nabla f$ starting at $q$ going to infinity along an end $E$. Given any sequence of points $(x_{k})_{k\in \N}$ going to infinity along this integral curve, i.e.~with $d_{g}(p,x_{k})\to \infty$ as $k\to \infty$, then, after passing to a subsequence, we have
 \begin{equation*}
 (M,g,x_{k})_{k\in \N}\longrightarrow (X,g_{X},x_{\infty})
 \end{equation*}
in the pointed smooth Cheeger-Gromov sense, where $X$ is a smooth four-dimensional Ricci shrinker splitting isometrically as $(N\times \R,g_{N} + dt^{2})$, with $N$ a smooth three-dimensional Ricci shrinker. In particular, $X$ is isometric to $\R^{4}$, $\R^{2}\times \mathbb{S}^{2}$ or its $\Z_{2}$ quotient, or $\R\times \mathbb{S}^{3}/\Gamma$, each with their canonical metric. Finally, $X$ is isometric to flat $\R^{4}$ if and only if $S(x_{k})\to 0$ as $k\to \infty$.
\end{theorem}

We remark that the scalar curvature Type~I assumption at $q$ is satisfied if $S$ is bounded along the integral curve of $\nabla f$ starting at $q$. In particular, we do not need a global scalar curvature bound, but only a bound along \emph{one} integral curve going to infinity to obtain the same conclusion as in \cite[Theorem~1.5]{mw19}. Nevertheless, it would of course be desirable to remove such an assumption completely.

The article is organised as follows. Section \ref{shrinker-preliminaries-section} starts with a review of the basic theory on Ricci shrinkers, including heat kernel functions. We then prove Theorem \ref{thm-sing-heat-kernel-intro} together with some other properties of the heat kernel based at the singular time. In Section \ref{shrinker-mf-section}, we first study the probability measure associated with the singular time heat kernel and then use this to associate a metric flow pair to a Ricci shrinker, proving also some of its basic properties. In Section \ref{conv-result-section} we prove Theorem \ref{scaling-compactness-thm-intro}, while Section \ref{reg-results-section} contains the proofs of Theorem \ref{good-convergence-intro} and of the other structural and convergence results mentioned above. In Section \ref{blow-up-section}, we apply the previous results to the case of blow-up sequences, thus proving Theorem \ref{thm-blow-up-intro} and Theorem \ref{thm-splitting-intro}; moreover, we study the four-dimensional case, thus obtaining Theorem \ref{thm-splitting-four-intro} and Theorem \ref{ends-four-dim}. 
Finally, Appendix \ref{schauder-appendix} contains a brief summary of Schauder estimates for the conjugate heat equation needed in Section \ref{shrinker-preliminaries-section}, while in Appendix \ref{mms-appendix} we collect some complementary facts regarding metric measure spaces and metric flows, some of which might be of independent interest.  

\textbf{Acknowledgements.} AB has been fully supported by a Phd studentship from SISSA.
RB has been partially supported by the Italian PRIN project ``Differential-geometric aspects of manifolds via global analysis'' (No.~20225J97H5).

%-----------------------------------------------------------------------------------------------------------
\section{Singular Time Heat Kernel on Ricci Shrinkers}\label{shrinker-preliminaries-section}

Let us start by collecting some preliminary facts about Ricci shrinkers. Given an $n$-dimensional Ricci shrinker $(M,g,f)$, it is well known that the quantity 
\begin{equation*}
c(g,f)=S+|\nabla f|^{2}-f,
\end{equation*}
where $S$ is the scalar curvature of $g$, is constant on $M$; see for example \cite{hm}. Moreover, $S$ is always nonnegative and only vanishes at some point if $M$ is isometric to $\R^{n}$ \cite{rps}. The potential $f$ always attains its minimum at some point $p\in M$, which we regard as the basepoint of the shrinker,
and its growth is controlled by the distance from $p$; more precisely, one has
\begin{equation}\label{fgrowth}
\frac{1}{4}(d(x,p)-5n)_{+}^{2}\leq f(x)+c(g,f) \leq \frac{1}{4}(d(x,p)+\sqrt{2n})^{2},
\end{equation}
where  $a_{+}=\max\{a,0\}$.

Since metric balls centered at the basepoint of the shrinker satisfy an Euclidean volume growth condition, estimate \eqref{fgrowth} implies that the weighted volume of $M$ is finite, i.e.~$\int_{M}e^{-f}d V_{g}<\infty$, and therefore we can normalise $f$ to obtain
\begin{equation}\label{potential-normalisation}
\int_{M}e^{-f}dV_{g}=(4\pi)^{\frac{n}{2}}.
\end{equation}
Note that, under the normalisation \eqref{potential-normalisation}, the constant $c(g,f)$ becomes independent of $f$, and it will be denoted by $c(g)$, although $f$ itself is not uniquely determined. Moreover, $c(g)$ coincides with minus the Perelman entropy of $(M,g)$:
\begin{equation*}
-c(g)=\mu_{g}=\int_{M}(|\nabla f|^{2}+S+f-n)(4\pi)^{-\frac{n}{2}}e^{-f}dV_{g} > - \infty.
\end{equation*} 
For the proofs of the above statements we refer again to \cite{hm}.

As mentioned in the introduction, it is possible to associate a Ricci flow to any smooth Ricci shrinker. Indeed, since $(M,g)$ is complete, the time dependent vector field $\frac{1}{\tau}\nabla f$, where $\tau=1-t$ and $t\in (-\infty,1)$, is complete, meaning that it generates a family of diffeomorphisms $(\psi_{t})_{t\in (-\infty,1)}$ satisfying
\begin{equation*}
\left\{
\begin{aligned}
\partial_{t}\psi_{t}(x) &=\frac{1}{\tau}\nabla f(\psi_{t}(x)),\\
\psi_{0}(x) &=x,
\end{aligned}
\right.
\end{equation*}
for every $x\in M$ and $t\in (-\infty,1)$; see \cite{zhang}. Defining the family of metrics $g(t)=\tau \psi_{t}^{\ast}g$, then it is standard to prove that $(M,g(t))_{t\in (-\infty,1)}$ is an ancient Ricci flow with $g(0)=g$.

\begin{defn} 
The ancient Ricci flow $(M,g(t))_{t<1}$ constructed in this way is called the \emph{Ricci flow associated with the Ricci shrinker} $(M,g,f)$. 
\end{defn}
 
Denoting by $f(t)$ the time dependent potential $f\circ \psi_{t}$, a version of equation \eqref{shrinker-eq} holds:
\begin{equation*}
\Ric_{g(t)}  + \nabla^{2}_{g(t)}f(t) =\frac{1}{2\tau}g(t).
\end{equation*}
From the definition of $f(t)$ it is easy to check that 
\begin{equation*}
\tau S+\tau |\nabla f(t)|_{g(t)}^{2}-f(t) \equiv c(g)
\end{equation*}
is constant in both the time and space variables; moreover, the normalisation
\begin{equation*}
\int_{M}e^{-f(t)}dV_{t}=(4\pi\tau)^{\frac{n}{2}}
\end{equation*}
holds for every $t\in (-\infty,1)$. Here and in the following $dV_{t}$ stands for the Riemannian volume element induced by $g(t)$. As one expects, also the Perelman entropy of $(M,g(t))$,
\begin{equation*}
\mu(g(t),f(t))=\int_{M}(\tau |\nabla_{g(t)} f(t)|_{g(t)}^{2}+\tau S+f(t)-n)(4\pi \tau)^{-\frac{n}{2}}e^{-f(t)}dV_{t},
\end{equation*}
coincides with $\mu_{g}$ and so is independent of time.

Since $f(t)$ satisfies the evolution equation
\begin{equation*}
\partial_{t} f(t)=|\nabla_{g(t)} f(t)|_{g(t)}^{2},
\end{equation*}
if we define $F(x,t)=\tau f(x,t)-\tau \mu_{g}$ we obtain the following relations 
\begin{align}
\label{potential-time-deriv} \partial_{t}F=\tau |\nabla_{g}(t) f|_{g(t)}^{2}-f +\mu_{g} &= -\tau S,\\
\label{potential-curvature-eq} \tau^{2}S + |\nabla_{g(t)} F|_{g(t)}^{2}&=F.
\end{align}

We also note that since the metrics $g(t)$ are self-similar, their curvature tensors and distances can be expressed in term of $g=g(0)$ only. For example, the $(1,3)$-curvature tensor satisfies $\Rm(x,t)=\psi_{t}^{\ast}\Rm(x)$, and we therefore obtain
\begin{align}
\label{riemann-tensor-time} |\Rm(x,t)| &=(1-t)^{-1}|\Rm(\psi_{t}(x))|,\\
\label{scalar-curvature-time} S(x,t) &=(1-t)^{-1}S(\psi_{t}(x)),
\end{align}
while for the distances we have
\begin{align}
d_{t}(x,y)&=\sqrt{1-t}\ d(\psi_{t}(x),\psi_{t}(y)),\label{change-of-distance}\\
B_{t}(x,r) &=\psi_{t}^{-1}\left(B(\psi_{t}(x),r/\sqrt{1-t})\right). \label{change-of-balls}
\end{align}
In \eqref{riemann-tensor-time}--\eqref{change-of-balls}, the quantities on the left hand side refer to the evolving metric $g(t)$, while the ones on the right hand side refer to the fixed metric $g$. 

From \eqref{riemann-tensor-time} it follows that the Ricci flow $(M,g(t))_{t<1}$ is of Type~I if and only if $\sup_{M}|\Rm|<\infty$. Moreover, with the exception of flat $\R^n$, the basepoint $p$ is always a singular point in the sense of \cite{EMT11}, since it holds $|\Rm(p,t)|=(1-t)^{-1}|\Rm(p)|$, and $T=1$ is always a singular time, i.e.~the flow can never be smoothly extended past this time. 
     
For the convenience of the reader, we recall the exact definitions of parabolic rescaling; we specialise to the case of Ricci shrinkers, but the definition is the same for a general Ricci flow. 

Let $(M,g(t))_{t<1}$ be the Ricci flow associated with a Ricci shrinker $(M,g,f)$.
If $\lambda>0$ is a positive number and $t_{0}\leq 1$, the Ricci flow $(M,g_{-t_{0},\lambda}(t))_{t <T}$, where
\begin{equation}\label{eq.parresc}
g_{-t_{0},\lambda}(t)=\lambda g(t_{0}+t/\lambda)
\end{equation}
and $T=\lambda(1-t_{0})$, is called a parabolic rescaling of $(M,g(t))_{t<1}$. 

This operation produces a new Ricci flow $(M,g(t))_{t< T}$, where now $T<\infty$ is the singular time. Note that the above includes also the case $\lambda=1$, which corresponds to a time-shift; if in addition $t_{0}=0$ we are not changing the flow at all. 

\begin{defn}\label{ricci-flow-induced} Every Ricci flow obtained by \eqref{eq.parresc} is called a \emph{Ricci flow induced by the Ricci shrinker} $(M,g,f)$. 
\end{defn}    
  
Given a point $q\in M$ and a sequence of positive numbers $\lambda_{k}\nearrow \infty$, the corresponding sequence of parabolically rescaled pointed Ricci flows $(M,g_{k}(t),q)_{t<T_{k}}$ is called a blow-up sequence of $M$ based at $q$ with scales $\lambda_{k}$. A particularly important case for us is when $t_0=1$ is the singular time.
     
\subsection{Heat Kernel at Regular Times on a Ricci Shrinker}\label{heat-kernel-section}

Let $(M,g(t))_{t\in [0,T]}$ be a Ricci flow over an $n$-dimensional manifold $M$. There are two particularly important partial differential equations one can consider, namely the heat equation
\begin{equation}
\partial_{t}u=\Delta u  \label{heat-eq}
\end{equation}
and the conjugate heat equation
\begin{equation}
\partial_{t}u=-\Delta u + S u. \label{conj-heat-eq}
\end{equation}
These two equations are well understood, especially in the static case, and a lot of effort has been devoted to their study. Note that in our setting, \eqref{heat-eq} and \eqref{conj-heat-eq} are coupled with the Ricci flow, in particular $\Delta = \Delta_{g(t)}$ is a time-dependent operator.

As usual, we introduce the following operators acting on smooth functions:
\begin{equation*}
\square = \partial_{t}-\Delta, \quad \square^{\ast}=\partial_{t}+\Delta -S,
\end{equation*}
that allow us to write \eqref{heat-eq} and \eqref{conj-heat-eq} more concisely as $\square u=0$ and $\square^{\ast}u=0$, respectively. 

One of the main results in the field regards the existence of a so-called heat kernel function.
It is well known that such a function exists and is unique when $M$ is a compact manifold, or if $M$ is complete with uniformly bounded scalar curvature; see \cite[Chapter~24]{CCGG-part3} for the relevant definitions and proofs. In the general setting, i.e.~without any curvature assumptions, a similar result is known to hold for Ricci shrinkers, which is the case we are interested in.

Let as usual $(M,g,f)$ be a Ricci shrinker, with associated Ricci flow $(M,g(t))_{t<1}$, and denote
\begin{equation*}
\mathcal{D}=\{(x,t, y,s)\colon x,y \in M, -\infty<s<t<1\}.
\end{equation*}
The following holds; see \cite[Theorem~7]{heatkernel}.

\begin{prop}[Existence and Properties of Heat Kernel on Ricci Shrinkers]\label{heat-kernel-thm}
There exists a positive smooth function $H\colon \mathcal{D}\to \R$
satisfying the following conditions.
\begin{enumerate}
\item For every $y\in M$ and $s<1$, the function $H(\cdot,\cdot \hspace{1pt} ; y,s)$ solves the heat equation \eqref{heat-eq}
\begin{equation*}
\square_{x,t}H(\cdot,\cdot \hspace{1pt} ; y,s)=0
\end{equation*}
on $M\times (s,1)$.
\item For every $x\in M$ and $t<1$, the function $H(x,t;\cdot,\cdot)$ solves the conjugate heat equation \eqref{conj-heat-eq}
\begin{equation*}
\square_{y,s}^{\ast}H(x,t;\cdot,\cdot)=0
\end{equation*}
on $M\times (-\infty,t)$.
\item $H$ converges to a Dirac delta. More precisely
\begin{align*}
\lim_{t\searrow s}H(\cdot,t; y,s)&=\delta_{y}, \quad \forall y\in M, s<1,\\
\ \lim_{s\nearrow t}H(x,t;\cdot,s)&=\delta_{x}, \quad \forall x\in M, t<1
\end{align*}
in the weak* sense. 
\end{enumerate}
Furthermore, $H$ satisfies the semigroup property
\begin{equation}\label{semigroup}
H(x,t;y,s)=\int_{M}H(x,t;z, \varrho)H(z, \varrho; y,s)dV_{\varrho}(z), \quad \forall (x,t,y,s)\in \mathcal{D}, \varrho\in (s,t)
\end{equation}
and the following integral relationships:
\begin{align}
\int_{M}H(x,t; y,s)dV_{t}(x)&\leq 1, \quad \forall y\in M, s<t<1; \label{integral-t}\\
\int_{M}H(x,t; y,s)dV_{s}(y)&=1, \quad \forall x\in M, s<t<1. \label{integral-s}
\end{align}
\end{prop}

\begin{rem}\label{rem.compareconvergences}
We recall here that if $(X,\tau)$ is a Hausdorff, locally compact topological space and
$\mathcal{B}(\tau)$ is its Borel $\sigma$-algebra, a sequence of measures $(\mu_{i})_{i\in \N}$ on $(X,\mathcal{B}(\tau))$ converges \emph{narrowly} to a measure $\mu$ on $(X,\mathcal{B}(\tau))$ if 
\begin{equation}\label{measure-conv-def}
\int_{X}f d\mu_{i}\longrightarrow \int_{X}f d\mu
\end{equation}
for every continuous, bounded function $f\colon X\to \R$. If $\mu_{i}$, $i \in \N$ and $\mu$ are probability measures, narrow convergence is equivalent to weak* convergence, i.e.~it is sufficient to prove \eqref{measure-conv-def} for $f\in C_{c}(X)$, the space of continuous functions on $X$ with compact support; see \cite[Problem~5.10]{bill}\end{rem}

Among the many additional properties satisfied by $H$, for the moment we only need the following on-diagonal upper bound, which holds for every 
$(x,t,y,s)\in \mathcal{D}$:
\begin{equation}\label{ultracontr}
H(x,t; y,s)\leq \frac{e^{-\mu_{g}}}{(4\pi(t-s))^{\frac{n}{2}}}.
\end{equation}
For a proof see \cite[Theorem~15]{heatkernel}.

Finally, note that the functions $\overline{u}=F+\frac{n}{2}t$ and $\overline{v}=(4\pi \tau)^{-n/2}e^{-f(t)}$ are solutions on $M\times (-\infty,1)$ to the heat equation \eqref{heat-eq} and to the conjugate heat equation \eqref{conj-heat-eq}, respectively. Moreover, $\int_{M}\overline{v}(x,t)dV_{g(t)}=1$ for any $t<1$; however $\overline{v}$ does not satisfy Theorem \ref{thm-sing-heat-kernel-intro}.

\subsection{Heat Kernel Based at the Singular Time}

We now prove Theorem \ref{thm-sing-heat-kernel-intro}, dividing it into several propositions for convenience of the reader. We will also deduce some additional properties of the heat kernel at the singular time.

As usual, let $(M,g,f)$ be a Ricci shrinker and let $(M,g(t))_{t<1}$ be its associated Ricci flow. 
Fix a point $q\in M$ and let $\mathsf{t}_{i}\nearrow 1$ be a sequence converging to the singular time $T=1$.
Thanks to the results from Section \ref{heat-kernel-section}, each function 
\begin{equation*}
v_{i}(y,s)=H(q,\mathsf{t}_{i}; y,s)
\end{equation*}
solves the conjugate heat equation \eqref{conj-heat-eq} on $M\times (-\infty,\mathsf{t}_{i})$.
We are going to show that the sequence $(v_{i})_{i\in \N}$ converges to some limit function $v$ and describe its properties.

\begin{prop}[Existence of Heat Kernel Based at the Singular Time]\label{singular-heat-kernel-existence}
There exists a smooth function $v \in C^{\infty}(M\times (-\infty,1))$ such that $(v_{i})_{i \in \N}$ subconverges to $v$ in $C^{\infty}_{\text{loc}}(M\times (-\infty,1))$. Moreover, $v$ is nonnegative, solves the conjugate heat equation \eqref{conj-heat-eq} on $M\times (-\infty,1)$ and, for every $s<1$, the function $v(\cdot,s)$ is a probability density with respect to the volume form $dV_{s}$. 
\end{prop}

\begin{proof}
We start proving that for every compact subset $\Omega \subseteq M\times (-\infty,1)$ and every $m\in \N$, the sequence $(v_{i})_{i\in \N}$ is uniformly bounded in $C^{m}(\Omega)$. It is sufficient to consider compact subsets of the form $\Omega=K\times [a,b]$, with $K\subseteq M$ compact and $[a,b]\subseteq (-\infty,1)$. 
Take $\delta>0$ sufficiently small so that $b+2\delta<1$. Set $T=b-a+2\delta$ and define $\lambda\colon [0,T]\to [a-\delta,b+\delta]$ by $\tau\mapsto b+\delta-\tau$. This is a bijection and $\lambda(I)=[a,b]$, where $I=[\delta,b-a+\delta]$.

Fix $i_{0}\in \N$ such that $\mathsf{t}_{i}>b+2\delta$ for $i\geq i_{0}$ and consider the functions 
\begin{equation*}
u_{i}\colon M\times [0,T]\to \R
\end{equation*}
defined by 
\begin{equation*}
u_{i}(y,\tau)=v_{i}(y,\lambda(\tau)).
\end{equation*}
Setting $\widetilde{g}(\tau)=g(\lambda(\tau))$, then $(M,\widetilde{g}(\tau))_{\tau \in [0,T]}$ is a backward Ricci flow on $M$ while $u_{i}$ satisfies the conjugate backward heat equation \eqref{conjheateq-backward}. From the Schauder estimates in Proposition \ref{schauder-higher-flow}, it follows that for every $(x_{0},\tau_{0})\in M\times I$ there exist $P(x_{0},\tau_{0},\varepsilon)$ and $P^{\ast}(x_{0},\tau_{0},\varepsilon)$, open neighbourhoods of $(x_{0},\tau_{0})$ in $M\times (0,T)$ for some $\varepsilon>0$ and with $P\subseteq P^{\ast}$, such that
\begin{equation}\label{estimate-ui}
\|u_{i} \|_{C^{m}(P)} \leq C_{m}\|u_{i}\|_{L^{\infty}(P^{\ast})}
\end{equation}
for every $m\in \N$. Note that $\varepsilon$ depends only on $\delta$, and thus on $I$, and $C_{m}$ is independent of $u_{i}$.

Now $\mathcal{U}=\{P(x_{0},\tau_{0},\varepsilon)\}_{(x_{0},\tau_{0})\in K\times I}$ is an open cover of $K\times I$, while $P^{\ast}(x_{0},\tau_{0},\varepsilon)\subseteq M\times [0,T]$ for every $(x_{0},\tau_{0})\in K\times I$. Since $K\times I$ is compact we can extract a finite sub-cover from $\mathcal{U}$ and \eqref{estimate-ui}, together with a standard argument, implies
\begin{equation*}
\|u_{i}\|_{C^m(K\times I)} \leq C_{m}\|u_{i}\|_{L^{\infty}(M\times [0,T])}.
\end{equation*}
Returning to the original functions $v_{i}$, observe that the estimate \eqref{ultracontr} gives
\begin{equation*}
v_{i}(y,s)\leq \frac{e^{-\mu_{g}}}{(4\pi(\mathsf{t}_{i}-s))^{\frac{n}{2}}} \leq \frac{e^{-\mu_{g}}}{(4\pi \delta)^{\frac{n}{2}}}
\end{equation*}
for every $(y,s)\in M\times [a-\delta,b+\delta]$, and therefore
\begin{equation}\label{estimate-vi}
\|v_{i}\|_{C^m(K\times [a,b])} \leq C_{m}\|v_{i}\|_{L^{\infty}(M\times [a-\delta,b+\delta])} \leq C_{m}\frac{e^{-\mu_{g}}}{(4\pi\delta)^{\frac{n}{2}}}.
\end{equation}

Since the product manifold $M\times (-\infty,1)$ can be exhausted by a family of compact subsets of the form $(K_{i}\times [-i,1-1/i])_{i \in \N}$, where $(K_{i})_{i\in \N}$ is an exhaustion of $M$, using \eqref{estimate-vi} together with the Arzelà-Ascoli theorem and a diagonal argument, it is possible to extract from $(v_{i})_{i\in \N}$ a subsequence that converges to a smooth limit function $v\in C^{\infty}(M\times (-\infty,1))$ in $C^{\infty}_{\text{loc}}(M\times (-\infty,1))$, proving the existence part of the proposition.

Let us prove the stated additional properties of $v$. First of all, $v\geq 0$ as each $v_{i}$ is positive. 
Since the convergence is smooth on compact subsets, equation \eqref{conj-heat-eq} passes to the limit, and thus $v$ solves the conjugate heat equation on $M\times (-\infty,1)$. It remains to prove that $v(\cdot,s)$ is a probability density for $dV_{s}$; the proof of this fact is a little bit more involved.

Fix a time $\overline{s}<1$. Then $\int_{M}v_{i}(y,\overline{s})dV_{\overline{s}}(y)=1$ due to \eqref{integral-s}. Since $v_{i}(\cdot,\overline{s})\to v(\cdot, \overline{s})$ pointwise on $M$, Fatou's lemma implies
\begin{equation*}
\lambda:=\int_{M}v(y,\overline{s})dV_{\overline{s}}(y)\leq 1.
\end{equation*}
We want to show that $\lambda=1$, or equivalently $1-\lambda=0$. Let $K$ be a compact subset to be determined later. Write
\begin{align*}
0 \leq 1-\lambda &= \left(1-\int_{K}v_{i}(y,\overline{s})dV_{\overline{s}}(y)\right)
+ \left(\int_{K}v_{i}(y,\overline{s})dV_{\overline{s}}(y)-\int_{K}v(y,\overline{s})dV_{\overline{s}}(y)\right)\\
&\quad + \left(\int_{K}v(y,\overline{s})dV_{\overline{s}}(y) -\lambda\right)
=I+I\! I+I\! I\! I.
\end{align*}
Clearly $I\! I\! I \leq 0$, while
\begin{equation*}
| I\! I |\leq \int_{K}|v_{i}(y,\overline{s})-v(y,\overline{s})|dV_{\overline{s}}(y)\leq \Vol_{\overline{s}}(K)\sup_{K}|v_{i}(y,\overline{s})-v(y,\overline{s})|.
\end{equation*}
It remains to estimate the term $I$. First of all, recall that there exists a smooth family of cut-off functions $(\phi^{r}(y,s))_{r>0}$ as in \cite[Lemma~3]{heatkernel}.
By construction, $\phi^{r}(\cdot,s)\equiv 1$ on $\{F(\cdot,s)\leq r\}$ and $\phi^{r}(\cdot,s)=0$ on $\{F(\cdot,s)\geq 2r\}$. Now from \eqref{potential-time-deriv} one sees that $\partial_{s}F(y,s)\leq 0$. Therefore $F(\cdot,s)\leq F(\cdot,\overline{s})$ on $M$ if $s\geq \overline{s}$.
Assume $r>0$ is large enough such that $q\in \{F(\cdot,\overline{s})\leq r\}$. Then $q\in \{F(\cdot,s)\leq r\}$ also for $s\geq \overline{s}$, that is $\phi^{r}(q,s)=1$ for such $s$. Defining $K=\{F(\cdot,\overline{s})\leq 3r\}$ then $\supp \phi^{r}(\cdot,\overline{s})\subseteq K$ and we have
\begin{equation*}
\int_{K}v_{i}(y,\overline{s})dV_{\overline{s}}(y)=\int_{M}v_{i}(y,\overline{s})\mathbf{1}_{K}(y)dV_{\overline{s}}(y) \geq \int_{M}v_{i}(y,\overline{s})\phi^{r}(y,\overline{s})dV_{\overline{s}}(y).
\end{equation*}
Therefore
\begin{equation*}
I=1-\int_{K}v_{i}(y,\overline{s})dV_{\overline{s}} \leq 1-\int_{M}v_{i}(y,\overline{s})\phi^{r}(y,\overline{s})dV_{\overline{s}}(y).
\end{equation*}
Consider the function 
\begin{equation*}
\I_{i}(s)=\int_{M}v_{i}(y,s)\phi^{r}(y,s)dV_{s}(y)
\end{equation*}
for $s\in [\overline{s},\mathsf{t}_{i})$. 
Then we can differentiate under the integral sign, obtaining
\begin{equation*}
\partial_{s}\I_{i}(s)=\int_{M}v_{i}(y,s)\square_{y,s}\phi^{r}(y,s)dV_{s}(y),
\end{equation*}
where we used the conjugate heat equation, the Green formulas and the evolution equation for the volume density under Ricci flow. Since by \cite[Lemma~3]{heatkernel} we have $|\square_{y,s}\phi^{r}(y,s)|\leq C(n)r^{-1}$, it follows that
\begin{equation*}
|\partial_{s}\I_{i}(s)|\leq \int_{M}v_{i}(y,s)|\square_{y,s}\phi^{r}(y,s)|dV_{s}(y)\leq C(n)r^{-1}
\end{equation*}
for every $s\in [\overline{s},\mathsf{t}_{i})$. 
Therefore
\begin{equation*}
|\I_{i}(s)-\I_{i}(\overline{s})|\leq \int_{\overline{s}}^{s} C(n)r^{-1}d\sigma=C(n)(s-\overline{s})r^{-1}.
\end{equation*}
Taking the limit for $s\nearrow \mathsf{t}_{i}$ one gets
\begin{equation*}
\lim_{s\nearrow \mathsf{t}_{i}}\I_{i}(s)=\lim_{s\nearrow \mathsf{t}_{i}}\int_{M}v_{i}(y,s)\phi^{r}(y,s)dV_{s}(y)=\delta_{q}(\phi^{r}(y,\mathsf{t}_{i}))=1,
\end{equation*}
since $v_{i}(\cdot,s)\to \delta_{q}$ in the weak* sense. 
At this point we obtain
\begin{equation*}
| I | = |1-\I_{i}(\overline{s})|\leq C(n)(\mathsf{t}_{i}-\overline{s})r^{-1} \leq C(n)(1-\overline{s})r^{-1} .
\end{equation*}
Combining the previous estimates, we have
\begin{equation*}
1-\lambda \leq \Vol_{\overline{s}}(K)\sup_{K}|v_{i}(y,\overline{s})-v(y,\overline{s})| + C(n)(1-\overline{s})r^{-1}.
\end{equation*}
Now given any $\varepsilon>0$ there exists $r>0$ such that $ C(n)(1-\overline{s})r^{-1} \leq \varepsilon$. Since $K=\{F(\cdot,\overline{s})\leq 3r\}$ is fixed and $v_{i}(\cdot,\overline{s})\to v(\cdot,\overline{s})$ uniformly on $K$, there exists $i_{1}\geq i_{0}$ such that 
\begin{equation*}
\Vol_{\overline{s}}(K)\sup_{K}|v_{i}(y,\overline{s})-v(y,\overline{s})| \leq \varepsilon
\end{equation*} 
for every $i\geq i_{1}$. It follows $1-\lambda \leq 2 \varepsilon$ and since $\varepsilon>0$ was arbitrary, we are done.
\end{proof}

\begin{defn} 
A function $v$ obtained as above is called a \emph{conjugate heat kernel based at the singular point $(q,1)$.} It will be denoted by $v_{q}$. 
\end{defn}

\begin{rem} Observe that, in general, $v_{q}$ could depend on the sequence $\mathsf{t}_{i}\nearrow 1$ and on the subsequence chosen in the proof.
\end{rem}

From the construction, it is possible to deduce some other properties of $v_{q}$, similar to the ones seen in Section \ref{heat-kernel-section} for the heat kernel at regular times. First of all, passing to the limit in \eqref{ultracontr} we obtain the inequality 
\begin{equation}\label{ultracontr-sing}
v_{q}(y,s)\leq \frac{e^{-\mu_{g}}}{(4\pi(1-s))^{\frac{n}{2}}}
\end{equation}
for every $y\in M$ and $s<1$. 

Using $H$ and $v_{q}$ we can introduce the following probability measures on the measurable space $(M,\mathcal{B}(M))$,
where $\mathcal{B}(M)$ is the Borel $\sigma$-algebra induced by the manifold topology.
For $s<t\leq 1$, define:
\begin{equation*}
d\nu_{(x,t);s}=\begin{cases}
H(x,t; \cdot, s)dV_{s},\quad &\text{if } x\in M\ \text{and } t<1,\\
v_{q}(\cdot,s)dV_{s}, \quad &\text{if } x=q\ \text{and } t=1.
\end{cases}
\end{equation*}
For $x=q$ and $t=1$, we also write $d\nu_{q;s} = d\nu_{(q,1);s}$. The following convergence result will allow us to extend the semigroup property \eqref{semigroup} to $\nu_{q;s}$. 

\begin{prop}[Narrow Convergence of Conjugate Heat Kernel Measures]\label{weakconv} 
The sequence of probability measures $(\nu_{(q,\mathsf{t}_{i});s})_{i\in \N}$ converges narrowly to $\nu_{q;s}$ for $i\to \infty$. 
\end{prop}
\begin{proof}
Take $f\in C_{c}(M)$ and let $K=\supp f$ be its support; also denote $C=\sup |f|$. Then
\begin{align*}
\left \vert \int_{M}f d\nu_{(q,\mathsf{t}_{i});s} - \int_{M}f d\nu_{q;s}\right\vert 
&= \left\vert \int_{K}(v_{i}(y,s)-v_{q}(y,s))f(y)dV_{s}(y)\right\vert\\
&\leq C\int_{K}|v_{i}(y,s)-v_{q}(y,s)|dV_{s}(y)\\
&\leq C \Vol_{s}(K)\sup_{K}|v_{i}(y,s)-v_{q}(y,s)|.
\end{align*}
Since $v_{i}(\cdot,s)\to v_{q}(\cdot,s)$ uniformly on $K$, we get weak* convergence. As all $\nu_{(q,\mathsf{t}_{i});s}$ as well as the limit $\nu_{q;s}$ are probability measures, this implies narrow convergence, as explained in Remark \ref{rem.compareconvergences}.
\end{proof}

\begin{prop}[Semigroup Property]\label{prop.semigroup}
We have
\begin{equation}\label{semigroup-sing}
v_{q}(y,s)=\int_{M}H(z,\varrho; y,s)v_{q}(z,\varrho)dV_{\varrho}(z)
\end{equation}
for every $y\in M$, $s<1$ and $\varrho \in (s,1)$.
\begin{proof}
Applying the semigroup property \eqref{semigroup} to $v_{i}$, we can write, for $y\in M, s<1$ and $\varrho\in (s,\mathsf{t}_{i})$ fixed,
\begin{equation}\label{semigroup-vi}
v_{i}(y,s)=\int_{M}H(z,\varrho; y,s)v_{i}(z,\varrho)dV_{\varrho}(z)= \int_{M}f(z)d\nu_{(q, \mathsf{t}_{i});s}(z),
\end{equation}
where $f(z)=H(z,\varrho; y,s)$. As the function $f$ is continuous and bounded thanks to \eqref{ultracontr}, the claim follows from Proposition \ref{weakconv} by passing to the limit in \eqref{semigroup-vi}.
\end{proof}
\end{prop}

As a direct consequence, we can prove the positivity of the singular time heat kernel and the equivalence between the measures introduced so far. 

\begin{cor}[Positivity of Singular Time Heat Kernel]\label{cor.positivity}
The singular time heat kernel $v_{q}$ is strictly positive everywhere. Moreover, for every $(x,t)\in M\times (-\infty,1)$ and $s<t$, the measures $dV_{s}$, $\nu_{(x,t);s}$ and $\nu_{q;s}$ are all equivalent. In particular, they all have full support.  
\end{cor}

\begin{proof}
Suppose towards a contradiction that there exists a point $(y,s)\in M\times (-\infty,1)$ with $v_{q}(y,s)=0$, and let $\varrho\in (s,1)$ be arbitrary. 
Then by the semigroup property
\begin{equation*}
v_{q}(y,s)=0=\int_{M}H(z,\varrho; y,s)v_{q}(z,\varrho)dV_{\varrho}(z).
\end{equation*}
As the integrand is nonnegative, we deduce $H(z,\varrho; y,s)v_{q}(z,\varrho)\equiv 0$ on $M$, and therefore $v_{q}(z,\varrho)\equiv 0$ since the standard heat kernel $H$ is positive. On the other hand, $v_{q}(\cdot,\varrho)$ is a probability density which gives the desired contradiction. 

For the additional statement, it is sufficient to observe that $\nu_{(x,t);s}$ and $\nu_{q;s}$ are equivalent to the Riemannian volume measure $dV_{s}$, having strictly positive densities with respect to it. The last statement follows, since $dV_{s}$ has (obviously) full support. 
\end{proof}

Theorem \ref{thm-sing-heat-kernel-intro} follows by combining Propositions \ref{singular-heat-kernel-existence} and \ref{prop.semigroup} as well as Corollary \ref{cor.positivity}.

%------------------------------------------------------------------------------------------------------------
\section{Ricci Shrinkers as Metric Flows}\label{shrinker-mf-section}
\subsection{Variance and H-Centers on Ricci Shrinkers}

Let $(M,g,f)$ be a Ricci shrinker and let $(M,g(t))_{t<1}$ be the associated Ricci flow. Obviously, $(M,g(t))$ is a metric space with the induced distance $d_{t}$, and since all the metrics $d_{t}$ induce the manifold topology the $\sigma$-algebras $\mathcal{B}(M,d_{t})$ coincide with $\mathcal{B}(M)$. As in \cite{heatkernel2}, one can consider the variance and the $W_{1}$-Wasserstein distance associated to $(M,d_{t})$, which will be denoted by 
$\var_{t}$ and $d_{W_{1}}^{t}$, respectively. The reader can consult \cite[Section~2]{Ba20_compactness} for a detailed exposition of the topic in the more general context of metric measure spaces. Here we limit ourselves to state the following useful result; see \cite[Lemma~2.8, Lemma~2.9]{Ba20_compactness}. 

\begin{prop}[Basic Relations Between Variance and Wasserstein Distance]\label{lsc-var-prop}  
Let $(X,d)$ be a complete, separable metric space, and let $\mathcal{P}(X,d)$ be the set of its probability measures.
\begin{enumerate}
\item For any $\mu_{1},\mu_{2} \in \mathcal{P}(X,d)$ we have
\begin{equation}\label{wdistance-variance}
d_{W_{1}}(\mu_{1},\mu_{2})\leq \sqrt{\var(\mu_{1},\mu_{2})}\leq d_{W_{1}}(\mu_{1},\mu_{2})+\sqrt{\var(\mu_{1})}+\sqrt{\var(\mu_{2})}.
\end{equation}
\item If a sequence $(\mu_{i})_{i\in \N}\subseteq \mathcal{P}(X,d)$ converges narrowly to 
$\mu\in \mathcal{P}(X,d)$ and $\var(\mu_{i})\leq C<\infty$ for every $i\in \N$, then it converges to $\mu$ also with respect to 
$d_{W_{1}}$, and the following inequality holds 
\begin{equation}\label{lsc-var}
\var(\mu)\leq \liminf_{i\to \infty}\var(\mu_{i}).
\end{equation}
\end{enumerate}
\end{prop}

We are mainly concerned with the properties of the probability measures induced by the standard and the singular time heat kernels. The following result gives a bound on their variance.

\begin{prop}[Bound on Variance of Heat Kernel]\label{concentration-shrinker} 
Defining $H_{n}=\frac{1}{2}(n-1)^{2}\pi^{2} + 4$, the following concentration properties hold. 
\begin{enumerate}
\item Let $(x,t)\in M\times (-\infty,1)$ and consider $\nu_{(x,t);s}$. Then
\begin{equation}\label{standard-prob-variance}
\var_{s}(\nu_{(x,t);s})\leq H_{n}(t-s)
\end{equation}
for every $s<t$. 
\item Let $q\in M$ and consider $\nu_{q;s}$. Then
\begin{equation}\label{conj-prob-variance}
\var_{s}(\nu_{q;s})\leq H_{n}(1-s)
\end{equation}
for every $s<1$.
\end{enumerate}
\end{prop}

\begin{proof}
Estimate \eqref{standard-prob-variance} is proved in \cite[Proposition~3.10]{heatkernel2}. To prove \eqref{conj-prob-variance}, let $\mathsf{t}_{i}\nearrow 1$ be the sequence defining $v_{q}$ and fix a time $s<1$. 
Applying the previous point to $\nu_{(q,\mathsf{t}_{i});s}$ for $\mathsf{t}_{i}>s$, we find
\begin{equation*}
\var_{s}(\nu_{(q,\mathsf{t}_{i});s})\leq H_{n}(\mathsf{t}_{i}-s)\leq H_{n}(1-s).
\end{equation*}
Since $\nu_{(q,\mathsf{t}_{i});s}\to \nu_{q;s}$ narrowly, Proposition \ref{lsc-var-prop} gives the claim.
\end{proof}

An important concept in Bamler's theory is the one of $H$-centers. 

\begin{defn}\label{def.Hcenter}
Given $H>0$, a point $(z,s)\in M\times (-\infty,1)$ is called: 
\begin{enumerate}
\item An \emph{$H$-center of} $(x,t)\in M\times (-\infty,1)$ if $s\leq t$ and
\begin{equation*}
\var_{s}(\nu_{(x,t);s},\delta_{z})\leq H(t-s);
\end{equation*}
\item An \emph{$H$-center of} $(q,1)$ if
\begin{equation*}
\var_{s}(\nu_{q;s},\delta_{z})\leq H(1-s).
\end{equation*}
\end{enumerate}
\end{defn}

\begin{rem} In both cases, \eqref{wdistance-variance} gives bounds on the $W_{1}$-Wasserstein distances from $\delta_{z}$ as well. 
\end{rem}

Generalising results of Bamler about compact Ricci flows or complete Ricci flows with bounded curvature to Ricci flows associated with Ricci shrinkers (without any curvature assumption), Li and Wang proved the existence of $H_{n}$-centers for the standard (i.e.~regular time) heat kernel in \cite[Proposition~3.12]{heatkernel2}. Similarly, for the singular time heat kernel we have the following.

\begin{prop}[Existence of $H_n$-Centers]\label{center-singular-time} 
For every $s< 1$, there exists an $H_{n}$-center $(z,s)$ of $(q,1)$ as in Point 2 of Definition \ref{def.Hcenter}.
\begin{proof}
Writing
\begin{align*}
\var_{s}(\nu_{q;s})&=\int_{M}\int_{M}d_{s}^{2}(x,y)d\nu_{q;s}(x)d\nu_{q;s}(y)=\int_{M}\left(\int_{M}d_{s}^{2}(x,y)d\nu_{q;s}(x)\right)d\nu_{q;s}(y)\\
&=\int_{M}\var_{s}(\nu_{q;s},\delta_{y})d\nu_{q;s}(y),
\end{align*}
we deduce
\begin{equation*}
\int_{M}\var_{s}(\delta_{y},\nu_{q;s})d\nu_{q;s}(y)\leq H_{n}(1-s)
\end{equation*}
by Proposition \ref{concentration-shrinker}. Being $\nu_{q;s}$ a probability measure, there exists at least one point $z\in M$ such that $\var_{s}(\delta_{z},\nu_{q;s})\leq H_{n}(1-s)$, and the corresponding $(z,s)$ is the desired $H_{n}$-center. 
\end{proof}
\end{prop}

\begin{rem}
Actually, as seen from the proof, for every fixed $s<1$ there exists a positive measure set of $H_{n}$-centers of $(q,1)$. The same holds for the standard heat kernel measure. Moreover, obviously, if $H \geq H_{n}$ and $(z,t)$ is an $H_{n}$-center, then it is also an $H$-center.  
\end{rem}

\begin{cor}[Wasserstein Distance to Dirac Delta]\label{centers-sequence} 
Given $s_{i}\nearrow 1$, there exists a sequence of points $(z_{i})_{i\in \N}$ in $M$ such that
\begin{equation*}
\lim_{i\to \infty}d_{W_{1}}^{s_{i}}(\nu_{q;s_{i}},\delta_{z_{i}})=0.
\end{equation*}
\end{cor}

\begin{proof}
For every $i \in \N$ let $(z_{i},s_{i})$ be an $H_{n}$-center of $(q,1)$. Then
\begin{equation*}
d_{W_{1}}^{s_i}(\nu_{q;s_{i}},\delta_{z_{i}})\leq \sqrt{\var_{s_{i}}(\nu_{q;s_{i}},\delta_{z_{i}})}\leq \sqrt{H_{n}(1-s_i)}\longrightarrow 0
\end{equation*}
as $i\to \infty$.
\end{proof}

\subsection{Heat Kernel under Parabolic Rescaling}

Consider a Ricci flow $(M,g_{-t_{0},\lambda}(t))_{t<T}$ induced by a Ricci shrinker $(M,g,f)$ as in Definition \ref{ricci-flow-induced}, where $T$ is the singular time. Let $H(x,t;y,s)$ be the heat kernel of the standard flow $(M,g(t))_{t<1}$ given by Proposition \ref{heat-kernel-thm}. If we define the function 
\begin{equation*}
H^{-t_{0},\lambda}(x,t; y,s)=\lambda^{-\frac{n}{2}}H(x,t_{0}+t/\lambda; y, t_{0}+s/\lambda)
\end{equation*}
for $(x,t; y,s)\in \mathcal{D}^{-t_{0},\lambda}$, where $\mathcal{D}^{-t_{0},\lambda}=\{(x,t;y,s)\colon x,y\in M, -\infty<s<t<\lambda(1-t_{0})\}$, it is easy to check that $H^{-t_{0},\lambda}$ satisfies an appropriate version of Proposition \ref{heat-kernel-thm} with respect to the flow $(M,g_{-t_{0},\lambda}(t))_{t<T}$. Therefore, all the theory seen in the previous sections carries over almost verbatim to this case, up to very minor modifications. 

Indeed, given a point $q\in M$, we can consider the singular heat kernel $v^{-t_{0},\lambda}_{q,T}$ based at $(q,T)$, constructed starting from a sequence of times $\mathsf{t}_{i}\nearrow T$ and taking the limit of (a subsequence of) the functions $H^{-t_{0},\lambda}(q,\mathsf{t}_{i};y,s)$ in $C^{\infty}_{\text{loc}}(M\times (-\infty,T))$. Once $v^{-t_{0},\lambda}_{q,T}$ is defined, we introduce the probability measures $\nu^{-t_{0},\lambda}_{(q,T);t}$ and $\nu^{-t_{0},\lambda}_{(x,t);s}$, where $x\in M$ and $s<t<T$, given by
\begin{equation*}
d\nu^{-t_{0},\lambda}_{(x,t);s}=H^{-t_{0},\lambda}(x,t;y,s)dV_{g_{-t_{0},\lambda}(s)}, \quad \quad 
d\nu^{-t_{0},\lambda}_{(q,T);s}=v^{-t_{0},\lambda}_{q,T}(y,s)dV_{g_{-t_{0},\lambda}(s)}.
\end{equation*}
Note that
\begin{equation*}
d\nu^{-t_{0},\lambda}_{(x,t);s}=d\nu_{(x,t_{0}+t/\lambda);t_{0}+s/\lambda}, \quad \quad d\nu^{-t_{0},\lambda}_{(q,T);s}=d\nu_{(q,1);t_{0}+s/\lambda}
\end{equation*}
and therefore \eqref{standard-prob-variance} holds for $\nu^{-t_{0},\lambda}_{(x,t);s}$ in the same form, while \eqref{conj-prob-variance} translates to
\begin{equation*}
\var_{s}(\nu^{-t_{0},\lambda}_{(q,T);s})\leq H_{n}(T-s),
\end{equation*}
where now the variances are computed with respect to the distance $d_{g_{-t_{0},\lambda}(s)}$.  
Similarly, Proposition \ref{center-singular-time} and Corollary \ref{centers-sequence} can be rephrased as follows.

\begin{prop}[$H_n$-Centers for Rescaled Heat Kernels]\label{centers-rescaling} 
Consider the probability measures $\nu^{-t_{0},\lambda}_{(q,T);s}$ for $s<T$.
\begin{enumerate}
\item For every $s<T$ there exists $z \in M$ such that
\begin{equation*}
\var_{s}(\nu^{-t_{0},\lambda}_{(q,T);s},\delta_{z})\leq H_{n}(T-s).
\end{equation*}
\item Given $s_{i}\nearrow T$ there exists a sequence of points $(z_{i})_{i\in \N}$ such that
\begin{equation*}
\lim_{i\to \infty} d_{W_{1}}^{s_{i}}(\nu^{-t_{0},\lambda}_{(q,T);s_{i}},\delta_{\z_{i}})=0.
\end{equation*}
\end{enumerate}
\end{prop}

\subsection{Metric Flows Associated to Ricci Shrinkers}

We now show how to associate a metric flow to a Ricci shrinker. For the main definitions and results regarding metric flows, the reader can consult \cite{Ba20_compactness}.
  
Let $(M,g(t))_{t<T}$ be an ancient Ricci flow induced by a Ricci shrinker $(M,g,f)$, with singular time $T<\infty$. Recall that if $x\in M$ and $t<T$ we introduced the family of probability measures $(\nu_{(x,t);s})_{s< t}$, whose densities are given by the heat kernel.

Given an open interval $I\subseteq (-\infty,T)$, we now construct a metric flow $\X$ over $I$. First of all, let
\begin{equation*}
\X=M\times I
\end{equation*}
and denote by $\mathfrak{t}\colon \X\to I$ the projection onto the second factor. Then each time slice $\X_{t}:=\mathfrak{t}^{-1}(t)$ is given by $\X_{t}=M\times\{t\}\simeq M$, and we endow $\X_{t}$ with the Riemannian distance $d_{t}=d_{g(t)}$, turning it into a complete metric space $(\X_{t},d_{t})$. 

For a point $\x=(x,t) \in \X$ define $\nu_{\x;s}$ as follows:
\begin{equation}\label{prob-mf}
\nu_{\x;s}=\begin{cases} 
\nu_{(x,t);s}\quad &\text{if } s<t; \\
\delta_{x}\quad &\text{if }\ s=t.
\end{cases}
\end{equation}
Then $\nu_{\x;s}\in \mathcal{P}(\X_{s})$ for every $s\in I$, $s\leq t$. Since the notations agree we will often write $\nu_{(x,t);s}$ instead of $\nu_{\x;s}$, explicitly referring to the time $t=\mathfrak{t}(\x)$. 
 
Finally, the desired \emph{metric flow} is the tuple
\begin{equation}\label{def.metricflow}
(\X,\mathfrak{t},(d_{t})_{t\in I},(\nu_{\x;s})_{\x \in \X,s\in I,s\leq \mathfrak{t}(\x)}),
\end{equation}
which will be concisely denoted by $\X$.

\begin{prop} 
$\X$ is a metric flow over the interval $I$ in the sense of Bamler \cite[Definition~3.1]{Ba20_compactness}.
\end{prop}

\begin{proof}
In view of the previous subsection, we can, without loss of generality, assume that $T=1$, in order simplify notation.

To prove the proposition, we must verify items $(1)-(7)$ in \cite[Definition~3.1]{Ba20_compactness}. However, we notice that Points $(1)-(5)$ follow directly from the above construction, so it only remains to check Points $(6)$ and $(7)$.
\begin{enumerate}\setcounter{enumi}{5}
\item Fix $s<t \in I$, $\mathsf{T}\geq 0$ and let $u_{s}\colon \X_{s}\to [0,1]$ be a measurable function. 
As stated in \cite[Lemma~3.3]{Ba20_compactness}, it is sufficient to consider $\mathsf{T}>0$ and $u_{s}$ with values in $(0,1)$. Suppose that $\Phi^{-1}\circ u_{s}$ is $\mathsf{T}^{-1/2}$-Lipschitz on $\X_{s}$ and denote $\Phi_{t}(x)=\Phi(t^{-1/2}x)$ for $t>0$. The assumption can be rephrased as $|\nabla (\Phi^{-1}\circ u_{s})|\leq \mathsf{T}^{-1/2}$ or, equivalently, $|\nabla (\Phi_{\mathsf{T}}^{-1}\circ u_{s})|\leq 1$ almost everywhere. From \cite[Lemma~5]{heatkernel}, the function
\begin{equation*}
u(x,\tau)=\int_{M}H(x,\tau; y,s)u_{s}(y)dV_{s}(y)=\int_{M}u_{s}(y)d\nu_{(x,\tau);s}(y),
\end{equation*}
with $(x,\tau)\in M\times [s,t]$, is the unique bounded solution of $\square u=0$ on $M\times (s,t]$ with initial value $u_{s}$. Since $\nu_{(x,\tau);s}$ is a probability measure, $u(\cdot,\tau)$ takes values in $(0,1)$ and \cite[Theorem~3.15]{heatkernel2} applies to $u\in C^{\infty}([s+\varepsilon,t])$, where $\varepsilon\in (0,t-s)$, giving
\begin{equation*}
|\nabla (\Phi_{\mathsf{T}+\tau-s-\varepsilon}^{-1}\circ u(\cdot,\tau))|\leq 1,
\end{equation*}
or in other words
\begin{equation*}
|\nabla (\Phi^{-1}\circ u(\cdot,\tau))|\leq (\mathsf{T}+\tau-s-\varepsilon)^{-1/2}
\end{equation*}
for every $\tau \in [s+\varepsilon,t]$. Choosing $\tau=t$ we see that $u_{t}(x)=u(x,t)$ is a $(\mathsf{T}+t-s-\varepsilon)^{-1/2}$-Lipschitz function on $\X_{t}$; letting $\varepsilon \to 0$ it follows that $u_{t}(x)$ is $(\mathsf{T}+t-s)^{-1/2}$-Lipschitz. Finally, noticing that
\begin{equation*}
u_{t}(x)=\int_{\X_{s}}u_{s}(y)d\nu_{(x,t);s}(y),
\end{equation*}
the claim follows. 
\item This is an application of the semigroup property \eqref{semigroup}. Indeed, let $t_{1}\leq t_{2}\leq t_{3}\in I$ and $\x \in \X_{t_{3}}$. We need to show that
\begin{equation}\label{repr-mf}
\nu_{\x;t_{1}}(S)=\int_{\X_{t_{2}}}\nu_{\y;t_{1}}(S)d\nu_{\x;t_{2}}(\y)
\end{equation}
for every Borel set $S\subseteq \X_{t_{2}}$. 

Assuming $t_{1}<t_{2}<t_{3}$ we have
\begin{align*}
\nu_{\x;t_{1}}(S)&=\int_{M}\mathbf{1}_{S}(y)d\nu_{(x,t_{3});t_{1}}(y)=\int_{M}\mathbf{1}_{S}(y)H(x,t_{3},y,t_{1})dV_{t_{1}}(y)\\
&=\int_{M}\mathbf{1}_{S}(y)\left(\int_{M}H(x,t_{3},z,t_{2})H(z,t_{2},y,t_{1})dV_{t_{2}}(z) \right)dV_{t_{1}}(y)\\
&=\int_{M}H(x,t_{3},z,t_{2})\left(\int_{M}\mathbf{1}_{S}(y)H(z,t_{2},y,t_{1})dV_{t_{1}}(y) \right)dV_{t_{2}}(z)\\
&=\int_{M}H(x,t_{3},z,t_{2})\nu_{(z,t_{2});t_{1}}(S)dV_{t_{2}}(z)\\
&=\int_{M}\nu_{(y,t_{2});t_{1}}(S)d\nu_{(x,t_{3});t_{2}}(y)=\int_{\X_{t_{2}}}\nu_{\y;t_{1}}(S)d\nu_{\x;t_{2}}(\y),
\end{align*}
as required. 

The remaining cases, that is when $t_{1}=t_{2}<t_{3}$, $t_{1}<t_{2}=t_{3}$ or $t_{1}=t_{2}=t_{3}$, can be proved similarly, recalling Definition \eqref{prob-mf}.\qedhere
\end{enumerate}
\end{proof}

\begin{defn} 
The metric flow $\X$ over $I$ constructed above is called the \emph{metric flow corresponding to the Ricci flow} $(M,g(t))_{t\in I}$.
\end{defn}

One of the main properties of the metric flow $\X$ is stated in the following result.

\begin{prop}[Properties of Metric Flow $\X$]
The metric flow $\X$ is $H_{n}$-concentrated and continuous in the sense of Bamler \cite[Definitions 3.21 and 4.7]{Ba20_compactness}.
\end{prop}

\begin{proof}
Let $s\leq t\in I$ and $\x,\y\in \X_{t}$; then \cite[Proposition~3.10]{heatkernel2} gives
\begin{equation*}
\var_{s}(\nu_{\x;s},\nu_{\y;s})=\var_{s}(\nu_{(x,t);s},\nu_{(y,t);s})\leq d_{t}^{2}(x,y)+H_{n}(t-s)
\end{equation*}
and the $H_{n}$-concentration follows.

For the continuity statement, suppose first that $(M,g(t))_{t<0}$ is obtained by time-shifting the Ricci flow associated with the shrinker. Then, as shown in \cite[Example~3.44]{Ba20_compactness}, it is possible to define a heat flow $(\mu_{t})_{t<0}$ on $\X$ so that $(\X,(\mu_{t})_{t<0})$ is a metric soliton. This implies that $\X$ is continuous, see Proposition \ref{ms-cont}, and the general case easily follows. 
\end{proof}

So far, in this subsection, the singular time heat kernel and the associated measure did not play any role, but they will show up now in the context of conjugate heat flows and metric flow pairs.  

Fix a point $q\in M$ and let $t_{\max}=\sup I \leq T$; 
recall that we have defined the family of probability measures $(\nu_{(q,t_{\max});t})_{t<t_{\max}}$,
with the \emph{caveat} that we use the singular time heat kernel $v_{q,T}$ when $t_{\max}=T$. For any $t\in I$ define
\begin{equation*}
\mu_{t}=\nu_{(q,t_{\max});t}.
\end{equation*}

\begin{prop}[Conjugate Heat Flow]
$(\mu_{t})_{t\in I}$ is a conjugate heat flow on the metric flow $\X$ in the sense of Bamler \cite[Definition 3.13]{Ba20_compactness}.
\end{prop}

\begin{proof}
For simplicity assume once again, without loss of generality, that $T=1$.

Given $s\leq t \in I$ we need to prove that
\begin{equation}\label{conj-heat-kernel-repr}
\mu_{s}(S)=\int_{\X_{t}}\nu_{\x;s}(S)d\mu_{t}(\x)
\end{equation}
for every $S\subseteq \X_{s}$ Borel set. 

If $t_{\max}<1$, this is simply a restatement of the reproduction formula \eqref{repr-mf}, with $t_{1}=s$, $t_{2}=t$, $t_{3}=t_{\max}$ and $\x=(q,t_{\max})$, since in this case $\nu_{\x;s}=\mu_{s}$ and $\nu_{\x;t}=\mu_{t}$.

Suppose $t_{\max}=1$, and thus $\mu_{s}$ is the heat kernel measure based at the singular point $(q,t_{\max})=(q,1)$. In this case, the claim will follow from the semigroup property \eqref{semigroup-sing}; indeed
\begin{align*}
\mu_{s}(S)&=\nu_{(q,1);s}(S)=\int_{M}\mathbf{1}_{S}(y)d\nu_{(q,1);s}(y)=\int_{M}\mathbf{1}_{S}(y)v_{q}(y,s)dV_{s}(y)\\
&=\int_{M}\mathbf{1}_{S}(y)\left(\int_{M}H(z,t; y,s)v_{q}(z,t)dV_{t}(z) \right)dV_{s}(y)\\
&=\int_{M}v_{q}(z,t)\left(\int_{M}\mathbf{1}_{S}(y)H(z,t; y,s)dV_{s}(y) \right)dV_{t}(z)\\
&=\int_{M}v_{q}(z,t) \nu_{(z,t);s}(S)dV_{t}(z)
=\int_{M}\nu_{(y,t);s}(S)d\nu_{(q,1);t}(y)=\int_{\X_{t}}\nu_{\y;s}(S)d\mu_{t}(\y),
\end{align*}
as required.
\end{proof}

Using the terminology of \cite[Definition~3.21, Definition~5.1]{Ba20_compactness}, and recalling Proposition \ref{concentration-shrinker}, we can state the following immediate summary.

\begin{cor}[Metric Flow Pair] 
The pair $(\X, (\mu_{t})_{t\in I})$ is an $H_{n}$-concentrated metric flow pair over $I$ that is fully defined over $I$.
\end{cor}

\begin{rem} The fact that $\mu_{t}$ has full support can also be seen as a consequence of \cite[Proposition~3.16]{Ba20_compactness}. Indeed, $\mu_{t}$ and $\nu_{(x,\tau);t}$ have the same support for $t<\tau<t_{\max}$ and $\supp \nu_{(x,\tau);t}=M$ by positivity of the heat kernel.
\end{rem}

\begin{defn} 
The metric flow pair $(\X, (\mu_{t})_{t\in I})$ over $I$ is called the \emph{metric flow pair corresponding to the pointed Ricci flow} $(M,g(t), (q,t_{\max}))_{t\in I}$; we also say that $(q,t_{\max})$ is the \emph{basepoint} of the metric flow pair.
\end{defn}

\begin{rem} 
We underline that, in defining the metric flow pair of a Ricci flow, its basepoint plays a fundamental role, since it specifies which is the conjugate heat flow to consider. The basepoint is always a space-time point from $M\times\{t_{\max}\}$.
\end{rem}

Note that thanks to the definition, $\mu_{t}$ satisfies a concentration property as well, since
\begin{equation}\label{var-metric-flow}
\var_{t}(\mu_{t})\leq H_{n}(t_{\max}-t)
\end{equation}
for every $t\in I$. In particular, 
\begin{equation}\label{var-goes-to-zero}
\lim_{t\nearrow t_{\max}}\var_{t}(\mu_{t})=0,
\end{equation}
and therefore (see \cite[Remark~5.9]{Ba20_compactness}) we can extend the metric flow pair $(\X,(\mu_{t})_{t\in I})$ to the final time $t_{\max}$, obtaining the \emph{extended} metric flow pair $(\X,(\mu_{t})_{t\in I\cup\{t_{\max}\}})$, which is fully defined over $I\cup\{t_{\max}\}$. In particular, $\X_{t_{\max}}=\{\x_{\max}\}$, $\mu_{t_{\max}}=\delta_{\x_{\max}}$ and $\nu_{\x_{\max};t}=\mu_{t}$ for every $t\in I\cup\{t_{\max}\}$, where $\x_{\max}=(q,t_{\max})$.

\begin{rem} 
When $I=(-\infty,0)$ and $T>0$, the extended metric flow pairs $(\X,(\mu_{t})_{t \leq 0 })$ are exactly the ones considered in \cite[Section~6]{heatkernel2}.
\end{rem}

Following \cite[Definition~3.24]{Ba20_compactness} one can introduce the concept of $H$-center of a metric flow. From this point of view, \cite[Proposition~3.12]{heatkernel2} is a special instance of \cite[Proposition~3.25]{Ba20_compactness}, being implied by the $H_{n}$-concentration of $\X$. Similarly, Proposition \ref{center-singular-time} and Corollary \ref{centers-sequence} can be rephrased in the present context as follows.

\begin{prop}[$H_n$-Centers of Metric Flow Pair]\label{centers-metric-flow} 
Consider the metric flow pair $(\X, (\mu_{t})_{t\in I})$ over $I$.
\begin{enumerate}
\item For every $t\in I$ there exists $\z\in \X_{t}$ such that
\begin{equation*}
\var_{t}(\mu_{t},\delta_{\z})\leq H_{n}(t_{\max}-t).
\end{equation*}
\item Given $t_{i}\nearrow t_{\max}$ there exists a sequence of points $(\z_{i})_{i\in \N}$, with $\z_{i}\in \X_{t_{i}}$, such that 
\begin{equation*}
\lim_{i\to \infty} d^{\X_{t_{i}}}_{W_{1}}(\mu_{t_{i}},\delta_{\z_{i}})=0.
\end{equation*}
\end{enumerate}
\end{prop}

Finally, we recall that if $(\X,(\mu_{t})_{t\in I'})$ is a general metric flow pair over an interval $I\subseteq \R$, given $t_{0}\in \R$ and $\lambda>0$ then $(\X^{-t_{0},\lambda},(\mu^{-t_{0},\lambda}_{t})_{t\in I'^{-t_{0},\lambda}})$ denotes its corresponding parabolic rescaling. For more about this concept see the Appendix \ref{parab-rescaling-mf-app} or \cite{Ba20_compactness}. 

%---------------------------------------------------------------------------------------------------------
\section{$\mathbb{F}$-Convergence of Ricci Shrinkers}\label{conv-result-section}

We are ready to prove Theorem \ref{scaling-compactness-thm-intro} from the introduction. As already mentioned, this result is analogous to \cite[Theorem~6.10]{heatkernel2}, and the main difference relies in the fact that in our case the final time of the metric flow pairs under consideration can be given by the \emph{singular} time of the flow. As we will see, this is not a major issue, since the proof of the theorem uses facts regarding general metric flow pairs.

For the relevant notation appearing in the statement of the theorem and its proof, recall Section~6.1 and 7.1 (and in particular Definitions 6.1 and 7.1) of Bamler \cite{Ba20_compactness}.

\begin{proof}[Proof of Theorem \ref{scaling-compactness-thm-intro}]
Let $(M_{k}^{n}, g_{k}(t))_{t<T_{k}}$, $k\in \N$, be a sequence of Ricci flows induced by Ricci shrinkers and defined over a maximal time-interval of the form $(-\infty,T_{k})$, where $0\leq T_{k}<\infty$. Given points $q_{k}\in M_{k}$, since $I=(-\infty,0)\subseteq (-\infty,T_{k})$ we can form the metric flow pairs $(\X^{k},(\mu^{k}_{t})_{t<0})$ corresponding to the pointed Ricci flows $(M_{k},g_{k}(t),(q_{k},0))_{t<0}$ as explained in the previous section. 

Furthermore, we can extend each of the above metric flow pair to the final time $t=0$, thus obtaining an extended metric flow pair $(\X^{k}, (\mu^{k}_{t})_{t\leq 0})$ fully defined over $(-\infty,0]$. 

Following the notation from \cite[Section~7.1]{Ba20_compactness}, we have $I^{k}=I^{\infty}=(-\infty,0]$ for every $k\in \N$ and $t_{\text{max}}=0\in I^{\infty}$. So we only need to verify that $(\X^{k},(\mu^{k}_{t})_{t\leq 0})\in \mathbb{F}_{(-\infty,0]}^{t_{\text{max}}}(H_{n},V,b,r)$ for some choice of $V\geq 0,r>0$ and $b\colon (0,1]\to (0,1]$. First of all, we have already proved that every metric flow $\X^{k}$ is $H_{n}$-concentrated. Second, we trivially have $(\X_{0},d_{0},\mu_{0})\in \mathbb{M}_{r}(V,b)$ for $V=0$ and $r>0$, $b\colon (0,1]\to (0,1]$ chosen arbitrarily, since $\X_{0}$ is a singleton. At this point we can apply the general compactness result from \cite[Theorem~7.6]{Ba20_compactness} and conclude. 
\end{proof}
 
\begin{rem}
In \cite[Theorem~6.10]{heatkernel2}, the parts regarding $\F^{\ast}_{(-\infty,0]}(H_{n})$ and the uniqueness of $\X^{\infty}$ are missing, but nevertheless they hold true since their theorem and ours derive from the same general result about metric flow pairs. 
\end{rem}

\begin{rem}
Note that, since $\X^{\infty}_{0}$ is a single point, $\var_{0}(\mu^{\infty}_{0})=0$ and therefore we also find $\var_{t}(\mu^{\infty}_{t})\leq H_{n}|t|$ for every $t\leq 0$ by \cite[Proposition~3.23]{Ba20_compactness}. 
\end{rem}

We end this section with a quite general consequence of $\F$-convergence; see Definition \ref{gromov-meas-conv} for the notion of pointed Gromov-measured convergence. 

\begin{prop}[Gromov-Measured Convergence]\label{conv-centers} 
Consider metric flow pairs $(\X^{k},(\mu^{k}_{t})_{t\leq 0}) \in \F^{\ast}_{(-\infty,0]}(H)$, for $k\in \N\cup\{\infty\}$ and some uniform $H>0$, such that 
\begin{equation*}
(\X^{k},(\mu^{k}_{t})_{t\leq 0}) \longrightarrow (\X^{\infty},(\mu^{\infty}_{t})_{t\leq 0})
\end{equation*} 
within a correspondence $\C$ on compact time-intervals. Suppose that the convergence is uniform at some time $t<0$, and let $\z_{k}\in \X^{k}_{t}$ be a sequence of $H$-centers of $\x^{k}_{0}\in \X^{k}_{0}$. Then there exist points $\y_{k} \in \X^{k}_{t}$ and $\y_{\infty}\in \X^{\infty}_{t}$ such that $\y_{k}\to \y_{\infty}$ strictly within $\C$ and $d^{k}_{t}(\z_{k},\y_{k})\leq r$ for some positive constant $r>0$. In particular, 
\begin{equation*}
(\X^{k}_{t},d^{k}_{t},\mu^{k}_{t},\y_{k})\longrightarrow (\X^{\infty}_{t},d^{\infty}_{t},\mu^{\infty}_{t},\y_{\infty})
\end{equation*}
in the pointed Gromov-measured sense.
\end{prop}

\begin{proof}
Since the convergence is uniform at time $t$, we have
\begin{equation*}
(\X^{k}_{t},d^{k}_{t},\mu^{k}_{t})\longrightarrow (\X^{\infty}_{t},d^{\infty}_{t},\mu^{\infty}_{t})
\end{equation*}
in the Gromov-Wasserstein $GW_{1}$-sense; see \cite[Remark~5.6]{Ba20_compactness}. Now, if $\z_{\infty}\in \X^{\infty}_{t}$ is an $H$-center of $\x^{\infty}_{0} \in \X^{\infty}_{0}$ and $r=\sqrt{10H|t|}$, we have $\mu_{t}^{k}(\overline{B}(\z_{k},r))\geq 9/10$ for every $k\in \N\cup\{\infty\}$, thanks to \cite[Lemma~3.26]{Ba20_compactness}. At this point one can proceed as in the proof of Proposition \ref{GW-pGm} to conclude.
\end{proof}

\section{Regularity Results}\label{reg-results-section}

In this section we prove Theorem \ref{good-convergence-intro} as well as the additional structure theory for tangent flows of the limit metric flow mentioned in the introduction. The idea is to proceed as in \cite[Sections~2.3 and 19]{Ba20_structure}, that is reducing the convergence of a sequence $(\X^{k},(\mu^{k}_{t})_{t\leq 0})$ of metric flow pairs as in Section \ref{conv-result-section} to the one of metric flow pairs obtained from Ricci flows defined up to the final time, and so fitting into the setting of Li and Wang's results \cite[Section~6]{heatkernel2}. We start with some lemmas that could be of independent interest. 

\subsection{Some Preliminary Lemmas}\label{some-lemmas-section}

In the following, we work with general metric flow pairs, not necessarily induced by Ricci shrinkers. 
The first result allows us to extend a given $\F$-convergence within a correspondence (as defined in \cite[Definition~6.1]{Ba20_compactness}) on $(-\infty,0)$ to $(-\infty,0]$.   

\begin{lemma}[Extending $\F$-Convergence to the Final Time]\label{extension-conv} 
Let $(\X^{i},(\mu^{i}_{t})_{t\leq 0})$, $i\in \N\cup\{\infty\}$, be metric flow pairs fully defined over $(-\infty,0]$. Suppose there exists a correspondence $\C$ over $(-\infty,0)$ between the metric flows $\X^{i}_{t<0}$, $i\in \N\cup\{\infty\}$, such that
\begin{equation*}
(\X^{i}_{t< 0},(\mu^{i}_{t})_{t<0}) \longrightarrow (\X^{\infty}_{t<0}, (\mu^{\infty}_{t})_{t<0})
\end{equation*}
within $\C$ on compact time-intervals of $(-\infty,0)$. Then $\C$ can be extended to a correspondence $\widetilde{\C}$ over $(-\infty,0]$ such that 
\begin{equation*}
(\X^{i},(\mu^{i}_{t})_{t \leq 0}) \longrightarrow (\X^{\infty}, (\mu^{\infty}_{t})_{t \leq 0})
\end{equation*}
within $\widetilde{\C}$ on compact time-intervals of $(-\infty,0]$.
\end{lemma}

\begin{proof}
Let $\C=\left((Z_{t},d^{Z}_{t})_{t<0}, (\varphi_{t}^{i})_{t\in I''^{,i},\; i\in \N\cup \{\infty\}}\right)$ be the given correspondence. 
Obviously, $\C$ can be extended to a correspondence $\widetilde{\C}$ over $(-\infty,0]$ defining $(Z_{0},d^{Z}_{0})$ 
arbitrarily and leaving all the other data unchanged.

Now we prove the desired convergence. Fix a compact time-interval $I^{*}\subseteq (-\infty,0]$. If $0\notin I^{*}$ the claim follows directly from the assumptions, and therefore we can suppose $I^{*}=[a,0]$. Let $\varepsilon \in (0,\sqrt{|a|}/2)$ and define $I^{*}_{\delta}=[a,-\delta]$ with $\delta=3\varepsilon^{2}$. Since we have convergence on $I^{*}_{\delta}$, there exists $i_{0}\in \N$ such that
\begin{equation*}
d_{\mathbb{F}}^{\C, I^{*}_{\delta}}\left((\X^{i}_{I^{*}_{\delta}},(\mu^{i}_{t})_{t\in I^{*}_{\delta}}), 
(\X^{\infty}_{I^{*}_{\delta}}, (\mu^{\infty}_{t})_{t\in I^{*}_{\delta}})\right)<\varepsilon
\end{equation*} 
for every $i\geq i_{0}$. Recalling the definition of $\F$-distance within a correspondence, for every $i\geq i_{0}$ there exist subsets $E_{i}\subseteq I^{*}_{\delta}$ with
\begin{equation*}
I^{*}_{\delta}\setminus E_{i}\subseteq I''^{,i}\cap I''^{,\infty}\cap I^{*}_{\delta},
\end{equation*}
and a family of couplings $(q^{i}_{t})_{t\in I^{*}_{\delta}\setminus E_{i}}$ between $\mu_{t}^{i}$ and $\mu_{t}^{\infty}$ such that
\begin{equation*}
|E_{i}|\leq \varepsilon^{2}
\end{equation*}
and
\begin{equation}\label{int-est}
\int_{\X^{i}_{t}\times \X^{\infty}_{t}}d_{W_{1}}^{Z_{s}^{i}}\left((\varphi_{t}^{i})_{\ast}\nu^{i}_{\x;s},(\varphi_{t}^{\infty})_{\ast}\nu_{\y;s}^{\infty}\right)dq^{i}_{t}(\x,\y) \leq \varepsilon
\end{equation}
for all $s,t\in I^{*}_{\delta}\setminus E_{i}$ with $s\leq t$. Now, defining $\widetilde{E}_{i}=E_{i}\cup (-\delta,0]$ we have $|\widetilde{E}_{i}|\leq \varepsilon^{2}+\delta = 4\varepsilon^2$ and $I^{*}\setminus \widetilde{E}_{i}=I^{*}_{\delta}\setminus E_{i}$; therefore, \eqref{int-est} still holds for every $s,t\in I^{*}\setminus \widetilde{E}_{i}$ with $s\leq t$. By the choice of $\delta$ we therefore deduce that 
\begin{equation*}
d_{\mathbb{F}}^{\widetilde{\C}, I^{*}}\left((\X^{i}_{I^{*}_{\phantom{\delta}}},(\mu^{i}_{t})_{t\in I^{*}}), 
(\X^{\infty}_{I^{*}}, (\mu^{\infty}_{t})_{t\in I^{*}})\right) \leq 2\varepsilon
\end{equation*} %%\phantom{\delta} is used to get same bracket size as above
for every $i\geq i_{0}$, and we are done.
\end{proof}

For the next lemmas, consider a given general metric flow pair $(\X,(\mu_{t})_{t\leq 0})\in \F_{(-\infty,0]}^{\ast}$, fully defined over $(-\infty,0]$. Concretely, the reader may think of one of the extended metric flow pairs $(\X^{k},(\mu^{k}_{t})_{t\leq 0})$ (for $k\in \N$ fixed) constructed in Section \ref{shrinker-mf-section}. Intuitively, we want to approximate this metric flow pair with a sequence of metric flow pairs corresponding to Ricci flows existing for a slightly longer time, obtained via a time-shift, but we need to work with a different conjugate heat flow than simply the time-shifted one. We first prove a technical lemma.

\begin{lemma}[$W_1$-Wasserstein-Distance]\label{w1-distance-points} 
Let $(\X,(\mu_{t})_{t\leq 0})\in \F_{(-\infty,0]}^{\ast}$ be a metric flow pair fully defined over $(-\infty,0]$ and let $t_{i}\nearrow 0$ be a sequence of negative numbers and suppose there exist points $\z_{i}\in \X_{t_{i}}$ such that 
\begin{equation*}
\lim_{i\to \infty} d^{\X_{t_{i}}}_{W_{1}}(\mu_{t_{i}},\delta_{\z_{i}})=0.
\end{equation*}
Then for every compact time-interval $I^{*}\subseteq (-\infty,0)$ we have
\begin{equation*}
\lim_{i\to \infty}\sup_{t\in I^{*}}d^{\X_{t}}_{W_{1}}(\mu_{t},\nu_{\z_{i};t})=0.
\end{equation*}
\end{lemma}

\begin{proof}
Let $I^{*}=[a,b]$, $b<0$, and take $i$ large enough such that $b<t_{i}$. From \cite[Proposition~3.16]{Ba20_compactness} we know that the function 
\begin{equation*}
t\mapsto d^{\X_{t}}_{W_{1}}(\mu_{t},\nu_{\z_{i};t})
\end{equation*}
is non-decreasing for $t\in (-\infty,t_{i}]$. Since $I^{*}\subseteq (-\infty,t_{i})$, it follows
\begin{equation*}
\sup_{t\in I^{*}}d^{\X_{t}}_{W_{1}}(\mu_{t},\nu_{\z_{i};t})
\leq d^{\X_{t_{i}}}_{W_{1}}(\mu_{t_{i}},\nu_{\z_{i};t_{i}})=d^{\X_{t_{i}}}_{W_{1}}(\mu_{t_{i}},\nu_{(z_{i},t_{i});t_{i}})
=d^{\X_{t_{i}}}_{W_{1}}(\mu_{t_{i}},\delta_{\z_{i}})\longrightarrow 0
\end{equation*} 
as $i\to \infty$, and we are done.
\end{proof}

Now, given $t_{i}\nearrow 0$ and $\z_{i}\in \X_{t_{i}}$ as in Lemma \ref{w1-distance-points}, consider, for $i\in \N$ fixed, the metric flow pair 
$(\X_{t<t_{i}}, (\nu_{\z_{i};t})_{t<t_{i}})$ over $(-\infty,t_{i})$. Applying a time-shift $t\mapsto t+t_{i}$, we obtain a new metric flow pair $(\X^{i}_{<0}, (\eta^{i}_{t})_{t<0})$ over $(-\infty,0)$, where $\X^{i}_{t}=\X_{t+t_{i}}$ and $\eta^{i}_{t}=\nu_{\z_{i};t+t_{i}}$ for every $t<0$. The next lemma shows that these metric flow pairs approximate $(\X,(\mu_{t})_{t\leq 0})$ on compact time-intervals.

\begin{lemma}[Convergence of Approximating Metric Flow Pairs]\label{conv-compact-interval-rescaling} 
Suppose that $\X$ is $H$-concentrated. If $I^{*}\subseteq (-\infty,0)$ is a compact time-interval then, in the metric space $(\mathbb{F}_{I^{*}},d_{\mathbb{F}})$, we have
\begin{equation*}
\lim_{i\to \infty}d_{\mathbb{F}}\left((\X_{I^{\ast}},(\mu_{t})_{t\in I^{\ast}}),(\X^{i}_{I^{\ast}},(\eta^{i}_{t})_{t\in I^{\ast}})\right)=0.
\end{equation*}
\end{lemma}

\begin{proof}
It is useful to introduce the metric flow pair $(\X^{i}_{<0},(\mu^{i}_{t})_{t<0})$ over $(-\infty,0)$, obtained time-shifting $(\X, (\mu_{t})_{t<0})$ by $t\mapsto t+t_{i}$ and restricting to negative times; explicitly, $\X^{i}_{t}=\X_{t+t_{i}}$ and $\mu^{i}_{t}=\mu_{t+t_{i}}$ for every $t<0$. In particular, in this sequence of metric flow pairs, the metric flow $\X^{i}_{<0}$ is the same as in our approximating sequence, but the conjugate heat flow is different.

Given $I^{\ast}$, we work in the metric space $(\mathbb{F}_{I^{*}},d_{\mathbb{F}})$; note that since the above metric flow pairs are fully defined over $(-\infty,0)$ their restrictions to $I^{*}$ define elements in this space. By the triangle inequality, we have
\begin{align*}
d_{\mathbb{F}}\left((\X_{I^{\ast}},(\mu_{t})_{t\in I^{*}}),(\X^{i}_{I^{\ast}},(\eta^{i}_{t})_{t\in I^{*}})\right)
&\leq d_{\mathbb{F}}\left((\X_{I^{\ast}},(\mu_{t})_{t\in I^{*}}),(\X^{i}_{I^{\ast}},(\mu^{i}_{t})_{t\in I^{*}})\right) \\
&\quad+ d_{\mathbb{F}}\left((\X^{i}_{I^{\ast}},(\mu^{i}_{t})_{t\in I^{*}}),(\X^{i}_{I^{\ast}},(\eta^{i}_{t})_{t\in I^{*}})\right).
\end{align*}
We estimate the two terms on the right hand side separately. Regarding the first one, we get 
\begin{equation*}
\lim_{i\to \infty}d_{\mathbb{F}}\left((\X_{I^{\ast}},(\mu_{t})_{t\in I^{*}}),(\X_{I^{\ast}}^{i},(\mu^{i}_{t})_{t\in I^{*}})\right)=0
\end{equation*}
since time-shifting is continuous with respect to $\F$-distance, see \cite[Proposition~6.1]{CMZ} for a detailed proof of this fact.
The second term can be estimated constructing an explicit correspondence as follows. 
Let 
\begin{equation*}
\C_{i}=\left((Z_{t},d_{t}^{Z})_{t\in I^{*}}, (\varphi_{t})_{t\in I^{*}}, (\psi_{t})_{t\in I^{*}}\right)
\end{equation*}
be the correspondence over $I^{\ast}$ between $\X^{i}_{I^{\ast}}$ and itself obtained setting 
$Z_{t}=\X^{i}_{t}=\X_{t+t_{i}}$ and $\varphi_{t}=\psi_{t}$ equal to the identity map. 
Without loss of generality we can assume $|t_{i}|<1$ for every $i$. If $I^{*}=[a,b]$, denoting $I^{\ast\ast}=[a-1,b-1]$
and untangling the above definitions of the heat flows, we have
\begin{equation*}
d^{\X^{i}_{t}}_{W_{1}}(\mu_{t}^{i},\eta^{i}_{t})=d^{\X_{t+t_{i}}}_{W_{1}}(\mu_{t+t_{i}}, \nu_{\z_{i};t+t_{i}})
\leq \sup_{\tau \in I^{\ast\ast} } d^{\X_{\tau}}_{W_{1}}(\mu_{\tau}, \nu_{\z_{i};\tau})
\end{equation*}
for every $t\in I^{*}$, noting that $\tau<t_{i}$ if $\tau \in I^{\ast\ast}$.
Applying \cite[Lemma~5.14]{Ba20_compactness}, it follows that
\begin{equation*}
d_{\mathbb{F}}^{\C_{i}}\left((\X^{i}_{I^{\ast}}, (\mu^{i}_{t})_{t\in I^{*}}), (\X^{i}_{I^{\ast}},(\eta^{i}_{t})_{t\in I^{*}})\right) \leq  
\sup_{\tau \in I^{\ast\ast} } d^{\X_{\tau}}_{W_{1}}(\mu_{\tau}, \nu_{\z_{i};\tau})\end{equation*}
and, by definition of $\mathbb{F}$-distance,
\begin{equation*}
d_{\mathbb{F}}\left((\X^{i}_{I^{\ast}}, (\mu^{i}_{t})_{t\in I^{*}}), (\X^{i}_{I^{\ast}},(\eta^{i}_{t})_{t\in I^{*}})\right) \leq 
 \sup_{\tau \in I^{\ast\ast} } d^{\X_{\tau}}_{W_{1}}(\mu_{\tau}, \nu_{\z_{i};\tau}).
 \end{equation*}
At this point, Lemma \ref{w1-distance-points} implies the claim. 
\end{proof}

Observe that $\lim_{t\nearrow 0}\var_{t}(\eta^{i}_{t})=0$, and therefore the metric flow pair $(\X^{i}_{<0}, (\eta^{i}_{t})_{t<0})$
can be extended to a metric flow pair fully defined over $(-\infty,0]$, where $\X^{i}_{0}=\{\z_{i}\}$ and $\eta^{i}_{0}=\delta_{\z_{i}}$;
we denote it by $(\X^{i},(\eta^{i}_{t})_{t\leq 0})$. Lemma \ref{extension-conv} will be applied to this setting. 

\subsection{Proof of Theorem \ref{good-convergence-intro} and Structure of Tangent Flows}

Let $(\X^{k}, (\mu^{k}_{t})_{t\leq 0})$, $k\in \N$, be a sequence of (extended) metric flow pairs fully defined over $(-\infty,0]$, as in Section \ref{conv-result-section}. After passing to a subsequence, $(\X^{k}, (\mu^{k}_{t})_{t\leq 0})$ satisfies Theorem \ref{scaling-compactness-thm-intro} for some limit metric flow pair $(\X^{\infty}, (\mu_{t}^{\infty})_{t\leq 0})\in \F^{\ast}_{(-\infty,0]}(H_{n})$ and some correspondence $\C$. 

Using the above results, we will now construct an approximating sequence converging to the same limit. In order to do so, for $k\in \N$ fixed and given a sequence of negative times $t_{i}\nearrow 0$, let $\z_{k,i}\in \X_{t_{i}}^{k}$ be the $H_n$-centers given by Proposition \ref{centers-rescaling}. Following the notation introduced in Section \ref{some-lemmas-section}, consider the double sequence of metric flow pairs $(\X^{k,i}, (\eta^{k,i}_{t})_{t\leq 0})$, for $i,k\in \N$. 

By construction, each of the above elements is obtained applying a time-shift $t\mapsto t+t_{i}$ to the metric flow pair $(\X^{k}, (\nu^{k}_{\z_{k,i};t})_{t<t_{i}})$ (and extending to $t=0$), and therefore is (the extension of) the metric flow pair corresponding to the pointed Ricci flow $(M_{k}, g_{k}(t+t_{i}), (z_{k,i}, 0))_{t<0}$, where $\z_{k,i}=(z_{k,i},t_{i})\in \X_{t_{i}}^{k}$. Note that the singular time of this Ricci flow is $T_{k}-t_{i}>0$, and therefore the construction fits into the setting of 
\cite[Section~6]{heatkernel2}. 

Recall that Lemma \ref{conv-compact-interval-rescaling} gives, for any $k\in \N$ fixed,
\begin{equation}\label{eq.lemma5.3summary}
\lim_{i\to \infty}d_{\mathbb{F}}\left((\X^{k}_{I^{\ast}},(\mu^{k}_{t})_{t\in I^{*}}), (\X^{k,i}_{I^{\ast}}, (\eta^{k,i}_{t})_{t\in I^{*}}) \right)=0
\end{equation}
in the metric space $(\mathbb{F}_{I^{*}},d_{\mathbb{F}})$, for every compact time-interval $I^{*}\subseteq (-\infty,0)$.

\begin{lemma}[Construction of Approximating Sequence with Same Limit]\label{double-conv-lemma} 
There exists an increasing sequence of indices $(i_{k})_{k\in \N}$ such that, for every compact time-interval $I^{*}\subseteq (-\infty,0)$, we have
\begin{equation*}
\lim_{i\to \infty} d_{\mathbb{F}}\left((\X^{k,i_{k}}_{I^{\ast}}, (\eta^{k,i_{k}}_{t})_{t\in I^{*}}), 
(\X^{\infty}_{I^{\ast}},(\mu^{\infty}_{t})_{t\in I^{*}}) \right)=0
\end{equation*}
in the metric space $(\mathbb{F}_{I^{*}},d_{\mathbb{F}})$.
\end{lemma}

\begin{proof}
First of all, let $I_{k}=[-k,-1/k]$ for $k\in \N$; in this way we obtain an exhaustion of $(-\infty,0)$ by compact subsets. By an inductive argument, from \eqref{eq.lemma5.3summary} it is easy to construct a sequence of indices $(i_{k})_{k\in \N}$ such that $i_{k}<i_{k+1}$ and
\begin{equation*}
d_{\mathbb{F}}\left ((\X^{k,i_{k}}_{I_{k}},(\eta^{k,i_{k}}_{t})_{t\in I_{k}}), (\X^{k}_{I_{k}},(\mu^{k}_{t})_{t\in I_{k}})\right)\leq 1/k
\end{equation*}
for every $k\in \N$. 

Now, we prove the claimed convergence. Fix $I^{*}\subseteq (-\infty,0)$ a compact interval and let $\varepsilon>0$. Then there exists $k_{0}\in \N$ such that $I^{*}\subseteq I_{k}$ and $1/k\leq \varepsilon/2$ for every $k\geq k_{0}$. It follows that
\begin{equation*}
d_{\mathbb{F}}\left ((\X^{k,i_{k}}_{I^{\ast}},(\eta^{k,i_{k}}_{t})_{t\in I^{\ast}}), (\X^{k}_{I^{\ast}},(\mu^{k}_{t})_{t\in I^{*}})\right)
\leq \varepsilon/2,
\end{equation*} 
as the $\mathbb{F}$-distance is monotone under restriction. Now, since up to a subsequence we have
\begin{equation*}
(\X^{k},(\mu^{k}_{t})_{t\leq 0})\longrightarrow (\X^{\infty},(\mu^{k}_{t})_{t \leq 0})
\end{equation*}
within $\C$ on compact time-intervals of $(-\infty,0]$, there exists $k_{1}\geq k_{0}$ such that
\begin{equation*}
d_{\mathbb{F}}\left((\X^{k}_{I^{\ast}},(\mu^{k}_{t})_{t\in I^{*}}), (\X^{\infty}_{I^{\ast}}, (\mu^{\infty}_{t})_{t\in I^{*}}) \right)\leq d_{\mathbb{F}}^{\C, I^{*}}\left((\X^{k}_{I^{\ast}},(\mu^{k}_{t})_{t\in I^{*}}), (\X^{\infty}_{I^{\ast}}, (\mu^{\infty}_{t})_{t\in I^{*}}) \right)\leq \varepsilon/2
\end{equation*}
for every $k\geq k_{1}$. Therefore,
\begin{equation*}
d_{\mathbb{F}}\left((\X^{k,i_{k}}_{I^{\ast}}, (\eta^{k,i_{k}}_{t})_{t\in I^{*}}), (\X^{\infty}_{I^{\ast}}, (\mu^{\infty}_{t})_{t\in I^{*}}) \right) \leq \varepsilon
\end{equation*}
for every $k\geq k_{1}$ by the triangle inequality for $d_{\F}$, and we are done. 
\end{proof}

\begin{rem}\label{double-conv-rem} 
Given any sequence of numbers $(m_{k})_{k\in \N}$, that we think of as associated with $(\X^{k}, (\mu^{k}_{t})_{t \leq 0})$, it is easy to see that the sequence of indices $(i_{k})_{k\in \N}$ constructed in Lemma \ref{double-conv-lemma} can be chosen to further satisfy $|t_{i_{k}}m_{k}|\leq 1$ for every $k\in \N$. This will prove useful in some later applications.
\end{rem}

Next, we show that the approximating sequence $(\X^{k,i_{k}}, (\eta^{k,i_{k}}_{t})_{t\leq 0})$ still converges to the same limit $(\X^{\infty}, (\mu_{t}^{\infty})_{t\leq 0})$ in the following stronger sense.

\begin{prop}[Convergence of Approximating Sequence]\label{prop-convergence} 
After passing to a subsequence, there exists a correspondence $\overline{\C}$ over $(-\infty,0]$ between the metric flows $\X^{k,i_{k}}$, $k\in \N$ and $\X^{\infty}$
such that
\begin{equation*}
(\X^{k,i_{k}}, (\eta_{t}^{k,i_{k}})_{t\leq 0}) \longrightarrow (\X^{\infty}, (\mu_{t}^{\infty})_{t\leq 0})
\end{equation*}
within $\overline{\C}$ on compact time-intervals of $(-\infty, 0]$. Moreover, the convergence is uniform over every compact subset $J\subseteq (-\infty,0]$ that contains only times at which
$\X^{\infty}$ is continuous. 
\begin{proof}
Thanks to Lemma \ref{double-conv-lemma}, we already know that for every compact time-interval $I^{*}\subseteq (-\infty,0)$ the sequence $(\X^{k,i_{k}}_{I^{\ast}},(\eta^{k,i_{k}}_{t})_{t \in I^{\ast}})$ converges towards $(\X^{\infty}_{I^{\ast}},(\mu^{\infty}_{t})_{t\in I^{\ast}})$ in $(\F_{I^{*}},d_{\F})$. Thus, applying first \cite[Theorem~6.6]{Ba20_compactness} and then Lemma \ref{extension-conv}, we obtain a correspondence $\C^{\ast}$ over $(-\infty,0]$ between the metric flows $\X^{k,i_{k}}, k\in \N$ and $\X^{\infty}$ such that 
\begin{equation}\label{5.6firstconvergence}
(\X^{k,i_{k}},(\eta_{t}^{k,i_{k}})_{t\leq 0}) \longrightarrow (\X^{\infty},(\mu^{\infty}_{t})_{t \leq 0})
\end{equation}
within $\C^{\ast}$ on compact time-intervals of $(-\infty,0]$.  

Now, since $(\X^{k,i_{k}},(\eta^{k,i_{k}}_{t})_{t\leq 0})\in \F^{\ast}_{(-\infty,0]}(H_{n})$ for every $k\in \N$, we can apply Theorem \ref{scaling-compactness-thm-intro} to this sequence. Therefore, after passing to a subsequence, there exists an $H_{n}$-concentrated and future continuous metric flow pair $(\Y^{\infty}, (\eta^{\infty}_{t})_{t\leq 0})$, representing a class in $\F^{\ast}_{(-\infty,0]}(H_{n})$, and a correspondence $\C^{\ast\ast}$ over $(-\infty,0]$ between the metric flows $\X^{k,i_{k}}$ and $\Y^{\infty}$ such that
\begin{equation}\label{conv-k-ik}
(\X^{k,i_{k}},(\eta^{k,i_{k}}_{t})_{t\leq 0}) \longrightarrow (\Y^{\infty},(\eta^{\infty}_{t})_{t \leq 0})
\end{equation}
within $\C^{\ast\ast}$ on compact time-intervals of $(-\infty,0]$. Since \eqref{5.6firstconvergence} still holds (within the correspondence $\C^{\ast}$) after taking the above subsequence, the uniqueness part of Theorem \ref{scaling-compactness-thm-intro} gives an isometry over $(-\infty,0]$ between the metric flow pairs $(\X^{\infty}, (\mu^{\infty}_{t})_{t\leq 0})$ and $(\Y^{\infty},(\eta^{\infty}_{t})_{t\leq 0})$, and therefore we can exchange them in \eqref{conv-k-ik}, obtaining the claim for an appropriate correspondence $\overline{\C}$.
\end{proof}
\end{prop}

\begin{rem} 
Note that one can always extend $\overline{\C}$ to be fully defined over $J=\{0\}$. In this way, since the final time-slices are singletons, the claim of Proposition \ref{prop-convergence} still holds, and moreover the convergence is time-wise at $t=0$.  
\end{rem}

At this point, Theorem \ref{good-convergence-intro} follows trivially.

\begin{proof}[Proof of Theorem \ref{good-convergence-intro}]
Thanks to Proposition \ref{prop-convergence} we are exactly in the setting of \cite[Section~6]{heatkernel2}, and we can therefore apply \cite[Theorem~6.12]{heatkernel2}, which is itself an adaptation of Bamler's results from \cite[Theorems~2.3--2.6]{Ba20_structure} to the case of Ricci flows induced by Ricci shrinkers.
\end{proof}

In the same way, from \cite[Theorem~6.12]{heatkernel2} follows also the next proposition regarding the convergence of the approximating sequence on the regular part $\reg^{\infty}$ of the limit. 

\begin{prop}[Smooth Convergence on Regular Part]\label{good-convergence-prop} 
Under the same assumptions as in Theorem \ref{good-convergence-intro}, the convergence of the approximating sequence $(\X^{k,i_{k}},(\eta^{k,i_{k}}_{t})_{t\leq 0})$ towards $(\X^{\infty},(\mu^{\infty}_{t})_{t\leq 0})$ obtained in Proposition \ref{prop-convergence} is smooth on $\reg^{\infty}$, meaning the following. There exist an increasing sequence $(U_{k})_{k\in \N}$ of open subsets of $\reg^{\infty}$ with $\bigcup_{k\in \N}U_{k}=\reg^{\infty}$, open subsets $V_{k}\subseteq M_{k}\times (-\infty,0)$, time preserving diffeomorphisms $\phi_{k}\colon U_{k}\to V_{k}$ and a sequence $\varepsilon_{k}\to 0$ such that the following conditions hold.
\begin{enumerate}
\item[(a)] We have
\begin{align*}
\| \phi_{k}^{\ast}\overline{g}^{k}-g^{\infty}\|_{C^{[\varepsilon_{k}^{-1}]}(U_{k})} &\leq \varepsilon_{k},\\
\|\phi_{k}^{\ast}\partial^{k}_{t}-\partial^{\infty}_{\mathfrak{t}}\|_{C^{[\varepsilon_{k}^{-1}]}(U_{k})}& \leq \varepsilon_{k},\\
\|\phi_{k}^{\ast}\overline{w}^{k}-w^{\infty}\|_{C^{[\varepsilon_{k}^{-1}]}(U_{k})}& \leq \varepsilon_{k},
\end{align*}
where $(\reg^{k},\mathfrak{t}^{k},\partial^{k}_{\mathfrak{t}},g^{k})$ is the Ricci flow spacetime associated with 
the Ricci flow $\left(M_{k},\overline{g}_{k}(t)\right)_{t < 0}$ given by $\overline{g}_{k}(t)=g_{k}(t+t_{i_{k}})$, while $\overline{w}^{k}\colon \reg^{k} \to \R$ and $w^{\infty}\colon \reg^{\infty}\to \R$ are the probability densities defined by 
\begin{equation*}
d\eta^{k,i_{k}}_{t}=\overline{w}^{k} dV_{\overline{g}_{k}(t)}, \quad \quad d\mu^{\infty}_{t}=w^{\infty}dV_{g^{\infty}_{t}},
\end{equation*}
respectively.
\item[(b)] Let $\y_{\infty}\in \reg^{\infty}$ and $\y_{k}\in M_{k}\times (-\infty,0)$. Then $\y_{k}\to \y_{\infty}$ within $\overline{\C}$ if and only if $\y_{k}\in V_{k}$ for $k$ large and $\phi_{k}^{-1}(\y_{k})\to \y_{\infty}$ in $\reg^{\infty}$.
\item[(c)] If the convergence is uniform at some time $t\in (-\infty,0)$, then for any compact subset $K\subseteq \reg^{\infty}_{t}$ we have
\begin{equation*}
\sup_{\x\in K\cap U_{k}} d^{Z}_{t}(\varphi_{t}^{k}(\phi_{k}(\x)),\varphi^{\infty}_{t}(\x))\longrightarrow 0,
\end{equation*}
where $\overline{\C}=\left((Z_{t},d_{t}^{Z})_{t\leq 0}, (\varphi^{k}_{t})_{t\in I''^{,k}, k\in \N\cup\{\infty\}}\right)$.
\end{enumerate}
\end{prop} 

Before continuing, we make some remarks about Theorem \ref{good-convergence-intro} and Proposition \ref{good-convergence-prop}.

\begin{rem}\label{good-convergence-remark} \begin{enumerate}
\item From the third point in the statement of Theorem \ref{good-convergence-intro}, it follows that each $\reg^{\infty}_{t}$ is connected, and therefore $\reg$ is connected too, being $\mathfrak{t}$ a submersion.
\item By definition, $\overline{w}^{k}(x,t)=H_{k}(z_{i_{k}},t_{i_{k}};x,t+t_{i_{k}})$ for every $x\in M_{k}$ and $t<0$, where $H_{k}$ is the (standard) heat kernel of $(M_{k},g_{k}(t))_{t<T_{k}}$. 
\item From \cite[Theorem~9.12]{Ba20_compactness} one has $w^{\infty}(\y)=K(\x_{0}; \y)$ for every $\y \in \reg^{\infty}$, where 
\begin{equation*}
K \colon \{(\x,\y)\in \X^{\infty}\times \reg^{\infty}\colon \mathfrak{t}(\x)> \mathfrak{t}(\y)\}\longrightarrow \R
\end{equation*}
is the heat kernel of $\X^{\infty}$, i.e.~$d\nu^{\infty}_{\x; s}=K(\x; \cdot)dV_{{g}^{\infty}_{s}}$ on $\reg_{s}^{\infty}$ for all $\x\in \X^{\infty}$. Moreover, $K$ is strictly positive, $K(\x; \cdot)\colon \reg^{\infty}_{< t}\to \R$ is smooth and solves the conjugate heat equation. 
\item In general, if $(\M,\mathfrak{t},\partial_{\mathfrak{t}},g)$ is a Ricci flow spacetime, the $C^{m}$-norm of a tensor field $A \in \Gamma^{h}_{k}(\ker d\mathfrak{t})$ over a subset $\Omega\subseteq \M$ is defined by 
\begin{equation*}
\|A\|_{C^{m}(\Omega)}=\sum_{r+s\leq m}\|\nabla^{r}\partial_{\mathfrak{t}}^{s}A\|_{L^{\infty}(\Omega)} + \|\partial_{\mathfrak{t}}^{s}\nabla^{r} A\|_{L^{\infty}(\Omega)},
\end{equation*}
where $\partial_{\mathfrak{t}}^{s}$ stands for the composition of the Lie derivative $\mathcal{L}_{\partial_{\mathfrak{t}}}$ with itself $s$ times and $\nabla$ is the covariant derivative induced by $g_{t}$ on the time-slice $\M_{t}$. In this context, if $A$ is a function then $\nabla A$ denotes its differential $d A$, and not its gradient vector field. 
\end{enumerate}
\end{rem}

From the smooth convergence on $\reg^{\infty}$ follows smooth convergence on every time-slice. Indeed, let $t<0$ and consider $\reg^{\infty}_{t}\subseteq \reg^{\infty}$. Defining $W_{k}=U_{k}\cap \reg_{t}^{\infty}$ and $\Phi_{k}\colon W_{k}\to M_{k}$ as the composition $\Phi_{k}=\pi_{1}\circ \phi_{k}\vert_{W_{k}}$, where $\pi_{1}\colon M_{k}\times (-\infty,0)\to M_{k}$ is the projection on the spatial component, then $(W_{k})_{k\in \N}$ is an exhaustion of $\reg_{t}^{\infty}$ by open subsets while each $\Phi_{k}$ is a smooth embedding, and it is easy to verify that 
\begin{align*}
\| \Phi_{k}^{\ast}\overline{g}^{k}(t)-g_{t}^{\infty}\|_{C^{[\varepsilon_{k}^{-1}]}(W_{k})} &\leq \varepsilon_{k},\\
\|\Phi_{k}^{\ast}\overline{w}^{k}_{t}-w^{\infty}_{t}\|_{C^{[\varepsilon_{k}^{-1}]}(W_{k})}& \leq \varepsilon_{k},
\end{align*} 
where $\overline{w}^{k}_{t}=\overline{w}^{k}(\cdot,t)$ and $w^{\infty}_{t}=w^{\infty}\vert_{\reg^{\infty}_{t}}$, and the norms are taken with respect to the metric $g^{\infty}_{t}$ at fixed time. 

It is possible to be more precise about the structure of tangent flows of $\X^{\infty}$.

\begin{prop}[Structure Theory for Tangent Flows]\label{metric-soliton-structure} 
Let $(\X^{\infty},(\mu^{\infty}_{t})_{t\leq 0})$ be as above and consider a tangent flow $(\X, (\nu_{\x;t})_{t \leq 0})$ of $\X^{\infty}$ at some point 
$\x_{\infty}$.
\begin{enumerate}
\item $(\X, (\nu_{\x;t})_{t \leq 0})$ is again the limit of a sequence of metric flow pairs induced by Ricci shrinkers and therefore it satisfies all the properties seen for $\X^{\infty}$. In particular, we have a regular-singular decomposition $\X_{<0}=\reg\sqcup \s $, where $\reg$ has the structure of a Ricci flow spacetime $(\reg,\mathfrak{t},\partial_{\mathfrak{t}},g)$ over $(-\infty,0)$.
\item $\X$ is the Gaussian soliton if and only if $\x_{\infty}\in \reg^{\infty}$.
\item Writing $\tau=-t$ and $d\nu_{\x_{\infty}}=K(\x;\cdot)dV_{g}=(4\pi\tau)^{-n/2}e^{-F}dV_{g}$ on $\reg$, then
\begin{equation}\label{metric-soliton-id}
\Ric + \nabla^{2}F - \frac{1}{2\tau}g=0, \quad \quad \tau(S+|\nabla F|^{2})+F=W, \quad \quad \partial_{\mathfrak{t}}F=|\nabla F|^{2},
\end{equation}
where $W$ is a constant.
\item There exist a singular space $(X,d)$ of dimension $n$, a probability measure $\mu$ on $X$ and an identification $\X_{<0}=X\times (-\infty,0)$ such that the following hold for every time $t<0$.
\begin{enumerate}
\item $(\X_{t},d_{t})=(X\times\{t\}, \sqrt{|t|}d)$.
\item $\reg=\reg_{X}\times (-\infty,0)$ and $\partial_{\mathfrak{t}}-\nabla f$ corresponds to the standard vector field on the second factor, i.e.~$\partial/\partial t$.
\item $(\reg_{t},g_{t})=(\reg_{X}\times\{t\}, |t|g_{X})$.
\item $\nu_{\x;t}=\mu$.
\end{enumerate}
\item There exists a unique family of probability measures $(\nu_{x;t})_{x\in X, t >0}$ on $X$ such that the tuple
\begin{equation*}
(X,d,\mu,(\nu_{x;t})_{x\in X,t > 0})
\end{equation*}
is a model for $(\X_{<0}, (\nu_{\x;t})_{t<0})$ corresponding to the same identification. Therefore, $(X,d,\mu)$ can be identified with $(\X_{<0}, (\nu_{\x;t})_{t<0})$ up to flow isometry.
\item $\mu(\s_{X})=0$ and $d\mu=(4\pi)^{-n/2}e^{-f_{X}}dV_{g_{X}}$ on $\reg_{X}$, where $f_{X}=F_{-1}\in C^{\infty}(\reg_{X})$ satisfies the first two soliton identities in \eqref{metric-soliton-id}.
\item We have $\dim_{\M}\s_{X}\leq n-4$, where $\dim_{\M}$ denotes the Minkowski dimension. 
\end{enumerate} 
\end{prop}

\begin{proof}
The proof of these statements is implicitly contained in the proof of \cite[Proposition~6.12]{heatkernel2}, see the discussion after \cite[Proposition~6.20]{heatkernel2}, and consists in verifying that the arguments needed for Bamler's results, see \cite[Theorem~2.18]{Ba20_structure}, carry on in the setting of Ricci flows induced by Ricci shrinkers even without curvature bounds. 
\end{proof}

%-------------------------------------------------------------------
\section{Blow-up Sequences}\label{blow-up-section}

In this section, we first apply the preceding results to the case of blow-up sequences, proving Theorem \ref{thm-blow-up-intro}. Then, under a Type I scalar curvature bound at the point $q$, we prove the splitting results from Theorem \ref{thm-splitting-intro}. Finally, in the last subsection we consider the four-dimensional case, proving Theorems \ref{thm-splitting-four-intro} and \ref{ends-four-dim}.

In all of this section, we let $(M,g,f)$ be a (single) Ricci shrinker and consider a blow-up sequence $(M,g_{k}(t),q)_{k\in \N}$ based at the point $q\in M$, with scales $\lambda_{k}\nearrow \infty$. By construction, each element is given by the parabolically rescaled pointed Ricci flow $(M,g_{-1,\lambda_{k}}(t), (q,0))_{t<0}$, and we can consider its corresponding (extended) metric flow pair $(\X^{k},(\mu^{k}_{t})_{t\leq 0})$ over $I=(-\infty,0]$. The sequence of metric flow pairs over $I=(-\infty,0]$ constructed in this way satisfies all the assumptions needed for Theorems \ref{scaling-compactness-thm-intro} and \ref{good-convergence-intro}, and Propositions \ref{good-convergence-prop} and \ref{metric-soliton-structure} to work. Therefore, after passing to a subsequence, we obtain all the objects and properties seen in the preceding sections; in particular, we will adopt the various notations used so far with their obvious meaning. Finally, denote by $(\X,(\mu_{t})_{t \leq 1})$ the (extended) metric flow pair corresponding to the pointed Ricci flow $(M,g(t),(q,1))_{t<1}$ associated with the Ricci shrinker $(M,g,f)$.   

\begin{proof}[Proof of Theorem \ref{thm-blow-up-intro}]
First of all, each $(\X^{k}, (\mu^{k}_{t})_{t \leq 0})$ is a parabolic rescaling of $(\X, (\mu_{t})_{t \leq 1})$, and the same remains true if we apply a time-shift $t\mapsto t+1$ to $(\X, (\mu_{t})_{t \leq 1})$; the new metric flow pair over $(-\infty,0]$ just obtained will be denoted with the same notation, and obviously is the (extended) metric flow pair corresponding to the pointed Ricci flow $(M,g(1+t), (q,0))_{t<0}$.
Therefore, recalling \cite[Definition~6.22]{Ba20_compactness}, it follows that the limit metric flow pair $(\X^{\infty}, (\mu^{\infty}_{t})_{t\leq 0})$ is a tangent flow of $(\X,(\mu_{t})_{t\leq 0})$ at the point $\x_{0}=(q,0)$.

Now, consider the constant sequence $(\X, (\mu_{t})_{t\leq 0})_{k\in \N}$.  Applying Theorem \ref{scaling-compactness-thm-intro} to it produces a limit metric flow pair $(\Y^{\infty},(\sigma_{t}^{\infty})_{t \leq 0}) \in \F_{(-\infty,0]}^{\ast}(H_{n})$, and it is easy to see that this limit is $(\X, (\mu_{t})_{t\leq 0})$ itself. 
Indeed, defining the trivial correspondence $\C_{0}=\left((Z_{t},d^{Z}_{t})_{t \leq 0}, (\varphi^{k}_{t})_{t\leq 0, k\in \N}\right)$, setting $Z_{t}=\X_{t}$ and $\varphi^{k}_{t}=\id_{\X_{t}}$ for every $t \leq 0$ and $k\in \N$, then $(\X, (\mu_{t})_{t\leq 0})_{k\in \N}$ converges to $(\X, (\mu_{t})_{t\leq 0})$ within $\C_{0}$ on compact time-intervals. Moreover, since metric flow pair $(\X, (\mu_{t})_{t\leq 0})$ is future continuous, $H_{n}$-concentrated and represents a class in $\F^{\ast}_{(-\infty,0]}(H_{n})$, the uniqueness part of Proposition \ref{scaling-compactness-thm-intro} implies that $(\Y^{\infty},(\sigma_{t}^{\infty})_{t \leq 0})$ and $(\X,(\mu)_{t \leq 0})$ are in fact isometric over $(-\infty,0]$, as claimed. 

For all these reasons, it makes sense to apply the second point of Proposition \ref{good-convergence-intro} to $\X$ in place of $\X^{\infty}$, obtaining the theorem up to the continuity statement. Finally, this continuity statement follows from Proposition \ref{ms-cont}.
\end{proof}

\subsection{Splitting Results}

In order to derive more information about the geometry of the limit $\X^{\infty}$ and obtain better convergence results, the first step consists in establishing when a point $(q,t)\in M\times [0,1)$ is an $H$-center of $(q,1)$ with respect to the singular time heat kernel measure $\nu_{(q,1);s}$. This is where the Type~I scalar curvature bound comes in.

\begin{prop}[$H$-Center of Singular Time Heat Kernel]\label{q-is-center} 
Suppose that the scalar curvature at the point $q\in M$ satisfies the Type~I bound from Definition \ref{def.TypeIscalar} with a constant $L>0$. Then there exists a positive constant $H=H(n,L)>0$ such that, for every $t\in [0,1)$, the point $(q,t)$ is an $H$-center of $(q,1)$. In particular, in this case the results from Section \ref{reg-results-section} work with $z_{i_{k}}=q$ and
\begin{equation*}
\eta^{k,i_{k}}_{t}=\nu_{(q,1+t_{i_{k}}/\lambda_{k});1+(t_{i_{k}}+t)/\lambda_{k}}
\end{equation*}
for every $t<0$ and $k\in \N$.
\end{prop}

\begin{proof}
For the standard heat kernel (at regular times), this follows directly from Perelman's Harnack inequality \cite{perelman} and the fact that the Type~I scalar curvature bound gives a suitable bound on the reduced length functional, see for example \cite[Proposition~4.4]{kahler}. One could obtain an analogous result working with the reduced length based at the singular time from \cite{EMT11}, but it is easier to work with regular time heat kernels and the fact that $\sqrt{\var(\cdot,\cdot)}$ satisfies a triangle inequality.

To this end, let $\mathsf{t}_{i}\nearrow 1$ be the sequence of times defining the singular time heat kernel $v_{q}$ and fix $t\in [0,1)$. For $\mathsf{t}_{i}>t$ we have
\begin{equation*}
S(q,\tau)\leq \frac{L}{1-\tau}\leq \frac{L}{\mathsf{t}_{i}-\tau}
\end{equation*}
for every $\tau\in [t,\mathsf{t}_{i})$. Then, as mentioned above, $(q,t)$ is an $H_{0}$-center of $(q,\mathsf{t}_{i})$ for some $H_{0}=H_{0}(n,L)>0$, thanks to \cite[Proposition~4.4]{kahler}. Therefore, using Propositions \ref{lsc-var-prop}, \ref{concentration-shrinker}, and \ref{centers-rescaling}
\begin{align*}
\var_{t}(\delta_{q},\nu_{(q,1);t})&\leq 2\var_{t}(\delta_{q},\nu_{(q,\mathsf{t}_{i});t}) + 2\var_{t}(\nu_{(q,\mathsf{t}_{i});t},\nu_{(q,1);t}) \\
&\leq 2H_{0}(\mathsf{t}_{i}-t) + 4 d_{W_{1}}(\nu_{(q,\mathsf{t}_{i});t},\nu_{(q,1);t})^{2} + 4\var_{t}(\nu_{(q,\mathsf{t}_{i});t}) 
+ 4\var_{t}(\nu_{(q,1);t})\\
&\leq (2H_{0}+8H_{n})(1-t)+4 d_{W_{1}}(\nu_{(q,\mathsf{t}_{i});t},\nu_{(q,1);t})^{2},
\end{align*}
and taking the limit as $\mathsf{t}_{i}\nearrow 1$ we conclude defining $H=2H_{0}+8H_{n}$.
\end{proof} 

Before continuing, we clarify what it means that a metric flow splits a line.

\begin{defn} Let $\X$ be a metric flow over $(-\infty,0)$. We say that $\X$ \emph{splits a line} if there exists a metric flow 
$\widetilde{\X}$ over $(-\infty,0)$ such that 
\begin{equation*}
\X \simeq \widetilde{\X} \times \R,
\end{equation*}
where the Cartesian product of two metric flows is given by \cite[Definition~3.38]{Ba20_compactness}.
\end{defn}

We are ready to prove Theorem \ref{thm-splitting-intro}. The proof is similar to the one of \cite[Theorem~4.6]{kahler}.

\begin{proof}[Proof of Theorem~\ref{thm-splitting-intro}]
Recalling Proposition \ref{good-convergence-prop}, let $\overline{g}_{k}(t)=g_{k}(t+t_{i_{k}})$ and consider the pointed Ricci flows $(M,\overline{g}_{k}(t), (q,0))_{t<0}$. Introduce, for every $k\in \N$, the function $f_{k}\colon M\times (-\infty,0]\to \R$ defined by:
\begin{equation*}
f_{k}(x,t)=c_{k}(F_{k}(x,t)-F_{k}(q,0)),
\end{equation*}
where 
\begin{equation*}
F_{k}(x,t)=F(x,1+(t+t_{i_{k}})/\lambda_{k})
\end{equation*}
for the ``rescaled potential'' $F$ introduced in the beginning of Section \ref{shrinker-preliminaries-section} and 
\begin{equation*}
c_{k}=\lambda_{k}^{\frac{1}{2}}F_{k}(q,0)^{-\frac{1}{2}}.
\end{equation*} 
Observe that $F_{k}(q,0)$ is always positive thanks to the assumptions on $q$, and in particular the definition of $c_{k}$ makes sense.

We think of the functions $f_{k}$ as coupled with the Ricci flow $(M,\overline{g}_{k}(t))_{t<0}$, and we are going to prove that the sequence $(\phi_{k}^{\ast} f_{k})_{k\in \N}$ pre-converges, in $C^{\infty}_{\text{loc}}(\reg^{\infty})$, to a smooth function $f_{\infty}\in C^{\infty}(\reg^{\infty})$. In order to do this we need to give uniform estimates on the derivatives of $\phi^{\ast}_{k}f_{k}$ on compact subsets; we divide the proof into several claims. 

First of all, denoting $\sigma_{k}=|t+t_{i_{k}}|$, standard computations give:
\begin{align}
|\nabla_{\overline{g}_{k}}f_{k}|^{2}_{\overline{g}_{k}}&=1+\lambda_{k}^{-1}c_{k}f_{k}-
\lambda_{k}^{-2}c_{k}^{2}\sigma_{k}^{2} S_{\overline{g}_{k}}, \label{gradient-norm} \\
\nabla^{2}_{\overline{g}_{k}}f_{k}&=\frac{1}{2}\lambda_{k}^{-1}c_{k}\overline{g}_{k}-
\lambda_{k}^{-1}c_{k}\sigma_{k}\Ric_{\overline{g}_{k}}, \label{hessian}\\
\partial_{t}f_{k}&=-\lambda_{k}^{-1}\sigma_{k} c_{k}S_{\overline{g}_{k}}. \label{time-derivative}
\end{align}

\begin{claim} 
The sequence $(F_{k}(q,0))_{k\in \N}$ is bounded below by a positive constant and we have
\begin{equation*}
 \lim_{k\to \infty}\lambda_{k}^{-1}c_{k}=0.
 \end{equation*}
 \end{claim}
 
\begin{proof}
We have $F(q,t)=(1-t)(f(\psi_{t}(q))-\mu_{g})$ and 
\begin{equation*}
\partial_{t}f(\psi_{t}(q))=\frac{|\nabla f|^{2}(\psi_{t}(q))}{1-t}=\frac{f(\psi_{t}(q))-\mu_{g}-S(\psi_{t}(q))}{1-t}\geq \frac{f(\psi_{t}(q))-\mu_{g}-L}{1-t}.
\end{equation*}
Here we used \eqref{scalar-curvature-time} and the Type~I scalar curvature assumption in the last step. Since $d(p,\psi_{t}(q))$ goes to infinity as $t\nearrow 1$, there exists $t_{0}\in (0,1)$ such that $f(\psi_{t}(q))-\mu_{g}>L$ for every 
$t\geq t_{0}$. 
Therefore, we obtain the differential inequality
\begin{equation*}
\frac{\partial_{t}f(\psi_{t}(q))}{f(\psi_{t}(q))-\mu_{g}-L}\geq \frac{1}{1-t},
\end{equation*}
which gives
\begin{equation*}
f(\psi_{t}(q))-\mu_{g}\geq L+\frac{f(\psi_{t_{0}}(q))-\mu_{g}-L}{1-t}=L+\frac{c}{1-t}
\end{equation*}
for every $t\in [t_{0},1)$, with $c>0$ positive constant. 

It follows that $F(q,t)=(1-t)\left(f(\psi_{t}(q))-\mu_{g}\right)\geq c$ for $t\in [t_{0},1)$ and the claim follows, since $F_{k}(q,0)=F(q,1+t_{i_{k}}/\lambda_{k})$ and $1+t_{i_{k}}/\lambda_{k} \geq t_{0}$ for $k$ large. 
\end{proof}

Let $K\subseteq \reg^{\infty}$ be a compact subset and let $\delta\in (0,1)$ be such that $K\subseteq \reg^{\infty}_{t<-\delta}$.
Fix a point $\y\in \reg_{\tau}^{\infty}$, where $\tau\in (-\delta/2,0)$, and define $(y_{k},\tau)=\phi_{k}(\y)$. We recall that the convergence is uniform at this time $\tau$. 

\begin{claim} 
There exists $L_{1}>0$ such that $d_{\overline{g}_{k}(\tau)}(q,y_{k})\leq L_{1}$ for every $k\in \N$.
\end{claim} 

\begin{proof}
Suppose by contradiction that, up to a subsequence, $r_{k}=d_{\overline{g}_{k}(\tau)}(q,y_{k})\to \infty$.
First of all, by the assumption and Proposition \ref{q-is-center} it follows that $(q,\tau)$ is an $H$-center of $(q,0)$ 
for $\eta^{k,i_{k}}_{t}$.
Since the balls $B_{\overline{g}_{k}(\tau)}(q,r_{k}/2)$ and $B_{\overline{g}_{k}(\tau)}(y_{k},1)$ are disjoint when $r_{k}>2$,
we have
\begin{equation*}
\eta^{k,i_{k}}_{\tau}\left(B_{\overline{g}_{k}(\tau)}(y_{k},1)\right)\leq 1-\eta_{\tau}^{k,i_{k}}\left(B_{\overline{g}_{k}(\tau)}(q,r_{k}/2)\right)
\leq 4H/r_{k}^{2}\longrightarrow 0,
\end{equation*}
where we used \cite[Lemma~3.26]{Ba20_compactness}. 

Let $B_{g^{\infty}_{\tau}}(\y,\varepsilon)\subseteq \reg_{\tau}^{\infty}$ be a geodesic ball with compact closure in 
$\reg_{\tau}^{\infty}$, where
$\varepsilon \in (0,1/4)$. Then, recalling the discussion after Remark \ref{good-convergence-remark},
$A=\overline{B}_{g^{\infty}_{\tau}}(\y,\delta) \subseteq W_{k}$ for large $k$, and we can write, using
Point (a) from Proposition \ref{good-convergence-prop},
\begin{align*}
\mu_{\tau}(A)-\eta^{k,i_{k}}_{\tau}(\Phi_{k}(A))&=\int_{A} (w^{\infty}(\x)-\Phi^{\ast}_{k}\overline{w}^{k}(\x))dV_{g^{\infty}_{\tau}}(\x) \\
&\quad + \int_{A}\Phi^{\ast}_{k}\overline{w}^{k}(\x)dV_{g^{\infty}_{\tau}}(\x) - 
\int_{A}\Phi^{\ast}_{k}\overline{w}^{k}(\x)dV_{\Phi_{k}^{\ast}\overline{g}_{k}(\tau)}(\x).
\end{align*}
For the first term, we have
\begin{align*}
\left |\int_{A} (w^{\infty}(\x)-\Phi^{\ast}_{k}\overline{w}^{k}(\x))dV_{g^{\infty}_{\tau}}(\x)\right |
&\leq \int_{A}|w^{\infty}(\x)-\Phi^{\ast}_{k}\overline{w}^{k}(\x)|dV_{g^{\infty}_{\tau}}(\x) \\
&\leq \Vol_{g^{\infty}_{\tau}}(A) \sup_{A} |w^{\infty}(\x)-\Phi^{\ast}_{k}\overline{w}^{k}(\x)| \longrightarrow 0
\end{align*}
thanks to smooth convergence. Furthermore,
\begin{align*}
&\left | \int_{A}\Phi^{\ast}_{k}\overline{w}^{k}(\x)dV_{g^{\infty}_{\tau}}(\x) - \int_{A}\Phi^{\ast}_{k}\overline{w}^{k}(\x)dV_{\Phi_{k}^{\ast}\overline{g}_{k}(\tau)}(\x)\right| \\
&\qquad\qquad\qquad\qquad \leq C(A)\sup_{A}|\Phi^{\ast}_{k}\overline{w}^{k}(\x)| \sup_{A} |g^{\infty}_{\tau}(\x) -  \Phi_{k}^{\ast}\overline{g}_{k}(\tau)(\x) \longrightarrow 0
\end{align*}
again by smooth convergence, and since $|\Phi^{\ast}_{k}\overline{w}^{k}|$ is uniformly bounded on $A$. We thus conclude that
\begin{equation*}
|\mu^{\infty}_{\tau}(A)-\eta^{k,i_{k}}_{\tau}(\Phi_{k}(A))|\longrightarrow 0.
\end{equation*}
Now, it is standard to see that $\Phi_{k}(A)\subseteq B_{\overline{g}_{k}(\tau)}(y_{k}, 2\varepsilon)\subseteq B_{\overline{g}_{k}(\tau)}(y_{k}, 1)$ for $k$ large, which implies $\mu^{\infty}_{\tau}(A)=0$, contradicting the fact that $\supp \mu^{\infty}_{\tau}=\X^{\infty}_{\tau}$.
\end{proof}

\begin{claim}\label{claim-stima-fk-yk} 
There exists $L_{2}>0$ such that $|f_{k}(y_{k},\tau)|\leq L_{2}$ for every $k\in \N$.
\end{claim}

\begin{proof}
First of all, we confront $f_{k}(y_{k},\tau)$ with $f_{k}(q,\tau)$. If $\gamma\colon [0,1]\to M$ is a minimising geodesic for the metric $\overline{g}_{k}(\tau)$ connecting $q$ to $y_{k}$, then 
\begin{equation*}
|\partial_{t}f_{k}(\gamma(t),\tau)|=|\langle \nabla_{\overline{g}_{k}(\tau)} f_{k}(\gamma(t),\tau),\dot{\gamma}(t)\rangle |
\leq L_{1} |\nabla_{\overline{g}_{k}(\tau)} f_{k}(\gamma(t),\tau)|_{\overline{g}_{k}(\tau)} 
\leq L_{1} \sqrt{1+\lambda_{k}^{-1}c_{k}f_{k}(\gamma(t),\tau)}.
\end{equation*}
Therefore, if we define $\alpha(t)=1+\lambda_{k}^{-1}c_{k}f_{k}(\gamma(t),\tau)$, the preceding inequality implies
\begin{equation*}
 |\dot{\alpha}(t)| \leq \lambda_{k}^{-1}c_{k}L_{1}\sqrt{\alpha(t)}.
\end{equation*}
Integrating over $[0,1]$ we obtain
\begin{equation*}
\left |\sqrt{\alpha(1)}- \sqrt{\alpha(0)}\right |\leq \frac{1}{2}L_{1}\lambda_{k}^{-1}c_{k},
\end{equation*} 
and using the elementary inequality $|x-y|\leq |\sqrt{x}-\sqrt{y}|^{2}+2\sqrt{y}|\sqrt{x}-\sqrt{y}|$ it follows that
\begin{equation*}
\left |\alpha(1)-\alpha(0)\right | \leq \frac{1}{4}L_{1}^{2}\lambda_{k}^{-2}c_{k}^{2} + L_{1}\lambda_{k}^{-1}c_{k}\sqrt{\alpha(0)}.
\end{equation*}
Using the definition of $\alpha(t)$, we get
\begin{equation*}
|f_{k}(y_{k},\tau)- f_{k}(q,\tau)|\leq \frac{1}{4}L_{1}^{2}\lambda_{k}^{-1}c_{k}+L_{1}\sqrt{1+\lambda_{k}^{-1}c_{k}f_{k}(q,\tau)}.
\end{equation*}
Now we estimate $f_{k}(q,\tau)$ using $f_{k}(q,0)=0$. Consider the function $\beta(t)=f_{k}(q,t)$ for $t\in [\tau,0]$. Then
\begin{equation*}
|\dot{\beta}(t)|=|\partial_{t}f_{k}(q,t)|= \lambda_{k}^{-1} c_{k}\sigma_{k}(t) S_{\overline{g}_{k}(t)}(q,t) \leq \lambda_{k}^{-1}c_{k}L.
\end{equation*}
It follows that
\begin{equation*}
|\beta(\tau)-\beta(0)|\leq \lambda_{k}^{-1}c_{k}L |\tau| \leq \lambda_{k}^{-1}c_{k}L,
\end{equation*}
and hence
\begin{equation*}
|f_{k}(q,\tau)|\leq \lambda_{k}^{-1}c_{k}L.
\end{equation*}
Finally, we obtain
\begin{align*}
|f_{k}(y_{k},\tau)| &\leq \frac{1}{4}L_{1}^{2}\lambda_{k}^{-1}c_{k}+L_{1}\sqrt{1+\lambda_{k}^{-1}c_{k}f_{k}(q,\tau)} + 
\lambda_{k}^{-1}c_{k}L \\
&\leq \frac{1}{4}L_{1}^{2}\lambda_{k}^{-1}c_{k}+L_{1}\sqrt{1+\lambda_{k}^{-2}c_{k}^{2}L} + \lambda_{k}^{-1}c_{k}L.
\end{align*}
Since $\lambda_{k}^{-1}c_{k}\to 0$ as $k\to \infty$, there exists $L_{2}>0$ such that $|f_{k}(y_{k},\tau)|\leq L_{2}$ for every $k\in \N$,
as desired.
\end{proof}

The next result will enable us to control the functions $f_{k}$ over the whole compact set $K$ using curves. Observe that we can endow $\reg^{\infty}$ with the Riemannian metric $\overline{g}_{\infty}=g_{\infty} + d\mathfrak{t}^{2}$, thanks to the decomposition $T\reg^{\infty} = \ker d \mathfrak{t} \oplus \langle \partial_{\mathfrak{t}}\rangle$. Here and in the rest of the proof we write $\mathfrak{t}$ instead of $\mathfrak{t}^{\infty}$ to simplify the notations. 

\begin{claim} 
For every $\z\in K$ there exists a piecewise smooth curve $\gamma_{\z}\colon I_{\z} \to \reg^{\infty}$, such that
\begin{enumerate}
\item $\gamma_{\z}$ connects $\z$ to $\y$;
\item $\bigcup_{\z\in K}\gamma_{\z}(I_{\z})$ is contained in a compact subset $\widetilde{K} \subseteq \reg^{\infty}$;
\item $\sup_{\z\in K}\sup_{t\in I_{\z}}|\dot{\gamma}_{\z}(t)|_{\overline{g}_{\infty}}<\infty$;
\item $I_{\z}=[0,b(\z)]$ and $\sup_{\z\in K}b(\z)<\infty$.
\end{enumerate}
\end{claim}

\begin{proof}
Since $K$ is compact we can find points $\z_{1},\ldots,\z_{m}\in K$ and geodesic balls for $\overline{g}_{\infty}$ covering $K$, i.e. $K\subseteq \bigcup_{i=1}^{m}B(\z_{i},\varepsilon_{i})$. Moreover, we can suppose that each of these balls has compact closure in $\reg^{\infty}$. For every $1\leq i \leq m$, let $\eta_{i}\colon J_{\z_{i}}\to \reg^{\infty}$ be a unit-speed piecewise smooth curve from $\z_{i}$ to $\y$. Now, given a general $\z\in K$ with $\z\in B(\z_{i},\varepsilon_{i})$, define $\gamma_{\z}\colon I_{\z}\to \reg^{\infty}$ as the composition of $\eta_{i}$ with a unit-speed geodesic inside $B(\z_{i},\varepsilon_{i})$ connecting $\z$ to $\z_{i}$, where $I_{\z}=[0,b(\z)]$. Then $\gamma_{\z}$ clearly satisfies the first and the third properties above, and moreover
\begin{equation*}
\bigcup_{\z\in K}\gamma_{\z}(I_{\z})\subseteq \bigcup_{i=1}^{m}\left(\overline{B}(\z_{i},\varepsilon_{i})\cup \eta_{i}(J_{\z_{i}})\right)=\widetilde{K},
\end{equation*}
which is compact. Finally, from the definition of $\gamma_{\z}$ it follows
\begin{equation*}
b(\z) = L_{\overline{g}_{\infty}}(\gamma_{\z})= d(\z,\z_{i}) + L_{\overline{g}_{\infty}}(\eta_{i}) \leq \max\left\{ \varepsilon_{j} + L_{\overline{g}_{\infty}}(\eta_{j})\colon j=1,\ldots,m \right\},
\end{equation*}
and we are done. 
\end{proof}

\begin{claim}\label{claim-est-fk} 
There exists $L_{3}>0$ such that $|f_{k}(\phi_{k}(\y))|\leq L_{3}$ for every $\y \in K$ and $k\in \N$.
\end{claim}

\begin{proof}
Since $\widetilde{K}\subseteq \reg^{\infty}$ from the previous claim is compact, we have $\widetilde{K}\subseteq U_{k}$ for $k$ large enough. Define 
\begin{equation*}
\phi_{k}(\gamma_{\z}(t))=(\gamma^{k}_{\z}(t), s_{\z}(t))\in \phi_{k}(U_{k})
\end{equation*}
for any $\z\in K$ and $t\in [0,b(\z)]$, where $s_{\z}(t)=\mathfrak{t}(\gamma_{\z}(t))$. Then 
\begin{equation*}
\partial_{t}(f_{k}\circ \phi_{k}\circ \gamma_{\z})(t)= \left \langle \nabla_{\overline{g}_{k}(s(t))} f_{k}(\phi_{k}(\gamma_{\z}(t))), \dot{\gamma}^{k}_{\z}(t)\right \rangle + \partial_{t} f_{k}(\phi_{k}(\gamma_{z}(t))) \dot{s}_{\z}(t).
\end{equation*}
Writing $\dot{\gamma}_{\z}(t)=w(t)+a(t)\partial_{\mathfrak{t}}\in \ker d\mathfrak{t} \oplus \langle \partial_{\mathfrak{t}} \rangle$ we have, since $d\phi_{k}$ respects the tangent bundles decomposition,
\begin{equation*}
\dot{\gamma}_{\z}^{k}(t)=d\pi_{1}\left(d\phi_{k}(\dot{\gamma}_{\z}(t))\right)=d\phi_{k}(w(t)),
\end{equation*}
where $\pi_{1}\colon M\times (-\infty,0) \to M$ is the projection on the spatial component, while
\begin{equation*}
\dot{s}_{\z}(t)=d\mathfrak{t}(\dot{\gamma}_{\z}(t))=a(t).
\end{equation*}

In the following, $C$ denotes a positive constant that can vary from line to line, but which is independent of $\z\in K$ and $k\in \N$. From smooth convergence we have $\phi^{\ast}_{k}\overline{g}_{k}\leq Cg_{\infty}$, and therefore
\begin{equation*}
|\dot{\gamma}^{k}_{\z}(t)|_{\overline{g}_{k}}=|w(t)|_{\phi_{k}^{\ast}\overline{g}_{k}}\leq C|w(t)|_{g_{\infty}} 
\leq C|w(t)|_{\overline{g}_{\infty}}\leq C.
\end{equation*}
On the other hand, 
\begin{equation*}
|\dot{s}_{\z}(t)|=|a(t)|\leq |\dot{\gamma}_{\z}(t)|_{g_{\infty}}\leq 1
\end{equation*}
and therefore we have
\begin{align*}
|\partial_{t}(f_{k}\circ \phi_{k}\circ \gamma_{\z})(t)| &\leq C|\nabla_{\overline{g}_{k}} f_{k}(\phi_{k}(\gamma_{\z}(t)))|_{\overline{g}_{k}}
+ |\partial_{t}f_{k}(\phi_{k}(\gamma_{\z}(t)))| \\
&\leq C\sqrt{1+\lambda_{k}^{-1}c_{k}f_{k}(\phi_{k}(\gamma_{\z}(t)))} + 
\lambda_{k}^{-1}c_{k}\sigma_{k}(s_{\z}(t)) S_{\overline{g}_{k}}(\phi_{k}(\gamma_{\z}(t)))\\
&\leq C \left(\sqrt{1+\lambda_{k}^{-1}c_{k}f_{k}(\phi_{k}(\gamma_{\z}(t)))} + 
\lambda_{k}^{-1}c_{k}\sigma_{k}(s_{\z}(t)) \right) \\
&\leq C \left(\sqrt{1+\lambda_{k}^{-1}c_{k}f_{k}(\phi_{k}(\gamma_{\z}(t)))} + 
1 \right) \\
&\leq C \sqrt{2+\lambda_{k}^{-1}c_{k}f_{k}(\phi_{k}(\gamma_{\z}(t)))},
\end{align*}
where we used $S_{\overline{g}_{k}}(\phi_{k}(\gamma_{\z}(t)))\leq C$, which follows from smooth convergence,
$\lambda_{k}^{-1}c_{k} \to 0$ and $\sigma_{k}(s_{\z}(t))\leq C$.
Setting 
\begin{equation*}
\alpha(t)=2+\lambda_{k}^{-1}c_{k}f_{k}(\phi_{k}(\gamma_{\z}(t)),
\end{equation*}
it follows that
\begin{equation*}
\dot{\alpha}(t) \leq C\lambda_{k}^{-1}c_{k} \sqrt{\alpha(t)}.
\end{equation*}
Therefore, integrating over $I_{\z}=[0,b(\z)]$, we have 
\begin{equation*}
\left |\sqrt{\alpha(b(\z))}-\sqrt{\alpha(0)}\right |\leq \frac{1}{2}C\lambda_{k}^{-1}c_{k}b(\z)
\end{equation*}
and so
\begin{equation*}
|\alpha(b(\z))-\alpha(0)|\leq \frac{1}{4}C^{2}\lambda_{k}^{-2}c^{2}_{k}b(\z)^{2} + C\lambda_{k}^{-1}c_{k}b(\z)\sqrt{\alpha(b(\z))}.
\end{equation*}
Since $b(\z)\leq C$ we find
\begin{equation*}
|\alpha(b(\z))-\alpha(0)|\leq \frac{1}{4}C^{4}\lambda_{k}^{-2}c^{2}_{k} + C^{2}\lambda_{k}^{-1}c_{k}\sqrt{\alpha(b(\z))}.
\end{equation*}
Now, using the definition of $\alpha(t)$, we obtain
\begin{equation*}
|f_{k}(y_{k},\tau)-f_{k}(\phi_{k}(\z))| \leq \frac{1}{4}C^{4}\lambda_{k}^{-1}c_{k} + 
C^{2}\sqrt{2+\lambda_{k}^{-1}c_{k}f_{k}(y_{k},\tau)},
\end{equation*}
and therefore there exists $L_{3}>0$ such that
\begin{equation*}
|f_{k}(\phi_{k}(\z))|\leq L_{3}
\end{equation*}
for every $\z\in K$ and $k\in \N$, where we used Claim \ref{claim-stima-fk-yk}. 
\end{proof}

\begin{claim} 
After passing to a subsequence, the sequence of functions $(\phi_{k}^{\ast}f_{k})_{k\in \N}$ converges 
in $C^{\infty}_{\text{loc}}(\reg^{\infty})$ to a smooth function $f_{\infty}\in C^{\infty}(\reg^{\infty})$.
\end{claim}

\begin{proof}
Let $K\subseteq \reg^{\infty}$ be a compact subset and fix an integer $m\in \N$. We want to derive uniform upper bounds for $\|\phi^{\ast}_{k}f_{k}\|_{C^{m}(K)}$, and in order to do so we first estimate the norms $\|f_{k}\|_{C^{m}_{\overline{g}_{k}}(\Omega_{k})}$, where $\Omega_{k}=\phi_{k}(K)\subseteq M\times (-\infty,0)$. Note that here and in the following the norms are as in \eqref{cm-norm-def}.

As usual, we denote by $C_{m}$ a positive constant depending on $K$ and $m$ but independent of $k\in \N$, that can vary from line to line, while $r,s\in \N$ is any pair of integers satisfying $r+s\leq m$.

First of all, if $m=0$, we conclude thanks to Claim \ref{claim-est-fk}; therefore, we assume $m\geq 1$. 
From Point~(a) in Proposition \ref{good-convergence-prop} it is easy to deduce that
\begin{equation}\label{unif-est-curvature}
\sup_{k\in \N}\| \Ric_{\overline{g}_{k}}\|_{C^{m}_{\overline{g}_{k}}(\Omega_{k})} \leq C_{m},
\quad \quad \sup_{k\in \N} \|S_{\overline{g}_{k}}\|_{C^{m}_{\overline{g}_{k}}(\Omega_{k})}\leq C_{m}.
\end{equation}

Consider the following system of equations 
\begin{equation}\label{system-eq}
\left\{
\begin{aligned}
\nabla^{2}_{\overline{g}_{k}}f_{k}(x,t) &=\frac{1}{2}\lambda_{k}^{-1}c_{k}\overline{g}_{k}(t)-
\lambda_{k}^{-1}c_{k}\sigma_{k}(t)\Ric_{\overline{g}_{k}(t)},\\
\partial_{t}f_{k}(x,t) &=-\lambda_{k}^{-1}c_{k}\sigma_{k}(t) S_{\overline{g}_{k}}(x,t).
\end{aligned}\right.
\end{equation}
From \eqref{system-eq}, using \eqref{gradient-norm}, Claim \ref{claim-est-fk} and \eqref{unif-est-curvature}, it follows
\begin{equation}\label{unif-est-f-order2}
\| \partial_{t}^{i}\nabla_{\overline{g}_{k}}^{j} f_{k}\|_{L^{\infty}(\Omega_{k})}
+ \| \nabla_{\overline{g}_{k}}^{j}\partial_{t}^{i} f_{k}\|_{L^{\infty}(\Omega_{k})} \leq C_{m}
\end{equation}
for every $k\in \N$ and $i,j\in \N$ with $i+j\leq 2$. 
Recall that here $\nabla_{\overline{g}_{k}} f_{k}$ stands for $d f_{k}$, and obviously 
$\partial_{t} (d f_{k})=d (\partial_{t}f_{k})$.
If $m\leq 2$ these are the desired bounds, and we are done.

Suppose $m>2$. Starting from the system \eqref{system-eq}, one can easily write down higher order equations for $f_{k}$ which, 
together with \eqref{unif-est-curvature} and \eqref{unif-est-f-order2}, give
\begin{equation}\label{unif-est-f}
\|\partial_{t}^{r}\nabla_{\overline{g}_{k}}^{s} f_{k}\|_{L^{\infty}(\Omega_{k})} + 
\|\nabla_{\overline{g}_{k}}^{s} \partial_{t}^{r} f_{k}\|_{L^{\infty}(\Omega_{k})} \leq C_{m}
\end{equation}
for every $k\in \N$. 
Now, from \eqref{unif-est-f}, Point~(a) in Proposition \ref{good-convergence-prop} and \eqref{unif-est-curvature}, we get
\begin{equation}
\|\partial_{\mathfrak{t}}^{r}\nabla_{g_{\infty}}^{s} \phi_{k}^{\ast} f_{k}\|_{L^{\infty}(K)} +
\|\nabla_{g_{\infty}}^{s} \partial_{\mathfrak{t}}^{r} \phi_{k}^{\ast} f_{k}\|_{L^{\infty}(K)} \leq C_{m}
\end{equation}
for every $k\in \N$, or in other words 
\begin{equation*}
\sup_{k\in \N} \|\phi^{\ast}_{k}f_{k}\|_{C^{m}_{g_{\infty}}(K)}\leq C_{m}.
\end{equation*}
Since $K\subseteq \reg^{\infty}$ and $m\in \N$ have been chosen arbitrarily, the convergence finally follows from the Arzelà-Ascoli Theorem, proving the claim.
\end{proof}

At this point, passing equations \eqref{gradient-norm}, \eqref{hessian}, \eqref{time-derivative} to the limit, we obtain
\begin{equation*}
|\nabla_{g^{\infty}} f_{\infty}|_{g_{\infty}}\equiv 1, \quad \quad \nabla^{2}_{g^{\infty}}f_{\infty} \equiv 0, 
\quad \quad \partial_{\mathfrak{t}}f_{\infty}\equiv 0
\end{equation*}
on $\reg^{\infty}$. We conclude using \cite[Theorem~15.50]{Ba20_structure}. 
\end{proof}

\subsection{The Four-Dimensional Case}

In this last part, we specialise what we have seen so far to dimension $n=4$. We start proving Theorem \ref{thm-splitting-four-intro}.

\begin{proof}[Proof of Theorem~\ref{thm-splitting-four-intro}]
From the previous results, we know that $\X^{\infty}$ splits a line, i.e.~$\X^{\infty}=\widetilde{\X}\times \R$ as metric flows. Moreover, by Bamler's result \cite[Theorem 2.46]{Ba20_structure} (respectively its generalisation from \cite[Theorem 6.12]{heatkernel2} to Ricci flows induced by Ricci shrinkers) each time slice $\X^{\infty}_{t}$ is a four-dimensional orbifold Ricci shrinker with isolated singularities. Since $\s_{t}^{\infty}=\widetilde{\s}_{t}\times \R$, it follows $\widetilde{\s}_{t}=\emptyset$, otherwise $\s^{\infty}_{t}$ would not be discrete; therefore, $\s^{\infty}=\emptyset$ and $\reg^{\infty}=\X^{\infty}$, as claimed. Note that we also have $\widetilde{\X}=\widetilde{\reg}$, and in the following we will write $\X^{\infty}$, $\widetilde{\X}$ instead of $\reg^{\infty}$ and $\widetilde{\reg}$, respectively.

Now, Proposition \ref{metric-soliton-structure} gives $\X^{\infty}=X\times (-\infty,0)$, where 
$(X,g_{X})=(\X^{\infty}_{-1},g^{\infty}_{-1})$ is a smooth Ricci shrinker of dimension four with potential $f_{X}$. 
In particular, $X=\X^{\infty}_{-1}=\widetilde{\X}_{-1}\times \R$ and $(N,g_{N})=(\widetilde{\X}_{-1}, \widetilde{g}_{-1})$
is a smooth three-dimensional Ricci shrinker, where the potential $\widetilde{f}$ is given by 
\begin{equation*}
f_{X}(x,y)=\widetilde{f}(x)+\frac{1}{4}|y-y_{0}|^{2},
\end{equation*}
for some $y_{0}\in \R$, as one can easily verify.
 
Since $\X^{\infty}=\reg^{\infty}$, the convergence from Proposition \ref{good-convergence-prop} is smooth everywhere. Fix $t=-1$, and recall that the convergence is uniform at this time. Therefore, Proposition \ref{conv-centers} gives points $\y_{k}\in \X^{k,i_{k}}_{-1}$ and $\y_{\infty}\in \X^{\infty}_{-1}$ such that $\y_{k}\to \y_{\infty}$ strictly within $\C$, and it follows $\phi_{k}^{-1}(\y_{k})\to \y_{\infty}$ in $\reg^{\infty}$ thanks to Point (b) in Proposition \ref{good-convergence-prop}.

Now, under the identification $\X^{\infty}_{-1}=X$, it follows that $(W_{k})_{k\in \N}$ forms an exhaustion of $X$ by open subsets
and $\Phi_{k}^{\ast}\overline{g}_{k}(-1)\to g_{X}$ in $C^{\infty}_{\text{loc}}(X)$; here we are using the notations after Remark 
\ref{good-convergence-remark}.
If we set $\y_{k}=(y_{k},-1)$ and $\y_{\infty}=(y_{\infty},-1)$, then obviously $y_{k}\in \Phi_{k}(W_{k})$ and 
$\Phi_{k}^{-1}(y_{k})\to y_{\infty}$ in $X$. This implies that 
\begin{equation*}
(M,\overline{g}_{k}(-1),y_{k})\longrightarrow (X,g_{X},y_{\infty})
\end{equation*}
in the pointed smooth Cheeger-Gromov sense. 
Again from Proposition \ref{conv-centers} we have $d_{\overline{g}_{k}(-1)}(q,y_{k})\leq r$ for some uniform $r>0$, 
and thus we can find $q_{\infty}\in X$ such that 
\begin{equation*}
(M,\overline{g}_{k}(-1), q)\longrightarrow (X,g_{X},q_{\infty})
\end{equation*}
in the pointed smooth Cheger-Gromov sense. 

We now translate the above convergence in terms of the original shrinker $(M,g, f)$. 
Start writing 
\begin{equation*}
\overline{g}_{k}(-1)=g_{k}(-1+t_{i_{k}})=\lambda_{k}g(1+(-1+t_{i_{k}})/\lambda_{k})=(1-t_{i_{k}})\psi^{\ast}_{s_{k}+\vartheta_{k}}g,
\end{equation*}
where $s_{k}=1-1/\lambda_{k}$ and $\vartheta_{k}=t_{i_{k}}/\lambda_{k}$. Note that $s_{k}\in (0,1)$, while $\vartheta_{k}\in (-1,0)$.

\begin{claim} 
Let $h_{k}=\psi^{\ast}_{s_{k}+\vartheta_{k}}g$. Then $\Phi_{k}^{\ast} h_{k}\to g_{X}$
in $C^{\infty}_{\text{loc}}(X)$. 
\end{claim}

\begin{proof}
Given a compact subset $K\subseteq X$, we have, with respect to the metric $g_{X}$,
\begin{equation*}
|\nabla^{m}(\Phi^{\ast}_{k}h_{k} - g_{X})| \leq |\nabla^{m}(\Phi^{\ast}_{k} h_{k} - \Phi_{k}^{\ast} \overline{g}_{k}(-1))| 
+ |\nabla^{m}(\Phi_{k}^{\ast}\overline{g}_{k}(-1)- g_{X})|,
\end{equation*}
and since the second term on the right hand side goes to zero as $k\to \infty$, it is sufficient to estimates the first one. 
Now,
\begin{equation*}
|\nabla^{m}(\Phi^{\ast}_{k}h_{k} - \Phi_{k}^{\ast}\overline{g}_{k}(-1))|=|t_{i_{k}}||\nabla^{m}\Phi^{\ast}_{k}h_{k}|
=\frac{|t_{i_{k}}|}{(1-t_{i_{k}})}|\nabla^{m}\Phi^{\ast}_{k}\overline{g}_{k}(-1)|.
\end{equation*}
Having smooth convergence on $K$, it follows that $|\nabla^{m}\Phi^{\ast}_{k}\overline{g}_{k}(-1)|\leq C$ for some positive constant $C>0$, and therefore
\begin{equation*}
|\nabla^{m}(\Phi^{\ast}_{k} h_{k} - \overline{g}_{k}(-1))| \longrightarrow 0
\end{equation*}
as $k\to \infty$; this concludes the proof. 
\end{proof}

The claim implies that $(M,\psi_{s_{k}+\vartheta_{k}}^{\ast}g, q)\to (X,g_{X},q_{\infty})$, and therefore from the invariance of smooth convergence under diffeomorphisms we obtain
\begin{equation*}
(M,g,\psi_{s_{k}+\vartheta_{k}}(q))\longrightarrow (X,g_{X},q_{\infty}).
\end{equation*}
Now we estimate the distance between $\psi_{s_{k}+\vartheta_{k}}(q)$ and $\psi_{s_{k}}(q)$. Defining $\gamma_{k}\colon [s_{k}+\vartheta_{k},s_{k}]\to M$ as $\gamma_{k}(\tau)=\psi_{\tau}(q)$, it follows that
\begin{align*}
d_{g}(\psi_{s_{k}+\vartheta_{k}}(q), \psi_{s_{k}}(q)) &\leq L_{g}(\gamma_{k})
=\int_{s_{k}+\vartheta_{k}}^{s_{k}} \frac{|\nabla_{g}f(\gamma_{k}(\tau))|}{1-\tau}d\tau 
\leq \int_{s_{k}+\vartheta_{k}}^{s_{k}} \frac{\sqrt{f(\gamma_{k}(\tau))-\mu_{g}}}{1-\tau}d\tau \\
&\leq \frac{\sqrt{f(\gamma_{k}(s_{k}))-\mu_{g}}}{1-s_{k}} \int_{s_{k}+\vartheta_{k}}^{s_{k}}d\tau
= |\vartheta_{k}| \frac{\sqrt{f(\gamma_{k}(s_{k}))-\mu_{g}}}{1-s_{k}},
\end{align*}
where we used that $f(\gamma_{k}(\tau))-\mu_{g}$ is non-decreasing and $(1-\tau)^{-1}\leq (1-s_{k})^{-1}$ on $[s_{k}+\vartheta_{k},s_{k}]$. Using the definitions of $s_{k}$ and $\vartheta_{k}$, we get
\begin{equation*}
d_{g}(\psi_{s_{k}+\vartheta_{k}}(q), \psi_{s_{k}}(q)) \leq |t_{i_{k}}| \sqrt{f(\psi_{1-1/\lambda_{k}}(q))-\mu_{g}}.
\end{equation*}
Recalling Remark \ref{double-conv-rem}, we can suppose that the right hand side of the above expression is bounded by $1$, or in other words $d_{g}(\psi_{s_{k}+\vartheta_{k}}(q), \psi_{s_{k}}(q)) \leq 1$. Therefore, setting $x_{k}=\psi_{s_{k}}(q)$ there exists a point $x_{\infty}\in X$ such that 
\begin{equation*}
(M,g,x_{k})\longrightarrow (X,d,x_{\infty})
\end{equation*}
in the pointed smooth Cheeger-Gromov sense, and we are done. 

For the last assertion, observe that $S(x_{k})\to S(x_{\infty})$ thanks to (pointed) smooth convergence, and we conclude from \cite{rps}, recalling what we already stated in Section \ref{shrinker-preliminaries-section}.
\end{proof}

At this point, our last result is almost immediate.

\begin{proof}[Proof of Theorem~\ref{ends-four-dim}]
First of all, the vector fields $\nabla f$ and $(1-t)^{-1}\nabla f$ have the same integral curves. Therefore, we can write $x_{k}=\psi_{t_{k}}(q)$ for some sequence of times $(t_{k})_{k\in \N}$. Now, $x_{k}\to \infty$ is equivalent to $f(x_{k})-\mu_{g} \to \infty$, thanks to \eqref{fgrowth}. Moreover, since $|\nabla f |^{2}\leq f-\mu_{g}$, we have $f(x_{k})-\mu_{g} \leq (1-t_{k})^{-1}(f(q)-\mu_{g})$ and therefore $t_{k}\to 1$. Up to a subsequence, we can assume that $t_{k}\nearrow 1$ monotonically. Now, define $\lambda_{k}=(1-t_{k})^{-1}$ and consider the blow-up sequence $(M,g_{k}(t),q)_{k\in \N}$ based at $q$ with scales $\lambda_{k}$. Finally, we conclude applying Theorem \ref{thm-splitting-four-intro} to this sequence, noting that $1-1/\lambda_{k}=t_{k}$.
\end{proof}

\appendix

%------------------------------------------------------------------------------------------------------
\section{Schauder Estimates for the Conjugate Heat Equation}\label{schauder-appendix}

In the first part of this appendix, we recall some basic theory regarding classical parabolic partial differential equations; for reference the reader can consult \cite{krylov}. In the second part, we then apply the theory to the conjugate heat equation along a general Ricci flow.

\subsection{Parabolic Schauder Theory}

Let us start with some general notations. Points in $\R^{n+1}$ are denoted by $(x,t)=(x^{1},\ldots,x^{n},t)$, where $x^{i}$ are the spatial coordinates and $t$ is the time variable. If $Q\subseteq \R^{n+1}$ is a domain, the function space $C^{k}_{1}(Q)$ is given by
\begin{equation*}
C^{k}_{1}(Q)=\{u\colon Q\to \R\colon D^{\gamma}u, \partial_{t}u\in C^{0}(Q)\ \forall \gamma \in \N^{n}, |\gamma|\leq k\},
\end{equation*}
where $D^{\gamma}u=D^{\gamma}_{x}u$ has the usual meaning for a multi-index $\gamma\in \N^{n}$. Similarly, the symbol $D^{k}u$ stands for the collection of all the partial derivatives of $u$ of order $k$ with respect to the spatial variables $x$, or in other words
\begin{equation*}
D^{k}u=(D^{\gamma}u)_{|\gamma|= k}.
\end{equation*}
We can also introduce the supremum norm of $D^{k}u$, which is defined by
\begin{equation*}
\|D^{k}u\|_{L^{\infty}(Q)}=\max_{|\gamma|=k}\|D^{\gamma}u\|_{L^{\infty}(Q)}.
\end{equation*}

Given points $z_{1}=(x_{1},t_{1})$ and $z_{2}=(x_{2},t_{2})\in \R^{n+1}$, their parabolic distance is the number
\begin{equation*}
\varrho(z_{1},z_{2})=|x_{1}-x_{2}|+|t_{1}-t_{2}|^{1/2}.
\end{equation*}
For a fixed constant $\alpha \in (0,1)$ and a function $u\in C^{2}_{1}(Q)$, we define the following parabolic H\"older norms and semi-norms:
\begin{align*}
[u]_{\alpha,\alpha/2;Q}&=\sup_{z\neq w\in Q}\frac{|u(z)-u(w)|}{\varrho(z,w)^{\alpha}},\\
\| u \|_{\alpha,\alpha/2;Q}&=\| u \|_{L^{\infty}(Q)}+[u]_{\alpha,\alpha/2;Q},\\
[u]_{2+\alpha,1+\alpha/2;Q}&=[\partial_{t}u]_{\alpha,\alpha/2;Q}+[D^{2}u]_{\alpha,\alpha/2;Q},
\end{align*}
where
\begin{equation*}
[D^{2}u]_{\alpha,\alpha/2;Q}=\sum_{i,j=1}^{n}[\partial_{i}\partial_{j}u]_{\alpha,\alpha/2;Q}.
\end{equation*}
Finally, we set
\begin{equation*}
\| u \|_{2+\alpha,1+\alpha/2;Q}=\|u \|_{L^{\infty}(Q)}+\|\partial_{t}u\|_{L^{\infty}(Q)}+\|Du\|_{L^{\infty}(Q)}+\|D^{2}u\|_{L^{\infty}(Q)}
+[u]_{2+\alpha,1+\alpha/2;Q}.
\end{equation*}
At this point, one can define the following function spaces:
\begin{align*}
C^{\alpha,\alpha/2}(Q)&=\{u\in C^{2}_{1}(Q) \colon \|u\|_{\alpha,\alpha/2;Q}<\infty\},\\
C^{2+\alpha,1+\alpha/2}(Q)&=\{u\in C^{2}_{1}(Q) \colon \|u\|_{2+\alpha,1+\alpha/2;Q}<\infty\}.
\end{align*}

We will always work with linear differential operators $L$ of second order:
\begin{equation}\label{operator}
L=\sum_{i,j=1}^{n} a^{ij}(x,t)\partial_{i}\partial_{j}+\sum_{i=1}^{n} b^{i}(x,t)\partial_{i}+c(x,t),
\end{equation}
where the coefficients $a^{ij},b^{i},c$ are smooth functions and $a^{ij}=a^{ji}$.

For $r>0$ denote $Q_{r}=B(0,r)\times (-r^{2},0)$. One can prove the following interior estimates.

\begin{prop}[Schauder Estimates]\label{schauder} 
Let $\alpha \in (0,1)$ and $R>0$. Let $u \in C^{2+\alpha,1+\alpha/2}(Q_{3R})$ be a solution of
\begin{equation*}
\partial_{t}u=Lu+f 
\end{equation*}
on $Q_{3R}$, where $f\in C^{\alpha,\alpha/2}(Q_{3R})$ and $L$ is as in \eqref{operator}. Suppose that the coefficients of $L$ satisfy the following estimates:
\begin{align*}
&\|a^{ij}\|_{\alpha,\alpha/2;Q_{2R}},\; \|b^{i}\|_{\alpha,\alpha/2;Q_{2R}},\; \|c\|_{\alpha,\alpha/2;Q_{2R}}\leq \Lambda, \\
&\lambda^{-1}|\xi |^{2}\leq \sum_{i,j=1}^{n} a^{ij}(x,t)\xi_{i}\xi_{j}\leq \lambda |\xi |^{2}, \quad \forall \xi \in \R^{n}, (x,t)\in Q_{2R},
\end{align*}
for some $\Lambda, \lambda >0$. Then
\begin{equation*}
\|u \|_{2+\alpha,1+\alpha/2;Q_{R}}\leq C\left(\|f\|_{\alpha,\alpha/2;Q_{2R}} + \|u\|_{L^{\infty}(Q_{2R})} \right),
\end{equation*}
where $C=C(n,\alpha, R, \Lambda,\lambda)>0$ is a constant. 
\end{prop}

It is also possible to obtain higher order estimates, both in time and space. Indeed one has

\begin{prop}[Higher Order Schauder Estimates]\label{schauder-higher} 
Let $\alpha \in (0,1),R>0$ and $k,m\in \N$. 
Let $u\in C^{2+\alpha,1+\alpha/2}(Q_{3R})$ be a solution of
\begin{equation*}
\partial_{t}u=Lu+f 
\end{equation*}
on $Q_{3R}$, where $L$ is as in \eqref{operator}. Suppose that the coefficients of $L$ and $f$ satisfy the following assumptions:
\begin{align*}
&\|\partial_{t}^{r}D^{\gamma}a^{ij}\|_{\alpha,\alpha/2;Q_{2R}},\; \|\partial_{t}^{r}D^{\gamma} b^{i}\|_{\alpha,\alpha/2;Q_{2R}},\;  \|\partial_{t}^{r}D^{\gamma} c\|_{\alpha,\alpha/2;Q_{2R}}\leq \Lambda, \\
&\lambda^{-1}|\xi |^{2}\leq \sum_{i,j=1}^{n} a^{ij}(x,t)\xi_{i}\xi_{j}\leq \lambda |\xi |^{2}, \quad \forall \xi \in \R^{n}, (x,t)\in Q_{2R}, \\
&\partial_{t}^{r}D^{\gamma}f\in C^{\alpha,\alpha/2}(Q_{2R}),
\end{align*}
for some $\Lambda, \lambda >0$ and every $r\leq m$ and $\gamma \in \N^{n}$ with $|\gamma|\leq k$.
Then $\partial_{t}^{r}D^{\gamma}u\in C^{2+\alpha,1+\alpha/2}(Q_{R})$ and 
\begin{equation*}
\|\partial_{t}^{r}D^{\gamma}u\|_{2+\alpha,1+\alpha/2;Q_{R}}\leq 
C\left(\sum_{s\leq r}\sum_{|\beta|\leq |\gamma|}\|\partial_{t}^{s}D^{\beta}f\|_{\alpha,\alpha/2;Q_{2R}} + \|u\|_{L^{\infty}(Q_{2R})} \right)
\end{equation*}
for every $r\leq m$ and $\gamma\in \N^{n}$ with $|\gamma|\leq k$, where $C=C(n,\alpha,R,k,m,\Lambda,\lambda)>0$ is a positive constant.
\end{prop}

\subsection{Estimates for the Conjugate Heat Equation along a Ricci Flow}

We now apply the results from the previous subsection to the following general setting. Consider a soluton $(M,g(\tau))_{\tau \in [0,T]}$ to the backward Ricci flow $\partial_{\tau}g=2\Ric$, and let $u\colon M\times (0,T)\to \R$ be a smooth solution of the backward conjugate heat equation:
\begin{equation}\label{conjheateq-backward}
\partial_{\tau}=\Delta u - Su,
\end{equation}
where $S$ denotes the scalar curvature. Given $m\in \N$ and a subset $\Omega\subseteq M\times [0,T]$, recall that the $C^{m}$-norm on $\Omega$ of a general time-dependent tensor field $A$ on $M$ is defined as\footnote{In this context, when $A$ is a function $\nabla A$ denotes its differential, not its gradient vector field.}
\begin{equation}\label{cm-norm-def}
\|A\|_{C^{m}(\Omega)}=\sum_{r+s \leq m} \|\partial_{t}^{r}\nabla^{s}A\|_{L^{\infty}(\Omega)}+ \|\nabla^{s}\partial_{t}^{r}A\|_{L^{\infty}(\Omega)}.
\end{equation}

We are interested in bounding $\|u\|_{C^{m}(\Omega)}$ in terms of its supremum norm, and we will use the interior Schauder estimates from the previous subsection. 

Suppose $m=1$. Consider a compact time-interval $I=[a,b]\subseteq (0,T)$ and take $\varepsilon>0$ such that $I_{\varepsilon}=[a-16\varepsilon^{2},b+\varepsilon^{2}/2]\subseteq (0,T)$. Given a point $(x_{0},\tau_{0})\in M\times I$ there exists a coordinate chart $(U,\varphi)$, centered at $x_{0}$, such that  $\varphi(U)=B(0,4\varepsilon)\subseteq \R^{n}$. On $B(0,4\varepsilon)\times (0,T)$ equation \eqref{conjheateq-backward} reads
\begin{equation*}
\partial_{\tau}u=g^{ij}\partial_{i}\partial_{j}u - g^{ij}\Gamma_{ij}^{k}\partial_{k}u - Su.
\end{equation*}
Introduce the function $v\colon Q_{4\varepsilon}\to \R$ defined as
\begin{equation*}
v(x,t)=u(x,\tau_{0}+t+\varepsilon^{2}/2).
\end{equation*}
Then $v$ is well defined and satisfies the equation
\begin{equation*}
\partial_{t}v=\widehat{g}^{\, ij}\partial_{i}\partial_{j}v-\widehat{g}^{\,ij}\widehat{\Gamma}_{ij}^{k}\partial_{k}u-\widehat{S}v,
\end{equation*}
where $\widehat{g}^{\, ij}, \widehat{\Gamma}^{k}_{ij}$ and $\widehat{S}$ are defined translating the time variable. Since $v\in C^{\infty}(Q_{4\varepsilon})$ and we have $\{(x,\tau_{0}+t+\varepsilon^{2}/2)\colon (x,t)\in Q_{3\varepsilon}\} \Subset  B(0,4\varepsilon)\times I_{\varepsilon}$, we conclude that $v\in C^{2+\alpha,1+\alpha/2}(Q_{3\varepsilon})$ and moreover there exist $\Lambda, \lambda>0$ such that
\begin{align*}
&\|\widehat{g}^{\, ij}\|_{\alpha,\alpha/2;Q_{3\varepsilon}},\;  \|\widehat{b}^{i}\|_{\alpha,\alpha/2;Q_{3\varepsilon}},\;  \|\widehat{S}\|_{\alpha,\alpha/2;Q_{3\varepsilon}}\leq \Lambda, \\
&\lambda^{-1}|\xi|^{2}\leq \sum_{i,j=1}^{n} \widehat{g}^{ij}(x,t)\xi_{i}\xi_{j}\leq \lambda |\xi|^{2},  \quad \forall \xi \in \R^{n}, (x,t)\in Q_{3R},
\end{align*}
where $\widehat{b}^{i}=\widehat{g}^{\, k\ell}\widehat{\Gamma}_{k\ell}^{i}$.
Therefore, applying Proposition \ref{schauder} to $v$ (with $f=0$), we get
\begin{equation*}
\|v\|_{2+\alpha,1+\alpha/2;Q_{\varepsilon}}\leq C\|v\|_{L^{\infty}(Q_{2\varepsilon})}.
\end{equation*}
Turning back to the original time $\tau$, the above inequality becomes
\begin{equation*}
\|u\|_{2+\alpha,1+\alpha/2;Q}\leq C\|u\|_{L^{\infty}(Q^{\ast})},
\end{equation*}
where $Q=B(0,\varepsilon)\times (\tau_{0}-\varepsilon^{2}/2, \tau_{0}+\varepsilon^{2}/2)$ and $Q^{\ast}=B(0,2\varepsilon)\times (\tau_{0}-7\varepsilon^{2}/2,\tau_{0}+\varepsilon^{2}/2)$. In particular, this implies the (weaker) estimate 
\begin{equation}\label{partial-deriv-est-chart}
\|\partial_{t}u\|_{L^{\infty}(Q)}+\|Du\|_{L^{\infty}(Q)} \leq C\|u\|_{L^{\infty}(Q^{\ast})}.
\end{equation}
Now let $P=\left(\varphi^{-1}\times \text{Id}\right)(Q)$ and 
$P^{\ast}=\left(\varphi^{-1}\times \text{Id}\right)(Q^{\ast})$.
An easy computation allows us to bound the covariant derivative of $u$ on $P$. Indeed, one has
\begin{equation*}
|\nabla u|^{2}=g^{ij}\partial_{i}u\partial_{j}u \leq \lambda |Du|^{2}\leq n\lambda \|Du\|^{2}_{L^{\infty}(P)}
\end{equation*}
which yields
\begin{equation*}
\|\nabla u\|_{L^{\infty}(P)}\leq \sqrt{n\lambda}\|Du\|_{L^{\infty}(P)}.
\end{equation*}
Therefore, from \eqref{partial-deriv-est-chart}, we obtain
\begin{equation}
\|u\|_{C^{1}(P)}=\|u\|_{L^{\infty}(P)}+ \|\partial_{t}u\|_{L^{\infty}(P)}+\|\nabla u\|_{L^{\infty}(P)} \leq C\|u\|_{L^{\infty}(P^{\ast})},
\end{equation}
for some positive constant $C>0$.

Arguing along the same lines, using that $u$ is a smooth solution of \eqref{conjheateq-backward} and that Proposition \ref{schauder-higher} gives bounds on all the mixed partial derivatives of $u$, it is possible to bound all the $C^{m}$-norms of $u$, obtaining the following result.

\begin{prop}[$C^m$-Estimates]\label{schauder-higher-flow} 
Let $I=[a,b]\subseteq (0,T)$ and let $\varepsilon>0$ be as above. 
For every $(x_{0},\tau_{0})\in M\times I$ there exist $P=P(x_{0},\tau_{0},\varepsilon)$ and 
$P^{\ast}=P^{\ast}(x_{0},\tau_{0},\varepsilon)$, open neighbourhoods of $(x_{0},\tau_{0})$ in $M\times (0,T)$ with
$P\subseteq P^{\ast}$, such that for every $m\in \N$
\begin{equation}\label{schauder-higher-flow-ineq}
\|u\|_{C^{m}(P)} \leq C_{m} \|u\|_{L^{\infty}(P^{\ast})},
\end{equation}
where $C_{m}>0$ is a constant independent of $u$. 
\end{prop}

\begin{rem}\begin{enumerate}
\item In Proposition \ref{schauder-higher-flow}, $P$ and $P^{\ast}$ are indeed product chart domains, obtained by restricting $(U\times (0,T),\varphi\times \id)$.
\item In order to get the higher order estimates, one also needs to bound the mixed partial derivatives of the metric $g$, since it is evolving in time; for the same reason the mixed partial/covariant derivatives appearing in \eqref{cm-norm-def} are in general not the same.
\item The constants $C_{m}$ depend on all the various choices made, in particular they depend on $x_{0},\tau_{0}, \varepsilon, (U,\varphi), g, k,m$, and hence on the geometry of $(M,g)$ in a neighbourhood of $(x_{0},\tau_{0})$, but not on $u$.      
\end{enumerate}
\end{rem}

\section{Complements on Metric Measure Spaces and Metric Flows}\label{mms-appendix}

\subsection{Convergence of Metric Measure Spaces}

In \cite[Section~2.3]{Ba20_compactness} Bamler introduced the notion of Gromov-$W_{1}$-Wasserstein-distance, $GW_{1}$-distance for short, between metric measure spaces and the related convergence. It will be useful to have some results linking it to other types of convergences, in particular pointed ones, since we are mainly interested in working with non-compact spaces. In \cite[Definition~3.9]{GMS}, the authors introduced the following concept.

\begin{defn}\label{gromov-meas-conv} 
A sequence $(X_{i},d_{i},\mu_{i},x_{i})_{i\in \N}$ of pointed metric measure spaces converges to the pointed metric measure space $(X_{\infty},d_{\infty},\mu_{\infty},x_{\infty})$ in the \emph{pointed Gromov-measured} sense, if there exist isometric embeddings 
$\varphi_{i}\colon (X_{i},d_{i})\to (Z,d_{Z})$, for $i \in \N\cup \{\infty\}$, into a common complete and separable metric space
$(Z,d_{Z})$, such that
\begin{enumerate}
\item $(\varphi_{i})_{\ast}\mu_{i}\to (\varphi_{\infty})_{\ast}\mu_{\infty}$ narrowly;
\item $\varphi_{i}(x_{i})\to \varphi_{\infty}(x_{\infty})\in \supp ((\varphi_{\infty})_{\ast}\mu_{\infty})$
in $(Z,d_{Z})$.
\end{enumerate}
\end{defn}

It is easy to prove that the convergence just introduced is weaker than the $GW_{1}$ one. 

\begin{prop}[$GW_1$-Convergence Implies Gromov-Measured Convergence]\label{GW-pGm} 
Let $(X_{i},d_{i},\mu_{i})_{i\in \N}$ be a sequence of complete metric measure spaces converging to a complete metric measure space $(X_{\infty},d_{\infty},\mu_{\infty})$ in the $GW_{1}$-sense, and suppose that all the measures have full support. Then there exist points $x_{i}\in X_{i}$ for $i\in \N\cup\{\infty\}$, such that $(X_{i},d_{i},\mu_{i},x_{i})_{i\in \N}$ converges to $(X_{\infty},d_{\infty},\mu_{\infty},x_{\infty})$ in the pointed Gromov-measured sense. 
\end{prop}

\begin{proof}
First of all, there exist sequences of isometric embeddings $\varphi_{i}\colon (X_{i},d_{i})\to (Z,d_{Z})$ and couplings $q_{i}\in \mathcal{P}(Z\times Z)$ between $(\varphi_{i})_{\ast}\mu_{i}$ and $(\varphi_{\infty})_{\ast}\mu_{\infty}$, for $i\in \N\cup \{\infty\}$, realising the $GW_{1}$-convergence, i.e.~such that 
\begin{equation*}
\lim_{i\to \infty}\int_{Z\times Z}d_{Z}(x,y)dq_{i}(x,y)=0.
\end{equation*}
Moreover, we can assume $(Z,d_{Z})$ is complete and separable; see for example \cite[Lemma~2.12]{Ba20_compactness} and the proof of \cite[Theorem~2.14]{Ba20_compactness}.

Now choose points $y_{i}\in X_{i}$ and radii $r_{i}>0$ such that $\mu_{i}(\overline{B}(y_{i},r_{i}))\geq 9/10$ for every $i\in \N\cup \{\infty\}$, and denote $Y_{i}=\varphi_{i}(\overline{B}(y_{i},r_{i}))\subseteq Z$.

We claim there exist points $z_{i}\in Y_{i}$ such that $z_{i}\to z_{\infty}$ in $Z$. Indeed, suppose by contradiction this is not the case. Then there exists $\varepsilon>0$ such that, up to a subsequence, $d_{Z}(x,y)\geq \varepsilon$ for every $(x,y)\in Y_{i}\times Y_{\infty}$ and $i\in \N$. It follows
\begin{equation*}
\int_{Y_{i}\times Y_{\infty}}d_{Z}(x,y)dq_{k}(x,y)\geq \varepsilon q_{i}(Y_{i}\times Y_{\infty})
\end{equation*}
and therefore
\begin{equation}\label{qk-conv-zero}
q_{i}(Y_{i}\times Y_{\infty})\leq \varepsilon^{-1} \int_{Z\times Z}d_{Z}(x,y)dq_{k}(x,y)\longrightarrow 0.
\end{equation}
Now
\begin{equation*}
q_{i}((Z\setminus Y_{i})\times Y_{\infty})\leq q_{i}((Z\setminus Y_{i})\times Z)=(\varphi_{i})_{\ast}\mu_{i}(Z\setminus Y_{i})
=\mu_{i}(X_{i}\setminus \overline{B}(y_{i},r_{i}))\leq 1/10
\end{equation*}
and similarly 
\begin{equation*}
q_{i}(Y_{i}\times (Z\setminus Y_{\infty}))\leq q_{i}(Z\times (Z\setminus Y_{\infty}))=
(\varphi_{\infty})_{\ast}\mu_{\infty}(Z\setminus Y_{\infty})=\mu_{\infty}(X_{\infty}\setminus \overline{B}(y_{\infty},r_{\infty}))\leq 1/10.
\end{equation*}
Therefore,
\begin{equation*}
1=q_{i}(Z\times Z)= q_{i}(Y_{i}\times Y_{\infty})+q_{i}((Z\setminus Y_{i})\times Y_{\infty})+q_{i}(Z\times (Z\setminus Y_{\infty}))
\leq 1/5+q_{i}(Y_{i}\times Y_{\infty})
\end{equation*}
which implies $q_{i}(Y_{i}\times Y_{\infty})\geq 4/5$ for $i$ large enough, in contradiction to \eqref{qk-conv-zero}.

At this point, defining $x_{i}=\varphi_{i}^{-1}(z_{i})\in X_{i}$ we get $\varphi_{i}(x_{i})\to \varphi_{\infty}(x_{\infty})$ and we conclude.
\end{proof}

\subsection{Continuity of Metric Solitons}

A simple but useful consequence of the definition of metric soliton, \cite[Definition~3.42]{Ba20_compactness}, is the following continuity result.

\begin{prop}[Metric Solitons are Continuous]\label{ms-cont} 
Let $(\X,(\mu_{t})_{t<0})$ be an $H$-concentrated metric soliton. Then $\X$ is continuous. 
\begin{proof}
Using the notations of the cited definition, we have
\begin{equation*}
\int_{\X_{t}}\int_{\X_{t}}d_{t}(\x,\y)d\mu_{t}(\x)d\mu_{t}(\y) = \sqrt{|t|}\int_{X}\int_{X}d(x,y)d\mu(x)d\mu(y)
\end{equation*}
for every $t<0$. 
Since $\X$ is $H$-concentrated, it follows that $\var_{t}(\mu_{t})\leq H|t|$ for every $t\leq 0$ (see \cite[Proposition~3.43]{Ba20_compactness}), and therefore the above function is finite and continuous on $(-\infty,0)$. We hence conclude using \cite[Theorem~4.9]{Ba20_compactness}.
\end{proof}
\end{prop}

\subsection{Parabolic Rescaling and Metric Flows}\label{parab-rescaling-mf-app}

Besides the classical context, the parabolic rescaling operation can be performed also on general metric flows and metric flow pairs. Here we want to enlighten the relationship between these two constructions. We start giving the necessary notions.

Let $\X$ be a metric flow over $I\subseteq \R$. Given $t_{0}\in \R$ and $\lambda>0$,
define $I^{-t_{0},\lambda}=\{\lambda(t-t_{0})\colon t\in I\}$. Then 
\begin{equation*}
\left(\X^{-t_{0},\lambda}, {\mathfrak{t}}^{-t_{0},\lambda}, (d^{-t_{0},\lambda}_{t})_{t\in I^{-t_{0},\lambda}}, 
(\nu^{-t_{0},\lambda}_{\x;s})_{\x\in \X^{-t_{0},\lambda}, s\in I^{-t_{0},\lambda}, s\leq \mathfrak{t}^{-t_{0},\lambda}(\x)}\right),
\end{equation*}
where 
\begin{itemize}
\item $\X^{-t_{0},\lambda}=\X$ and ${\mathfrak{t}}^{-t_{0},\lambda}=\lambda(\mathfrak{t}-t_{0})$,
\item $d^{-t_{0},\lambda}_{t}=\sqrt{\lambda} d_{t_{0}+t/\lambda}$, 
\item $\nu^{-t_{0},\lambda}_{\x; s}=\nu_{\x; t_{0}+s/\lambda}$,
\end{itemize}
is a metric flow over $I^{-t_{0},\lambda}$, that will be denoted by $\X^{-t_{0},\lambda}$ and called a \emph{parabolic rescaling} of $\X$. In particular,
\begin{equation*}
\X^{-t_{0},\lambda}_{t}=\X_{t_{0}+t/\lambda}
\end{equation*}
for every $t\in I^{-t_{0},\lambda}$.

Similarly, if $(\X, (\mu_{t})_{t\in I'})$ is a metric flow pair over $I$ then $(\X^{-t_{0},\lambda}, (\mu^{-t_{0},\lambda}_{t})_{t\in I'^{,- t_{0},\lambda}})$ is a metric flow pair over $I^{-t_{0},\lambda}$, where $\X^{-t_{0},\lambda}$ is as above and
\begin{equation*}
\mu^{-t_{0},\lambda}_{t}=\mu_{t_{0}+t/\lambda}
\end{equation*}
for all $t\in I'^{,-t_{0},\lambda}=\{\lambda(t-t_{0})\colon t\in I' \}$.

Now let $(M,g(t))_{t<T}$ be a Ricci flow (possibly induced by a Ricci shrinker $(M,g,f)$), where $T<\infty$ is the singular time, and consider its corresponding metric flow $\X$ over an open interval $I\subseteq (-\infty,T)$. Given $t_{0}\leq T$ and $\lambda>0$ we can parabolically rescale $\X$, obtaining the metric flow $\X^{-t_{0},\lambda}$ over the interval $I^{\ast}=I^{-t_{0},\lambda}=\{t\colon t_{0}+t/\lambda \in I \}$, given as above. Note that, by construction, a point $\x\in \X^{-t_{0},\lambda}_{t}=\X_{t_{0}+t/\lambda}$ is of the form $\x=(x,t_{0}+t/\lambda), x\in M$. 

On the other hand, starting from $(M,g(t))_{t<T}$ we can instead first form the rescaled Ricci flow $(M,g_{-t_{0},\lambda}(t))_{t<T^{\ast}}$, where $T^{\ast}=\lambda(T-t_{0})$ is the new singular time, and consider its related quantities, in particular the family of probability measures 
$(\nu^{g_{-t_{0},\lambda}}_{(x,t);s})_{x\in M, s\leq t\leq T^{\ast}}$.
Since $I^{\ast}\subseteq (-\infty,T^{\ast})$ is an open interval, we are allowed to construct the metric flow $\X^{\ast}$
over $I^{\ast}$ corresponding to $(M,g_{-t_{0},\lambda}(t))_{t\in I^{\ast}}$, which can be expressed in terms of $\X$ as
\begin{itemize}
\item $\X^{\ast}_{t}= \X_{t_{0}+t/\lambda}$,
\item $d^{\ast}_{t}=\sqrt{\lambda}d_{t_{0}+t/\lambda}$,
\item $\nu^{\ast}_{\x;s}=\nu_{\x; t_{0}+s/\lambda}$,
\end{itemize}
for every $s\leq t\in I^{\ast}$ and $\x\in \X^{\ast}_{t}$, where we are using the identification 
\begin{equation*}
\X^{\ast}_{t}\longrightarrow \X_{t_{0}+t/\lambda},\quad (x,t)\mapsto (x,t_{0}+t/\lambda).
\end{equation*}
Therefore $\X^{\ast}$ coincides with the parabolic rescaling of $\X$, that is $\X^{\ast}=\X^{-t_{0},\lambda}$.

Now fix a point $q\in M$ and let $(\X, (\mu_{t})_{t\in I})$ be the metric flow pair over $I$ corresponding to the pointed Ricci flow $(M,g(t),(q,t_{\max}))_{t\in I}$, where $t_{\max}=\sup I$. Then we can consider its parabolic rescaling, which is the metric flow pair over $I^{\ast}$ given by $(\X^{-t_{0},\lambda},(\mu^{-t_{0},\lambda}_{t})_{t\in I^{\ast}})$, where
\begin{equation*}
\mu^{-t_{0},\lambda}_{t}=\mu_{t_{0}+t/\lambda}
\end{equation*}
for every $t\in I^{\ast}$. In particular $\mu^{-t_{0},\lambda}_{t}=\nu_{(q,t_{\max});t_{0}+t/\lambda}$. 

On the other hand, for the metric flow pair $(\X^{\ast}, (\mu^{\ast}_{t})_{t\in I^{\ast}})$ corresponding to the pointed Ricci flow $(M,g_{-t_{0},\lambda}(t), (q, t_{\max}^{\ast}))_{t\in I^{\ast}}$, where $t_{\max}^{\ast}=\sup I^{\ast}$, we have, under the above identification
\begin{equation*}
\mu^{\ast}_{t}=\nu^{g_{-t_{0},\lambda}}_{(q,t_{\max}^{\ast});t}=\nu_{(q,t_{\max});t_{0}+t/\lambda}=\mu_{t_{0}+t/\lambda}
\end{equation*}
since $t_{0}+t_{\max}^{\ast}/\lambda=t_{\max}$. Therefore $\mu^{\ast}_{t}=\mu^{-t_{0},\lambda}_{t}$, that is 
$(\X^{\ast},(\mu^{\ast}_{t})_{t\in I^{\ast}})=(\X^{-t_{0},\lambda}, (\mu^{-t_{0},\lambda})_{t\in I^{\ast}})$.

%%%---------------------------------------------------------------------------

\makeatletter
\def\@listi{%
  \itemsep=0pt
  \parsep=1pt
  \topsep=1pt}
\makeatother
{\fontsize{10}{12}\selectfont

}
\vspace{-2mm}
\printaddress

\end{document}